\documentclass[11pt]{article}

\usepackage{enumitem}
\usepackage{amsmath,amsthm,amsfonts,amssymb,amscd}
\usepackage{hyperref}
\usepackage{geometry}
\usepackage{leftindex}
\usepackage[normalem]{ulem}
\usepackage{xcolor}
\usepackage{color}

\usepackage{multirow,booktabs}

\usepackage{fancyhdr}
\usepackage{mathrsfs}
\usepackage{wrapfig}
\usepackage{setspace}
\usepackage{calc}
\usepackage{multicol}
\usepackage{cancel}

\usepackage{xcolor}
\colorlet{shadecolor}{orange!15}
\theoremstyle{definition}

\theoremstyle{plain}
\newtheorem{thm}{Theorem}[section]
\newtheorem{lem}[thm]{Lemma}
\newtheorem{prop}[thm]{Proposition}
\newtheorem{cor}[thm]{Corollary}

\theoremstyle{remark}
\newtheorem{defn}[thm]{Definition}
\newtheorem{rem}[thm]{Remark}
\newtheorem{exam}[thm]{Example}

\newcommand{\cl}{\mathcal}

\newcommand{\wt}{\widetilde}

\newcommand{\sbf}{\boldsymbol}
\newcommand{\mbf}{\mathbf}
\newcommand{\bb}{\mathbb}
\newcommand{\mrm}{\mathrm}

\newcommand{\EqD}{\overset{d}{=}}

\def\pp#1{ \left(#1\right) }
\def\pb#1{ \left[#1\right] }
\def\pc#1{ \left\{#1\right\} }

\usepackage{natbib}

\usepackage{hyperref} % Ensure proper handling of hyperlinks

\begin{document} 
% Ensure math content in hyperlinks uses \texorpdfstring if needed
\setcounter{section}{0}
\setcounter{figure}{0}
\setcounter{table}{0}
\setcounter{equation}{0}

\begin{center}
{\LARGE \bf  Multiple Extremal Integrals }\\
\end{center}

\begin{center}
Shuyang Bai$^{a}$, Jiemiao Chen$^{b}$
\end{center}

\footnotetext{Department of Statistics, University of Georgia, 310 Herty Drive, Athens, GA 30602, USA.
\par
$^{a}$\href{mailto:bsy9142@uga.edu}{bsy9142@uga.edu}\quad
$^{b}$\href{mailto:jc17876@uga.edu}{jc17876@uga.edu}
\par
The authors are ordered alphabetically.}

\begin{abstract}
We introduce the notion of multiple extremal integrals as an extension of single extremal integrals, which have played important roles in extreme value theory. The multiple extremal integrals are formulated in terms of a product-form random sup measure derived from the $\alpha$-Fr\'{e}chet random sup measure. We establish a LePage-type representation similar to that used for multiple sum-stable integrals, which have been extensively studied in the literature. This approach allows us to investigate the integrability, tail behavior, and independence properties of multiple extremal integrals. Additionally, we discuss an extension of a recently proposed stationary model that exhibits an unusual extremal clustering phenomenon, now constructed using multiple extremal integrals.

\end{abstract}

\textit{Keywords}: Multiple  stochastic integrals, Random sup measures, LePage representation, Extremes, $\alpha$-Fr\'{e}chet

Mathematics Subject Classification (2020) 60G70 (1st); 60F05 (2nd)

\section{Introduction}
 The main purpose of this paper is to introduce a notion of \emph{multiple extremal integral}, as an extension of the \emph{single extremal stochastic integral} that was   introduced by \cite[]{dehaan} and \cite[]{stoev2005extremal}:
\begin{equation}\label{undecoupled int}
  \leftindex^e \int_{E} f(u) M_\alpha (du),
\end{equation}
where $f$ is a univariate non-negative deterministic function on a measure space $(E, \mathcal{E}, \mu)$, and $M_\alpha$ is an independently scattered $\alpha$-Fr\'echet random sup measure ($\alpha >0$) with control measure $\mu$ (see Definition \ref{defn s} below for more details).  The single extremal integral plays an important role in extreme value theory. For instance, it provides convenient representations for max-stable processes, which   facilitate the study of their dependence properties (e.g., \cite[]{kabluchko2009spectral},  \cite[]{wang2010structure}, \cite[]{wang2013ergodic}, \cite[]{dombry2017ergodic}).
In this work, we shall introduce a multiple extremal integral,  formally expressed as  
\begin{equation}\label{eq:mult int intro}
I_k^e(f)= \leftindex^e \int_{E^k} f\left(u_1, \ldots, u_k\right) M_\alpha\left(d u_1\right) \ldots M_\alpha\left(d u_k\right),    
\end{equation}
where $f$ is a $k$-variate non-negative deterministic function defined on the product measure space $(E^k, \mathcal{E}^k, \mu^k)$ with $k \in \bb{Z}_+:=\{1,2,\ldots\}$, which vanishes on the diagonal set $\left\{ u_i=u_j,\ i \neq j\right\}$.

Multiple stochastic integrals have a rich history dating back to the seminal work of \cite[]{ito1951multiple} for the case of Gaussian random measure, which modified an early idea of \cite[]{wiener1938homogeneous} by excluding the diagonal set from the integral. Since then multiple stochastic integrals with respect to Gaussian and non-Gaussian random measures have been   extensively studied in the literature, as summarized by the recent monographs \cite[]{peccati2011wiener, major2013multiple, kallenberg2017random}. To the best of our knowledge, the literature has focused on multiple stochastic integrals with respect to additive random measures.  Multiple stochastic integrals with respect to random sup measures as in \eqref{eq:mult int intro} do not seem to have been considered in the literature. Following the convention introduced by It\^o, we exclude the diagonal 
set in the integration, which considerably simplifies the theory. The study 
of multiple extremal integrals involving diagonal sets is left for future work.

As observed by \cite[]{stoev2005extremal}, an extremal integral \eqref{undecoupled int} with respect to an $\alpha$-Fr\'echet random sup measure has a close connection with a stable integral, an additive stochastic integral with respect to an $\alpha$-stable random measure. So, naturally, the development of theories for multiple extremal integrals also bears a close connection to that of multiple stable integrals (e.g., \cite[]{rosifiski1986ito}, \cite[]{krakowiak1986random}, \cite[]{samorodnitsky1989asymptotic}), which will be explored throughout the paper.

The content of the paper is summarized as follows: Section \ref{sec:contr} presents the construction of multiple extremal integrals, followed by an investigation of integrability in Section \ref{sec:int}. We then discuss further properties of multiple extremal integrals, including interchanging limit and integral sign (Section \ref{sec:limit}), distributional tail behavior of integrals of fixed order (Section \ref{sec:tail}), an independence criterion between multiple integrals (Section \ref{s7}). In Section \ref{sec:mult reg}, we  present  an extension of a recently introduced stationary model in \cite[]{bai2024phase} that exhibits an unusual extremal clustering phenomenon, now constructed based on  multiple extremal integrals.
%We illustrate their relations, analogous to the concept of a version or modification of a stochastic processes, which may be of independent interest.\ 

Throughout the paper, 
% Suppose $f(t), g(t)$ are two real-valued functions defined on $(0,\infty)$.
%Then $f(t) \sim g(t)$ means $\lim _{t \rightarrow \infty} f(t) / g(t)=1$. 
the symbol $\vee$ stands for supremum, the symbol $\wedge$ stands for infimum, and   $\lambda$ is Lebesgue measure  on $\mathbb{R}^n$, $n \in \mathbb{Z}_+$ (or its Borel subsets).  
The underlying probability space will be denoted by $(\Omega,\mathcal{F},\mathbb{P})$ and $\bb{E}$ is the expectation sign. 
% is the underlying probability space, $\mathcal{L}^0(\Omega)$ is the space of all $\mathcal{F}$-measurable functions. 
On a measure space with measure $\mu$, we use  $ L_+^\alpha(\mu)$ to denote the space of all nonnegative measurable functions whose $\alpha$-th moment is finite, $\alpha\in (0,\infty)$. Suppose $f(x)$ and $g(x)$ are positive functions near a point $a$ (possibly $0$ or $\infty$).
We write $f(x) \sim g(x)$ if $f(x) / g(x) \rightarrow 1$ as $x \rightarrow a$.   

\section{Multiple extremal integrals: constructions and basic properties}\label{sec:contr}
\subsection{\texorpdfstring{\ensuremath{\alpha}}{alpha}-Fr\'{e}chet random sup measures} \label{s2}

The concept of (independently scattered) \emph{$\alpha$-Fr\'{e}chet random sup measure} plays an important role in extreme value theory. Its definitions in the literature, on the other hand, show subtle differences.   When viewed as a set-indexed stochastic process, it is often constructed in a pathwise manner through a Poisson point process (e.g., \cite[]{dehaan}; see also Section \ref{alt ppp} below), or identified as a measurable random element taking value in the space of sup measures equipped with the sup vague topology (e.g., \cite[]{vervaat1988random}).   
In this work,  we shall work with the ``weak'' definition described by \cite[]{stoev2005extremal}  for the construction of multiple extremal integrals. 
Alternative   constructions will be mentioned in {Remark \ref{rem:alt constr} and Section \ref{alt ppp}}  below.
Throughout the paper, we assume that $(E, \mathcal{E}, \mu)$ is a $\sigma$-finite measure space, and the measure spaces mentioned are assumed to have nonzero measures.

\begin{defn}\label{defn s}
(\cite[Definition 2.1]{stoev2005extremal})  An (independently scattered) $\alpha$-Fr\'{e}chet random sup measure  with control measure $\mu$   
  is a set-indexed stochastic process $M_\alpha=\left(M_\alpha(A)\right)_{A\in \mathcal{E}}$, where each $M(A)$, $A\in \mathcal{E}$, is a random variable taking value in $[0,\infty]$, and  the following conditions are satisfied:
  
\begin{enumerate}
\item (independently scattered) For any collection of disjoint sets $A_j \in \mathcal{E}, 1 \leq j \leq n$, $n \in \bb{Z}_+$, the random variables $M_\alpha\left(A_j\right), 1 \leq j \leq n$, are independent.
\item ($\alpha$-Fr\'{e}chet marginal) For any $A \in \mathcal{E}$, we have
$$
\mathbb{P}\left\{M_\alpha(A) \leq x\right\}=\exp \left\{-\mu(A) x^{-\alpha}\right\} \text{ for } x\in (0,\infty),
$$
that is, $M_\alpha(A)$ is $\alpha$-Fr\'{e}chet with scale coefficient $(\mu(A))^{1 / \alpha}$. Here, when $\mu(A)=0$ or $\mu(A)=\infty$, we understand $M_\alpha(A)$ as a random variable taking value $0$ or $\infty$ almost surely (a.s.), respectively.
\item ($\sigma$-maxitive) For any collection of sets $A_j \in \mathcal{E}, j \in \bb{Z}_+$, we have that
    \begin{align}\label{eq:sigma max M_alpha}
        M_\alpha \left(\bigcup_{j\geq 1} A_i\right) = \bigvee_{j\geq 1} M_\alpha \left( A_j\right) \text{  a.s..  } 
    \end{align}
\end{enumerate}

% In particular, if $\mu(A) = \infty$, then $M_\alpha(A) = \infty$.
\end{defn}
\begin{rem}\label{Rem:defn s}
   The definition stated here is slightly different from that of \cite[Definition 2.1]{stoev2005extremal}. First, in the $\sigma$-maxitive relation \eqref{eq:sigma max M_alpha}, we do not require disjointness of $A_j$'s, although the relation is equivalent to the one in \cite[]{stoev2005extremal} that requires disjointness.  Second, the domain of $M_\alpha$ is the full $\sigma$-field $\mathcal{E}$ instead of only those sets in $\mathcal{E}$ with finite $\mu$ measures.  The existence of $M_\alpha$ described in Definition \ref{defn s}  follows from a direct modification of the proof of \cite[Proposition 2.1]{stoev2005extremal}, by making use of a general version of Kolmogorov’s existence theorem (e.g., \cite[Theorem 8.23]{kallenberg2021foundations}) that allows for marginally $[0,\infty]$-valued processes.    The properties (i) $\sim$ (iii) above uniquely characterize the finite-dimensional distributions of the  process $M_\alpha$  in the sense that a $[0,\infty]$-valued process indexed by $\cl{E}$ has the same finite-dimensional distributions as $M_\alpha$  if and only if it satisfies these properties.
\end{rem}

Next, we describe a distributional representation of the random sup measure $M_\alpha$ known as the LePage representation, which will play a key role in the construction of multiple extremal integrals and the investigation of their distributional properties.  The origin of such a representation dates back to the spectral representation of max-stable processes \cite[]{dehaan}, and is analogous to the LePage series representation for stable, and more generally, infinitely divisible processes  \cite[]{lepage1989appendix,lepage1981multidimensional,lepage1981convergence}. The representation is also mentioned in \cite[Section 3]{stoev2005extremal}, although the version we shall describe allows an infinite control measure.

Since $\mu$ is a $\sigma$-finite measure,   there exists a probability measure $m$ equivalent to $\mu$, and hence there exists a version of the Radon-Nykodim derivative $\psi=\frac{d\mu}{dm}$ such that $\psi\in (0,\infty)$ $m$-a.e.. In particular, we can work with $\psi=1$ if $\mu$ is itself a probability measure. 
 
 % Based on this, we introduce the following definition for the case. 

\begin{defn}\label{defn l}
(LePage representation) Suppose $m$ is a probability measure on $(E, \mathcal{E})$ equivalent to $\mu$ with  $\psi=\frac{d\mu}{dm}\in (0,\infty)$ $m$-a.e..
Let $\left(\Gamma_i \right)_{i \in \bb{Z}_+}$ be the arrival times of a standard Poisson process on $[0, \infty)$, and $\pp{T_i}_{i \in \bb{Z}_+}$ be  a sequence of i.i.d.\ random variables with distribution $m$, independent of $\pp{\Gamma_i}_{i \in \bb{Z}_+}$. Then we define a set-indexed process $\pp{M_\alpha^L(A)}_{A\in \cl{E}}$ by  
\begin{equation}
\left( M_\alpha^L(A)\right)_{A \in \mathcal{E}} =\left(\bigvee_{i \geq 1} \mathbf{1}_{\{T_i \in A\}} \psi(T_i)^{1 / \alpha} \Gamma_i^{-1 / \alpha}\right)_{A \in \mathcal{E}}.
\end{equation}
\end{defn}
It can be verified (see \eqref{m and ml1} below) that the LePage representation $M^L_\alpha$ also satisfies properties (i) $\sim$ (iii) of Definition \ref{defn s}. Further, due to commutativity of supremums and $\bigvee_{j\ge 1}\mathbf{1}_{\{T_i \in A_j\}}=  \mathbf{1}_{\{T_i \in \bigcup_{j\geq 1} A_j\}} $, $i\in \bb{Z}_+$, the definition leads to the following  pathwise  $\sigma$-maxitive property: for $\bb{P}$-a.e.\ element $\omega\in \Omega$, where $\Omega$ is the underlying probability space, the relation 
\begin{equation}\label{eq l sigma add}
M_\alpha^L\left(\bigcup_{j\geq 1} A_j\right)(\omega)=\bigvee_{j\geq 1} \pp{M_\alpha^L\left(A_j\right)(\omega)}
\end{equation}
 holds for any countable collection of sets $A_j \in \mathcal{E}, j \in \bb{Z}_+$, a property stronger than that in Definition \ref{defn s} (iii) (see Remark \ref{Rem:s vs l} below). 
 \paragraph{} Next, we briefly review the single extremal integral in \cite[]{stoev2005extremal}.   
 For a simple function $f(u)=\sum_{j=1}^N a_j \mathbf{1}_{A_j}(u)\in L_{+}^\alpha(\mu)$, $N \in \mathbb{Z}_+$, where $A_j \in \mathcal{E}$,  $a_j\ge 0$, $j=1, \ldots, N$, the extremal integral of $f$ with respect to $M_\alpha$ is defined as
$$\leftindex^e \int_E f(u) M_\alpha(d u):=\bigvee_{1 \leq j \leq N} a_j M_\alpha\left(A_j\right).$$
It then follows that the integral  has an $\alpha$-Fréchet distributon with scale coefficient $\pp{\sum_{j=1}^N a_j^{\alpha}\mu(A_j)}^{1/\alpha}$.
For a general $f \in L_+^\alpha(\mu)$, the extremal integral is defined by monotone approximation with simple functions (see \cite[(2.19)]{stoev2005extremal}).  The construction  implies that $\leftindex^e \int_{E} f(u) M_\alpha(d u)$ is  $\alpha$-Fréchet with scale coefficient $\left(\int_E f(u)^\alpha \mu(d u)\right)^{1 / \alpha}$. Thus, the condition $\int_E f^\alpha(u) \mu(du) < \infty$  characterizes the integrability of the single extremal integral.   Let $\mathcal{L}$ denote the class of all nonnegative measurable functions on $E$. If $f \in \mathcal{L}$ but $f \notin L_{+}^\alpha(\mu)$, then the extremal integral of $f$ is equal to $+\infty$ a.s..

Some basic properties of this integral are listed below.
\begin{enumerate}
    \item (max-linearity) For all $f, g \in \mathcal{L}$ and $a, b \geq 0$, we have
    $$\leftindex^e \int_{E}(a f(u) \vee b g(u)) M_\alpha(d u) \stackrel{\text { a.s. }}{=}\left(a \leftindex^e \int_{E} f(u) M_\alpha(d u)\right) \vee\left(b \leftindex^e \int_{E} g(u) M_\alpha(d u)\right).$$
    % \item (scale coefficient isometry) For all $f \in L_{+}^\alpha(\mu)$, the random variable $\leftindex^e \int_{E} f(u) M_\alpha(d u)$ is  $\alpha$-Fréchet with scale coefficient $\left(\int_E f(u)^\alpha \mu(d u)\right)^{1 / \alpha}$ (hence, $\int_E f^\alpha(u) \mu(du) < \infty$  characterizes the integrability of the single extremal integral.).
    \item (independence) For $f, g \in \mathcal{L}$, the extremal integrals $\leftindex^e \int_{E} f(u) M_\alpha(d u)$ and $\leftindex^e \int_{E} g(u) M_\alpha(d u)$ are independent iff $f(u) g(u)=0  \ \mu$-a.e.
    \item  (monotonicity) For any $f, g \in \mathcal{L}$, we have    
$$
\leftindex^e \int_{E} f(u) M_\alpha(d u) \leq \leftindex^e \int_{E} g(u) M_\alpha(d u) \text { a.s., if and only if},\ f(u)\leq g(u)  \ \mu \text {-a.e.,}
$$
and the equality holds a.s. if and only if $f(u) = g(u) \ \mu\text {-a.e.}$.
\end{enumerate}
% Moving forward, we may view the integral via a Poisson point construction (LePage representation). 
Recall $m, \psi, T_i$'s, $\Gamma_i$'s are as defined in Definition \ref{defn l}. 
% For any $\alpha >0$, let $\left(Y_i, T_i\right), i \in \mathbb{Z}_+$ be a Poisson point process on $(0, \infty) \times E$ with intensity $\alpha y^{-(1+\alpha)} d y \times m(d v)$. Then,
% $\left(Y_i, T_i\right)=\left(\epsilon_i^{-1 / \alpha}, T_i\right), i \in \mathbb{Z}_+$, where $\left(\epsilon_i, T_i\right), i \in \mathbb{Z}_+$ is a Poisson point process on $(0, \infty) \times E$ with intensity $d x \times m(d v)$. By ordering the $\epsilon_i$ 's, one gets
% \begin{equation}\label{de lepage}
% \left\{\bigvee_{k \geq 1} Y_i f\left(T_i\right)\psi\left(T_i\right)^{1 / \alpha}\right\}_{f \in L_{+}^\alpha(\mu)} \stackrel{d}{=}\left\{\bigvee_{k \geq 1} \Gamma_i^{-1 / \alpha} f\left(T_i\right) \psi\left(T_i\right)^{1 / \alpha}\right\}_{f \in L_{+}^\alpha(\mu)},
% \end{equation}
% The left-hand side of (\ref{de lepage}) is the de Haan-type spectral representation since its form is similar to the representation for max-stable processes \cite[Theorem 3]{dehaan}, whereas the right-hand side of (\ref{de lepage}) is the LePage representation. 
Through a slight extension  of  \cite[Proposition 3.1]{stoev2005extremal} from the probability space on $[0,1]$ to the  probability space $(E,\mathcal{E}, m)$,  noting that $
\left(\int_E  f(x)^\alpha\psi(x) m(d x)\right)^{1 / \alpha}=\left(\int_E f(u)^\alpha \mu(d u)\right)^{1 / \alpha}
$, one has   the following LePage representation of a single extremal integral:
\begin{equation}\label{eq prep 2.4}
  \left( \leftindex^e \int_E f(u) M_\alpha(d u)\right)_{f \in \cl{L}}  \stackrel{d}{=}  \left(\bigvee_{k \geq 1} \Gamma_i^{-1 / \alpha} f\left(T_i\right) \psi\left(T_i\right)^{1 / \alpha}\right)_{f \in \cl{L}},
\end{equation}
where `` $\stackrel{d}{=}$ '' is understood as equality in finite-dimensional distributions (and will bear such a meaning between two processes throughout the paper),  
% Futhermore, by the left-hand side of inequality \eqref{suff single} in Lemma \ref{lem 3.17} below, one can derive 
% \begin{equation}
%    \mathbb{P}\left( \bigvee_{k \geq 1} \Gamma_i^{-1 / \alpha} f\left(T_i\right) \psi\left(T_i\right)^{1 / \alpha} \leq x \right) = \exp\left\{ -   \frac{\int_E f^\alpha(u) \mu(du)}{x^\alpha} \right\}.
% \end{equation}
% Hence, $\int_E f^\alpha(u) \mu(du) < \infty$  characterizes the existence of single extremal integral.
%\begin{prop}\label{Pro:fdd M ML}
In particular,  by setting $f$ as indicators,
% for
%   $ M_\alpha$     as in Definition   \ref{defn s}, and    $m, M_\alpha^L$   as in Definition   \ref{defn l},  where the measure $\mu$ is shared in both definitions,
  we have 
\begin{align}\label{m and ml1}
    \pp{ M_\alpha(A)}_{ A \in \mathcal{E}} \stackrel{d}{=} \pp{ M_\alpha^L(A) }_{ A \in \mathcal{E}}.
\end{align}

\begin{rem}\label{Rem:s vs l}
In Appendix \ref{dd rsp}, we show that a random sup measure $M_\alpha$ satisfying Definition \ref{defn s}  admits a modification which is a LePage representation on the same probability space under suitable regularity condition on the space $(E,\cl{E},\mu)$. This means that for every fixed set, the values of $M_\alpha$ and its LePage representation agree on that set a.s..  As an intermediate step of deriving the aforementioned fact,  we show that for any random sup measure in the sense of  \cite[Definition 11.2]{vervaat1988random}, there exists a LePage representation which is indistinguishable from the former one.

On the other hand, there exist random sup measures that satisfy Definition \ref{defn s}, but do not admit  LePage representations pathwise. 
For instance, 
take $(E,\cl{E},\mu)=([0,1],\cl{B}([0,1]),\lambda)$. Let $(\Gamma_i)_{i\in \bb{Z}_+}$ and $(T_i)_{i\in \bb{Z}_+}$ be as in Definition \ref{defn l} with $\mu=m$ and $\psi\equiv1$. Additionally, let $T_0$ be a random variable independent of everything else with $T_0\EqD T_1$. Now consider a set-indexed process $\widehat{M}_\alpha$ defined by $\widehat{M}_\alpha(A)  = \pp{\bigvee_{i\geq 1}\boldsymbol{1}_{\left\{T_i \in  A\,\right\}} \Gamma_i^{-1 / \alpha} }+ \boldsymbol{1}_{\{ \{T_0\} =  A\}}$, $A\in \cl{E}$. It can be verified that $\widehat{M}_\alpha$ satisfies Definition \ref{defn s}, but it does not satisfy the pathwise $\sigma$-maxitive property \eqref{eq l sigma add}.
\end{rem}

\subsection{Product random sup measures}\label{S 2.2}

We will proceed to construct multiple extremal integrals based on an $\alpha$-Fr\'echet random sup measure $M_\alpha$ that satisfies Definition \ref{defn s}. For technical reasons (mainly for being able to approximate off-diagonal sets by rectangles; see Theorem \ref{claim 5.3} below), from now on we impose a mild assumption: The measurable space $(E,\cl{E})$ is a Borel space, that is, there exists a bijection (Borel isomorphism) $\iota:E \leftrightarrow S$ such that both $\iota$ and $\iota^{-1}$ are measurable, with $S$ being a Borel subset of $[0,1]$. In this description, the space $[0,1]$ can also be replaced by an arbitrary Polish space \cite[Theorem 1.8]{kallenberg2021foundations}.
Let 
\[D^{(k)}=\left\{\left(x_1, \ldots, x_k\right)\in E^{k}\mid x_i = x_j  \text{ for some } 1\leq i, j \leq k, i \neq j\right\}\]  
denote the diagonal set of $E^k$. Under the Borel assumption, the diagonal set $D^{(k)}\in \cl{E}^k$ \cite[Theorem 3]{hoffmann1971existence}. 
  
Introduce the following multi-index sets for $k \in \bb{Z}_+$:
\begin{equation}\label{no order}
 \mathcal{D}_{k}=\left\{\boldsymbol{j} = (j_1,\ldots
 ,j_k)\in \bb{Z}_+^k \mid \text { all } j_1, \ldots, j_k \text { are distinct}\right\},   
\end{equation}
and
\begin{equation}\label{D<}
\mathcal{D}_{k, <}=\left\{\boldsymbol{j}=\left(j_1, \ldots, j_k\right) \in \bb{Z}_+^k\mid  j_1<\ldots<j_k\right\}.
\end{equation}
Throughout, we write $\left(E^k, \mathcal{E}^k, \mu^k\right)$ for the $k$-fold product measure space, use boldface letters to denote vectors, and use the same regular letter with a subindex to denote its component; for example $\boldsymbol{i} =\left(i_1, \ldots, i_k\right) \in \mathbb{Z}_{+}^k$. For a sequence $(a_j)_{j\in \bb{Z}_+}$, and $\boldsymbol{j} \in \mathcal{D}_{k}$,  we write $$[a_{\boldsymbol{j}}]: = a_{j_1}a_{j_2}\ldots a_{j_k}$$ and $[a_{\boldsymbol{j}[m:n]}]: = a_{j_m}a_{j_{m+1}}\ldots a_{j_n}$, $1\le m\le n\le k$.

The initial step towards defining the multiple extremal integral is to construct a product-measure-like random sup measure,  denoted by $M_\alpha^{(k)}$, on the off-diagonal space $(E^{(k)}, \mathcal{E}^{(k)})$, where 
$$
E^{(k)}:= E^{k}\setminus D^{(k)},\qquad  \mathcal{E}^{(k)}: = \mathcal{E}^k \cap (D^{(k)})^c= \pc{ A\cap \pp{D^{(k)}}^c\mid A\in \cl{E}^k}. 
$$
We shall refer to $ \mathcal{E}^{(k)}$ as the off-diagonal $\sigma$-field.

A similar approach  was used in \cite[]{samorodnitsky1991construction} for constructing a multiple stable integral.
We shall slightly abuse the notation by regarding $\mu^k$ also as a measure defined on $(E^{(k)}, \mathcal{E}^{(k)})$ through restriction.
Consider first an off-diagonal rectangle of the form $ A_1 \times A_2 \times \cdots \times A_k$, where $A_i \in \mathcal{E}$, $1 \leq i \leq k$,  and the sets $A_i$, $1 \leq i \leq k$, are  pairwise  disjoint. Denote the collection of such off-diagonal rectangles by $\mathcal{C}_k$.  We first define the  random sup measure $M_\alpha^{(k)}$ on  $\mathcal{C}_k$ by setting
\begin{equation}\label{prod M_a}
\begin{aligned}
 M^{(k)}_\alpha \left(A_1 \times A_2 \times \ldots \times A_k\right) & =    M_\alpha(A_1 ) M_\alpha( A_2) \cdots  M_\alpha(A_k),
\end{aligned}
\end{equation}
which is not to be confused with an $\alpha$-Fréchet random sup measure governed by a product control measure.
Next, let $\mathcal{F}_k$ denote the collection of finite unions of off-diagonal rectangles  in $\mathcal{C}_k$. For $A=\bigcup_{i=1}^m B_i\in \cl{F}_k$, where $m\in \bb{Z}_+$ and $B_i \in \mathcal{C}_k$, $1 \leq i \leq m$, we identify $M_\alpha^{(k)}\left(A\right)$ a.s.\ by the maxitive relation
 \begin{equation}\label{max}
M_\alpha^{(k)}\left(\bigcup_{i=1}^m B_i\right):=\bigvee_{i=1}^m M_\alpha^{(k)}\left(B_i\right).
% \stackrel{d}{=} \bigvee_{\boldsymbol{j} \in \mathcal{D}_k} \mathbf{1}\left\{T_{\boldsymbol{j}} \in \bigcup_{i=1}^m B_i\right\}\left[\psi\left(T_{\boldsymbol{j}}\right)\right]^{1 / \alpha}\left[\Gamma_{\boldsymbol{j}}\right]^{-1 / \alpha},
\end{equation}
% where the equality in distribution holds since  $\mathbf{1}\left\{T_{\boldsymbol{j}} \in \bigcup_{i=1}^m B_i\right\}=\bigvee_{i=1}^m \mathbf{1}\left\{T_{\boldsymbol{j}} \in   B_i\right\}$, $\mbf{j}\in \mathcal{D}_k$.
Additionally,  it can be verified  (see Lemma \ref{d ind}) that the definition of  $M_\alpha^{(k)}\left(A\right)$ is independent of the choice of $B_i$'s.
We now show that $M_\alpha^{(k)}$ extends its domain from  $\mathcal{F}_k$  to the off-diagonal $\sigma$-field 
$
\mathcal{E}^{(k)}. 
$
  We start with the case where the control measure $\mu$ is a finite measure, which will later be relaxed to a $\sigma$-finite measure.

\begin{thm}\label{claim 5.3}
Suppose $\mu$ is a finite measure.
For any $A \in \mathcal{E}^{(k)}$, $k\in \bb{Z}_+$,  there exists a sequence  $(A_n)_{n \in \mathbb{Z}_+} $, with $A_n \in \mathcal{F}_k, n\in \bb{Z}_+$,  such that $\mu^k(A_n \Delta A) \rightarrow 0$ as $n\rightarrow\infty$. In addition, the $L^\gamma$  limit of $M_\alpha^{(k)}\left(A_n\right)$ as $n\rightarrow\infty$ exists for any $\gamma \in (0, \alpha)$ and does not depend on the choice of the approximation sequence $(A_n)_{n \in \mathbb{Z}_+} $.  Denoting this (a.s.) unique limit as   $M_\alpha^{(k)}(A)$,
we have 
\begin{equation}\label{rep}
\begin{aligned}
\bigl(M_\alpha^{(k)}(A)\bigr)_{A \in \mathcal{E}^{(k)}, \ k\in \bb{Z}_+}
 \ & \stackrel{d}{=}\ 
\Biggl(
   \ \bigvee_{\boldsymbol{j}\in \mathcal{D}_k}
   \pc{
      \prod_{r=1}^k \Gamma_{j_r}^{-1/\alpha}\, 
      \psi(T_{j_r})^{1/\alpha}
   }
   \mathbf{1}_{\{(T_{j_1},\dots,T_{j_k}) \in A\}}
\ \Biggr)_{A \in \mathcal{E}^{(k)}, \ k\in \bb{Z}_+} 
\\&=\pp{\bigvee_{\boldsymbol{j} \in \mathcal{D}_{k}}\left[\Gamma_{\boldsymbol{j}}\right]^{-1 / \alpha} \left[\psi(T_{\boldsymbol{j}})\right]^{1/\alpha}     \mathbf{1}_{\left\{T_{\boldsymbol{j}} \in A \right\}}}_{A \in \mathcal{E}^{(k)},\ k\in \bb{Z}_+}.
\end{aligned}
\end{equation}
Further, for any   $\gamma\in (0,\alpha)$ and $r>\alpha$, there is a constant $c>0$ depending on $(r,\gamma)$ only, such that
\begin{equation}\label{eq:M_alpha moment control}
    \left\|M_\alpha^{(k)}(A)\right\|_\gamma \leq c\left(\mu^k(A)\right)^{1 / r}
\end{equation}
for any $A\in \cl{E}^{(k)}$.  Here, $\|\cdot\|_\gamma=\left(\mathbb{E}|\cdot|^\gamma\right)^{1 / \gamma}$ is interpreted as the usual $L^\gamma$ quasi-norm when $0<\gamma<1$. 
Moreover, the product random sup measure $M_\alpha^{(k)}$ defined above satisfies $\sigma$-maxitivity, that is, for any collection of sets $B_j \in \mathcal{E}^{(k)}, j \in \mathbb{Z}_{+}$, we have 
\begin{equation}\label{eq:sigma maxitive M_alpha^k}
    M_\alpha^{(k)}\left(\bigcup_{i=1}^\infty B_i\right) =  \bigvee_{i=1}^\infty M_\alpha^{(k)}( B_i) \text{ a.s..}
\end{equation}
\end{thm}
\begin{proof}
The proof strategy is to relate   \eqref{rep} to the LePage series representation of positive multiple stable integrals, and make use of some known estimates of the latter.  The details are included in Appendix \ref{Pf 2.6}.  
\end{proof}

We now describe the construction of $M_\alpha^{(k)}$. Recall $\mu$ is a $\sigma$-finite   measure. Hence there exist sets $\left(E_n\right)_{n \in \bb{Z}_+}$ with $E_n \in \mathcal{E}$ such that $E=\bigcup_{n=1}^{\infty} E_n$, $E_n \subset E_{n+1}$ and $\mu\left(E_n\right)\in(0,\infty)$ for each $n \in \bb{Z}_+$.  We start with a fixed $\alpha$-Fréchet random sup measure $M_\alpha$ with control measure $\mu$.   Then $M_{\alpha,E_n}(\cdot):=M_\alpha(\cdot \cap E_n)$ is an $\alpha$-Fr\'echet random sup measure on $(E,\cl{E})$ in the sense of Definition \ref{defn s} with a finite control measure $\mu_n(\cdot):=\mu(\cdot\cap E_n)$. 
Now, for each $n\in\bb{Z}_+$, the product random sup measure  $M_{\alpha,E_n}^{(k)}$ on $(E^{(k)},\cl{E}^{(k)})$ can be constructed as in Theorem~\ref{claim 5.3}.  Then we define 
\begin{equation}\label{eq:gen M_alpha k def}
M_\alpha^{(k)}(A):=\bigvee_{n=1} M_{\alpha, E_n}^{(k)}(A), \quad A \in \mathcal{E}^{(k)},
\end{equation}
where the definition does not depend on the choice of the sequence $(E_n)_{n\in\bb{Z}_+}$ as shown below. 

\begin{thm}\label{thm sf case}
The product random sup measures defined in \eqref{eq:gen M_alpha k def}  admits the LePage representation \eqref{rep} jointly, and satisfies the $\sigma$-maxitive property \eqref{eq:sigma maxitive M_alpha^k}. If $(F_n)_{n\in \bb{Z}_+}$ is another sequence of subsets in $\cl{E}$ satisfying the same properties as $(E_n)_{n\in\bb{Z}_+}$, then we have for any $A\in \cl{E}^{(k)}$ that 
\[
\bigvee_{n=1}^\infty M_{\alpha,E_n}^{(k)}(A)=\bigvee_{n=1}^\infty M_{\alpha, F_n}^{(k)}(A) \ \text{a.s.}.
\]
\end{thm}
\begin{proof}
Suppose $m$, $\psi=d\mu/dm$, $\pp{T_i}_{i \in \bb{Z}_+}$ and $\pp{\Gamma_i}_{i \in \bb{Z}_+}$ are as in Definition \ref{defn l}. Note that $d\mu_n/dm=\psi \mathbf{1}_{E_n}$. 
By relation \eqref{m and ml1}, one can verify that the sequence of random sup measures $\pp{ M_{\alpha,E_n}}_{n\in \bb{Z}_+}$ admits the joint LePage representation  
$$
\pp{M_{\alpha,E_n}(A)}_{A\in \cl{E}, n\in \bb{Z}_+}  \stackrel{d}{=} \pp{\bigvee_{i \geq 1}\psi\left(T_i\right)^{1 / \alpha} \Gamma_i^{-1 / \alpha}  \mathbf{1}_{\left\{T_i \in  A \cap E_n\right\}}}_{A\in \cl{E}, n\in \bb{Z}_+},
%=:\pp{M_{\alpha,E_n}^L(A)}_{A\in \cl{E}, n\in \bb{Z}_+},
$$
where ``$\stackrel{d}{=}$'' is understood as equality in finite-dimensional distributions in indices $A$ and $n$. It follows from this relation above and the construction of $M_{\alpha,E_n}^{(k)}$ in Theorem \ref{claim 5.3} (via also \eqref{prod M_a} and \eqref{max})  that 
\begin{equation}\label{eq:E_n prod series}
\pp{M_{\alpha,E_n}^{(k)}(A)}_{A\in \cl{E}^{(k)}, \ k\in \bb{Z}_+,   \, n\in \bb{Z}_+}  \stackrel{d}{=} \pp{ \bigvee_{\boldsymbol{j} \in \mathcal{D}_{k}}\left[\Gamma_{\boldsymbol{j}}\right]^{-1 / \alpha} \left[\psi(T_{\boldsymbol{j}})\right]^{1/\alpha}     \mathbf{1}_{\left\{T_{\boldsymbol{j}} \in A \cap E_n^{k} \right\}} }_{A \in \mathcal{E}^{(k)},  \ k\in \bb{Z}_+, \, n\in \bb{Z}_+}.
\end{equation}
The first conclusion regarding the LePage representation then follows from this and \eqref{eq:gen M_alpha k def}, noting that $\vee_{n\ge 1}\mathbf{1}_{\left\{T_{\boldsymbol{j}} \in A \cap E_n^k \right\}}=\mathbf{1}_{\left\{T_{\boldsymbol{j}} \in A  \right\}}$, $A\in \cl{E}^{(k)}$. The $\sigma$-maxitive property \eqref{eq:sigma maxitive M_alpha^k} follows from the LePage representation. 
 
Now we prove the last conclusion.  Let $M_{\alpha, E}=\bigvee_{\ell=1}^\infty M_{\alpha,E_\ell}^{(k)}$ and $M_{\alpha, F}= \bigvee_{\ell=1}^\infty M_{\alpha, F_\ell}^{(k)}$. The goal is to show $M^{(k)}_{\alpha, E}(A)=M^{(k)}_{\alpha, F}(A)$ a.s.\ for any $A\in\cl{E}^{(n)}$.
First, note that for any $n \in \mathbb{Z}_+$ and any $A \in \mathcal{F}_k$, $A \subset E_n^k \cap F_n^k$, we have 
  $  M_{\alpha, E_n}^{(k)}(A)
    =     M_{\alpha, F_n}^{(k)}(A)$ a.s. in view of \eqref{prod M_a} and \eqref{max}. 
 Based on the first conclusion of Theorem \ref{claim 5.3} but with $E$ replaced by  the subspace $E_n\cap F_n$,  we have for any $A \in \mathcal{E}^{(k)} \cap (E_n^k \cap F_n^k)$, $n \in \mathbb{Z}_+$, there exists a sequence of sets $\left(B_i\right)_{i \in \mathbb{Z}_{+}}$, where $B_i \in \mathcal{F}_k, B_i \subset E_n^k \cap F_n^k$ for each $i \in \mathbb{Z}_{+}$ such that $\mu^k\left( B_i \Delta A \right)\rightarrow 0$ as $i \rightarrow \infty$. So it follows from Theorem \ref{claim 5.3}  that $ M_{\alpha,E_\ell}^{(k)}\left(A\right) =   M_{\alpha,F_\ell}^{(k)}\left(A\right) $ a.s. for any $A \in \mathcal{E}^{(k)} \cap (E_n^k \cap F_n^k)$ if $\ell\ge n$. Note also that by \eqref{eq:E_n prod series}, we have $M_{\alpha,E_\ell}^{(k)}\left(A\right)\le M_{\alpha,E_{\ell+1}}^{(k)}\left(A\right)$ a.s., and a similar relation holds for  $F_\ell$.
    As a consequence, in view of \eqref{eq:gen M_alpha k def}, we have 
    $$M_{\alpha,E}^{(k)}\left(A  \cap E_n^k \cap F_n^k\right) = M_{\alpha, F}^{(k)}\left(A  \cap E_n^k \cap F_n^k\right) \ \text{a.s.}$$ for any $A \in \mathcal{E}^{(k)}$, $n \in \mathbb{Z}_+$.
    It suffices to let $n\rightarrow\infty$ in the relation above and apply the $\sigma$-maxitivity property, noting that $\cup_{n\ge 1} A\cap E_n^k \cap F_n^k  =A$.

\end{proof}

\begin{rem}
% When $\mu(E)=\infty$, it is possible that $M^{(k)}_\alpha(A)=\infty$ a.s.\ even when $\mu^k(A)<\infty$; One mechanism is to take $A$ not contained in a rectangle of form $B_1 \times \cdots \times B_k$ with $\mu\left(B_i\right)<\infty$ (see Example \ref{S1}) ; otherwise, by monotonicity, $M_\alpha^{(k)}(A) < \infty$ a.s.. This contrasts with $M_\alpha$, where $\mu(E)=\infty$ always implies $M_\alpha(A)=\infty$ a.s.. 
When \(\mu(E) = \infty\), it is possible that \(M^{(k)}_\alpha(A) = \infty\) a.s.\ even if 
\(\mu^k(A) < \infty\). This may happen for certain special sets \(A\) that are not contained in 
any rectangle of the form \(B_1 \times \cdots \times B_k\) with \(\mu(B_i) < \infty\) for all 
\(i\) (see Example~\ref{S1} below); otherwise, by monotonicity, \(M_\alpha^{(k)}(A) < \infty\) a.s.. 
This behavior contrasts with that of \(M_\alpha\), for which \(\mu(A) < \infty\) always implies 
\(M_\alpha(A) < \infty\) a.s.
\end{rem}

\subsection{Multiple extremal integrals}\label{s 2.3}
Up to this point, we have established the definition of the product random sup measure $M_\alpha^{(k)}(A)$ for $A$ in the off-diagonal $\sigma$-field $ \mathcal{E}^{(k)}$. Starting  with $M_\alpha^{(k)}$,  the construction of multiple extremal integral follows a routine path: First define the integral for simple functions on $\cl{E}^{(k)}$, and then extend to nonnegative measurable functions via monotone approximations. The details are given below.

Let $\mathcal{S}_k$, $k \in \bb{Z}_+$, be the collection of  non-negative  simple  functions on $\pp{E^{k},\cl{E}^{k}}$  vanishing on the diagonal set $D^{(k)}$, i.e.,  each $f \in \mathcal{S}_k$ is of the form
\begin{equation}\label{sim f}
    f\left(u_1, \ldots, u_k\right)=\sum_{i=1}^N a_i \mathbf{1}_{\{(u_1, \ldots, u_k) \in A_i\}}=\bigvee_{i=1}^N a_i \mathbf{1}_{\{(u_1, \ldots, u_k) \in A_i\}},
\end{equation}
where  $a_1, \ldots,a_N \in (0,\infty)$,  and $A_1, \ldots, A_N$ are  disjoint subsets  belonging to $\mathcal{E}^{(k)}$,  $N\in\bb{Z}_+$. We define the multiple extremal integral of the function $f$ in (\ref{sim f}) with respect to  $M^{(k)}_\alpha$ as 
\begin{equation}\label{simple Sk}
\begin{aligned}
I_k^e(f)  
% \leftindex^{e} {\int}_{E^{(k)}} f\left(u_1, \ldots, u_k\right) M_\alpha\left(d u_1\right) \ldots M_\alpha\left(d u_k\right)  
:=  \bigvee_{i=1}^N a_i M^{(k)}_\alpha\left(A_i\right).
\end{aligned}
\end{equation}

We first state some elementary properties of multiple extremal integrals of simple functions.
\begin{prop}\label{I^e_k up}
   Suppose $f,g\in \mathcal{S}_k$, $k \in \mathbb{Z}_+$. 
\begin{enumerate}
\item (max-linearity) For any constants $a, b \geq 0$, we have  $I_k^e(af \vee b g)  =aI_k^e( f) \vee b I_k^e( g)$ a.s..

\item (monotonicity) If \( f(\boldsymbol{u}) \leq g(\boldsymbol{u}) \) for \( \mu^k \)-a.e. \( \boldsymbol{u} \in E^{k} \), then \( I^e_k(f) \leq I^e_k(g) \) a.s..
\item (triangle inequality) $I_k^e(f+g)\le I_k^e(f) + I_k^e(g)$ a.s..  
\end{enumerate}

\end{prop}
\begin{proof}
The proof is similar to those for Propositions 2.2 (i), 2.3 and 2.8 in \cite[]{stoev2005extremal}, and we omit the details. 
 
\end{proof}
We are now ready to define multiple extremal integral of a $k$-variate ($k \geq 1$) non-negative measurable function $f$ on $E^{k}$ that vanishes on the diagonal set $D^{(k)}$ (we will simply say $f$ vanishes on the diagonals in the rest of the paper). Suppose $f_n \in \mathcal{S}_k$, satisfy $f_n(\boldsymbol{u}) \leq f_{n+1}(\boldsymbol{u}), \boldsymbol{u} \in  E^{k}$,  $n\in \mathbb{Z}_+$, and $\lim_n f_n (\boldsymbol{u}) = f(\boldsymbol{u}) $,  $ \mu^k$-a.e., denoted by $f_n \nearrow f$.  One can take, for example, classically
$f_n(\boldsymbol{u})=\sum_{j=1}^{n 2^n-1} j / 2^n \mathbf{1}_{\left\{f \in\left[j / 2^n,(j+1) / 2^n\right)\right\}}(\boldsymbol{u})+ n  \mathbf{1}_{\pc{f\ge n}}$, $n\in \bb{Z}_+$. 

\begin{defn}\label{def:mult extr int}
Suppose $f:E^k\mapsto [0,\infty]$ is measurable and vanishes on the diagonals. 
  The multiple extremal integral $I_k^e(f)$  is defined as  the a.s.\ limit of $I_k^e(f_n)$ as $n\rightarrow\infty$, where $f_n \in \mathcal{S}_k$, $n \in \bb{Z}_+$, and $f_n \nearrow f$ as $n\rightarrow\infty$.
\end{defn}

It can be verified through an argument similar to the one in Lebesgue measure theory (e.g., \cite[Lemma 1.20]{kallenberg2021foundations}) that the $I_k^e(f)$ defined above does not depend on the choice of the approximation sequence $(f_n)_{n\in \bb{Z}_+}$. We defer the details to Lemma \ref{General consistency} in the Appendix.  Some basic properties of $I_k^e(f)$ are summarized below.

\begin{cor}\label{cor:gen int}
   The max-linearity, monotonicity and triangle inequality in Proposition \ref{I^e_k up}  extend to the case where the integrands $f$ and $g$ are general  measurable functions:  $E^{k}\mapsto [0,\infty]$ that vanish on the diagonals for $k \in \mathbb{Z}_+$.   In addition, we have the LePage representation
\begin{equation}\label{eq4}
\pp{  I_k^e(f) }_{f\in \cl{L}_k,\, k\in \bb{Z}_+}  \overset{d}{=}  \pp{S_k^e(f)}_{f\in \cl{L}_k,\,  k\in \bb{Z}_+} := \pp{\bigvee_{\boldsymbol{j} \in \mathcal{D}_k}f\left(T_{j_1}, \ldots, T_{j_k}\right) \left(\prod_{r=1}^k \Gamma_{j_r}^{-1 / \alpha} \psi\left(T_{j_r}\right)^{1 / \alpha}\right)}_{f\in \cl{L}_k,\, k\in \bb{Z}_+},
\end{equation}
where  
\begin{equation}\label{eq:cl L_k}
\cl{L}_k:=\pc{f: E^k \mapsto [0,\infty] \mid  f\text{ is measurable and vanishes on the diagonals}}.
\end{equation}
Also, $m, \psi$, $\pp{T_i}_{i \in \bb{Z}_+}$ and $\pp{\Gamma_i}_{i \in \bb{Z}_+}$ are as in Definition \ref{defn l}. 
% Furthermore, we have $f=0$ $\mu^k$-a.e.\ if and only if $I_k^e(f)=0$ a.s..
\end{cor}
\begin{proof}
 The first three claims concerning items in Proposition \ref{I^e_k up} follow from Proposition \ref{I^e_k up} and Definition \ref{def:mult extr int}.   To see relation \eqref{eq4}, first note that it holds when $ \cl{L}_k$ is replaced by $\cl{S}_k$  in view of Theorem \ref{thm sf case} and \eqref{simple Sk}. Then apply the approximation in Definition \ref{def:mult extr int}. 
  
\end{proof}

Given a multiple extremal integral $I_k(f)$ with respect to an $\alpha$-Fr\'echet  random sup measure $M_\alpha$, $\alpha>0$, we note that
the power transform $\pp{I_k^e(f)}^r$, $r>0$,   results in a multiple extremal integral as well, but with respect to the $(\alpha/r)$-Fr\'echet  random sup measure $M_{\alpha}^r$. When  $k =1$, this fact has been also mentioned in \cite[Proposition 2.9]{stoev2005extremal}.
Next, we show that the extremal integral remains invariant under any permutation of the coordinates of the integrand. For a measurable $f:E^{k}\mapsto [0,\infty]$ that vanishes on the diagonals, we define its max-symmetrization as
\begin{equation}\label{sym def}
\widetilde{f}\left(u_1, \ldots, u_k\right)=\bigvee_{\pi \in \Theta_k} f\left(u_{\pi(1)}, \ldots, u_{\pi(k)}\right),
\end{equation}
where $\Theta_k$ consists of all permutations (one-to-one mappings) $\pi:\{1,2, \ldots, k\} \mapsto\{1,2, \ldots, k\}$.   It is worth noting that, in contrast, a similar symmetrization often performed for multiple additive stochastic integrals involves an additive average over $\Theta_k$ instead.
\begin{prop}\label{sim inv}
 For a function  $f \in \cl{L}_k$, $k \in \mathbb{Z}_+$, we have  $I_k^e(f)=I_k^e(\widetilde{f})$ a.s.. 
\end{prop}
\begin{proof}
The argument is routine and we only provide a sketch.
In view of \eqref{prod M_a},  the conclusion holds for $f=\mbf{1}_A$ when $A\in \cl{C}_k$, and then extends to $A \in \mathcal{F}_k$ via \eqref{max}. Through a  set symmetric difference approximation  in the spirit of Theorem \ref{claim 5.3}, the conclusion extends to any $A\in \cl{E}^{(k)}\cap E_n$, where  each $E_n$, $n\in \bb{Z}_+$, is as in Theorem \ref{thm sf case}, and then to  general  $A\in \cl{E}^{(k)}$ by letting $n\rightarrow\infty$ and applying $\sigma$-maxitivity.  At last, the conclusion extends to $f\in \cl{S}_k$ by max-linearity, and then to general measurable $f$ via Definition \ref{def:mult extr int}.
\end{proof}

 \begin{rem}\label{rem:alt constr}
 
In this paper, we primarily develop the definition of multiple extremal integrals with respect to $M_\alpha$ given in Definition~\ref{defn s}, rather than the Poisson point process–based construction (i.e., using $M_\alpha^P$ in \eqref{eq:pois RSM} below). The reason is that the former notion is strictly weaker: every $M_\alpha^P$ is an $M_\alpha$, but not conversely (see Remark~\ref{Rem:s vs l}). This choice therefore affords greater generality. Nevertheless, the Poisson point process–based construction is discussed in Section~\ref{alt ppp}.

A third possible approach is to define multiple extremal integrals directly via the LePage   representation \eqref{eq4}, as in \cite[]{samorodnitsky1991construction} for multiple stable integrals. This method, however, requires the prior specification of a probability measure $m$ equivalent to $\mu$, with the associated Radon–Nikodym derivative $\psi = d\mu/dm \in (0,\infty)$ $m$-a.e.. While the choice of $m$ does not affect the distribution of the resulting multiple extremal integral, its introduction makes the construction less intrinsic.

 \end{rem}

 \begin{rem}  
 One may also construct multiple extremal integrals with respect to the other two classical types of max-stable distributions: \emph{Gumbel} and \emph{reverse Weibull}. Recall that if $\xi$ is a standard $\alpha$-Fr\'echet random variable with $\alpha \in (0,\infty)$, then $\ln(\xi^\alpha)$ has the standard Gumbel distribution, while $1/(-\xi)$ follows the standard $\alpha$-reverse-Weibull distribution. Therefore, once a multiple $\alpha$-Fr\'echet extremal integral $I_k^e(f)$ has been constructed, the monotone transformations $\ln\bigl( I_k^e(f)^\alpha \bigr)$ and $1/\bigl(-I_k^e(f)\bigr)$ can be viewed as multiple Gumbel and multiple reverse-Weibull extremal integrals, respectively. 
It is worth noting that, in the multiple Gumbel case, the multiplicative relations such as \eqref{prod M_a} are replaced by additive ones. Moreover, multiple extremal integrals with respect to \emph{random inf measures}---obtained by replacing the role of the supremum by the infimum in the definition of random sup measures---can be constructed analogously by exploiting the reflection relation between supremum and infimum. We omit the details.
\end{rem}

\subsection{Alternative construction via Poisson point process} \label{alt ppp}
 An alternative approach to constructing multiple extremal integrals can be formulated with the enumerated points of a Poisson point process.  For this purpose, we shall follow \cite[P.15]{kallenberg2021foundations} to assume  that the Borel space $(E,\cl{E})$ is  localized by a sequence $E_n\in \cl{E}$, $n\in \bb{Z}_+$, such that $E_n\subset E_{n+1}$ and $\cup_n E_n=E$. A subset $B\subset E$ is said to be bounded if $B\subset E_n$ for some $n\in \bb{Z}_+$. We also assume $\mu$ is a locally finite measure on $(E,\cl{E})$, i.e., $\mu$ is finite on bounded subsets in $\cl{E}$. Below we understand the Borel space $E\times (0,\infty)$ as localized by $E_n\times [1/n,\infty)$, $n\in \bb{Z}_+$.

Now let $N$ be a Poisson point process  on $S = E\times (0,\infty)$ with intensity measure $\mu \times \nu_\alpha$, where $\nu_\alpha(dx) = \alpha x^{-\alpha-1}dx$, $x > 0$. Let $\mathcal{S}$ be the Borel $\sigma$-field on $S$. We regard $N$ as a random element taking value in $\cl{M}_S$, the measurable space of locally finite measures on $(S,\mathcal{S})$; see, e.g., \cite[P.44]{kallenberg2021foundations}.   In view of \cite[Theorem 2.19 (i)]{kallenberg2021foundations}, one can enumerate the atoms of $N$ measurably, i.e.,  there exists a measurable mapping $\phi:\cl{M}_S\mapsto E^\infty \times (0,\infty)^\infty $, such that $N=\sum_{i=1}^\infty \delta_{(\xi_i,\eta_i)}$ with $(\xi_i,\eta_i)_{i\in \bb{Z}_+}:=\phi(N)$. 
A random sup measure can then be defined as
\begin{equation}\label{eq:pois RSM}     
M_\alpha^P(A):=\bigvee_{i \geq 1} \mathbf{1}_{\{\xi_i \in A\}} \eta_i, \ A\in \cl{E},
\end{equation}
which satisfies the pathwise $\sigma$-maxitive property \eqref{eq l sigma add} in addition to Definition \ref{defn s}. Moreover, for a function $f \in \cl{L}_k$, the multiple integral of $f$ with respect to $M_\alpha^P$ can be directly expressed as
\begin{equation}\label{eq:ppp rep}
\leftindex^e \int_{E^k} f\left(u_1, \ldots, u_k\right) M_\alpha^P\left(d u_1\right) \ldots M_\alpha^P\left(d u_k\right)= \kappa_f\pp{(\xi_i,\eta_i)_{i\in \bb{Z}_+}}:=\bigvee_{\boldsymbol{j} \in \mathcal{D}_{k}}f(\xi_{\mbf{j}}) \left[\eta_{\boldsymbol{j}}\right],
\end{equation}
so that the multiple integral can be identified as a measurable map $\kappa_f \circ \phi$ of $N$. On the other hand, with the notation in Definition \ref{defn l}, one may introduce  $\wt{N}:=\sum_{i=1}^\infty\delta_{(T_i, \, \psi(T_i)^{1/\alpha} \Gamma_i^{-1/\alpha})}$, which we claim to be a Poisson point process on $E\times (0,\infty)$ with intensity measure $\mu\times \nu_\alpha$ as well. Indeed, this can be derived by first noting that $\sum_{i=1}^\infty\delta_{(T_i, \,  \Gamma_i^{-1/\alpha})}$ is a Poisson point process with mean measure $m\times \nu_\alpha$, and then applying the mapping theorem \cite[Theorem 15.3]{kallenberg2021foundations} via the map $(x,y)\mapsto (x,\psi^{1/\alpha}(x)y)$. Hence 
\begin{equation}\label{D.2}
    \pp{\kappa_f\circ \phi(N)}_{f\in \cl{L}_k}\EqD \pp{\kappa_f\circ \phi(\wt{N})}_{f\in \cl{L}_k}=\pp{S_k^e(f)}_{f\in \cl{L}_k},
\end{equation}
where $S_k^e(f)$ is the LePage representation as in \eqref{eq4}.

\section{Integrability}\label{sec:int}

% \subsection{Basic characterizations and
% properties}
 
\begin{defn}
We say that  a function $f \in \cl{L}_k$ (see \eqref{eq:cl L_k}) is (multiple-)integrable with respect to a random sup measure $M_\alpha$ with control measure $\mu$ on $E$ as in Definition \ref{defn s}, if $I_k(f)<\infty$ a.s., where  $I_k(f)$ is the multiple extremal integral with respect to $M_\alpha$ as in Definition \ref{def:mult extr int}. 
%The class of integrable functions is denoted as $\mathcal{I}_k (\mu)$.
\end{defn}
In view of \eqref{eq4},  for  a function $f \in \cl{L}_k$, we have 
\begin{equation}\label{eq class}
 \text{$I_k(f)<\infty$ a.s.\ if and only if  $S^e_k(f)<\infty$ a.s..}
\end{equation}
 We first establish a zero-one law for the integrability of multiple extremal integrals.
\begin{prop}\label{01law}
   For $f \in \mathcal{L}_k$, $k \in \mathbb{Z}_+$, $\mathbb{P}\left( I_k^e(f) < \infty \right) = 0 \text{ or } 1$. 
\end{prop}
\begin{proof}
 First, we assume $\psi=1$ without loss of generality. Second, we assume that $f<\infty$ a.e.; otherwise, one can show $S_k^e(f)= \infty$ a.s..
Now we claim that
$$
\bigvee_{j \in \mathcal{D}_k} f\left(T_{\boldsymbol{j}}\right)\left[\Gamma_{\boldsymbol{j}}\right]^{-1 / \alpha} < \infty \text { a.s. } \quad \Longleftrightarrow \bigvee_{\boldsymbol{j} \in \mathcal{D}_k} f\left(T_{\boldsymbol{j}}\right)[\boldsymbol{j}]^{-1 / \alpha} < \infty \text { a.s.. }
$$
This follows because $\left(T_i\right)_{i \in \mathbb{Z}_{+}}$is independent of $\left(\Gamma_i\right)_{i \in \mathbb{Z}_{+}}$, and $\Gamma_j / j \rightarrow 1$ a.s.\ as $j \rightarrow \infty$, by the strong law of large numbers. Here, we may additionally assume $\alpha = 1$; otherwise apply the monotone power transform $x\mapsto x^{\alpha}$. Let 
$$A=\left\{ \bigvee_{\boldsymbol{j} \in \mathcal{D}_k} f\left(T_{\boldsymbol{j}}\right)[\boldsymbol{j}]^{-1}=\infty\right\}.$$
We shall show that $A$ is in the exchangeable $\sigma$-field w.r.t $\left(T_i\right)_{i \in \mathbb{Z}_{+}}$, so that the Hewitt–Savage zero-one law (\cite[Theorem 4.15]{kallenberg2021foundations}) applies. Next, we introduce $A_q=\{f\left(T_{\boldsymbol{j}}\right)[\boldsymbol{j}]^{-1} >q\text{ for some } \boldsymbol{j} \in \mathcal{D}_k\},q \in \mathbb{Q}_{+}$ (positive rationals).
Then $A=\bigcap_{q \in \mathbb{Z}_+} A_q$ modulo null sets. Now fix a bijection (permutation) $\pi_m: \mathbb{Z}_{+} \rightarrow \mathbb{Z}_{+}$ that is identity when restricted to $\{m+1, m+2, \ldots\}$,   $m \in \mathbb{Z}_{+}$. It suffices to show 
\begin{equation}\label{goal}
    A^{\pi_m}:=\left\{\bigvee_{\boldsymbol{j} \in \mathcal{D}_k} f\left(T_{\pi_m(\boldsymbol{j})}\right)[\boldsymbol{j}]^{-1}=\infty\right\} = A
\end{equation}
modulo null sets, where $\pi_m(\sbf{j})=(\pi_m(j_1),\ldots,\pi_m(j_k))$.
Set $ A_q^{\pi_m}=\{f\left(T_{\pi_m(\boldsymbol{j})}\right)[\boldsymbol{j}]^{-1} >q \text{ for some } \boldsymbol{j} \in \mathcal{D}_k\},q \in \mathbb{Q}_{+}$. Note that $A^{\pi_m} = \bigcap_{q \in \mathbb{Z}_{+}} A_q^{\pi_m}$ modulo null sets.
Hence, to verify \eqref{goal}, it is enough to show for any $q \in \mathbb{Q}_{+}$, $$A_q^{\pi_m} \subset A_{q m^{-k}} \text{ and }  A_q \subset A_{q m^{-k}}^{\pi_m}.$$
Indeed, this follows from the fact that    $m^{-k}\le \frac{[\pi_m(\sbf{j})]}{[\sbf{j}]}\le m^k$ for any $\boldsymbol{j} \in \mathcal{D}_k$.
\end{proof}

\subsection{Sufficient conditions}\label{suff sec}
Recall \cite[]{stoev2005extremal} established that in the case $k=1$, the condition $f\in L^\alpha_+(\mu)$ is  both necessary and sufficient   for the integrability (i.e., $I_1^e(f) < \infty$ a.s.). When $k\ge 2$, the condition $f\in L^\alpha_+(\mu^k)$ is only a necessary condition for integrability, but not a sufficient one.  We shall establish some sufficient conditions in this section.

First,  in view of \eqref{eq class}, it suffices to establish a sufficient condition for $S_k^e(f)<\infty$ a.s.. For this purpose, we shall take advantage of some known results on multiple stable integrals. Following \cite[]{samorodnitsky1989asymptotic},  we adopt the following notation: For a  function $f \in \cl{L}_k$, where $k\ge 2$, and a measure $\mu$ on $E$, set
\begin{equation}\label{defn l space}
\begin{aligned}
     L^\alpha \ln ^{k-1} L(f,\mu)   = \int_{E^k} f^\alpha(\boldsymbol{u}) \left(1+ (\ln_{+}f(\boldsymbol{u}))^{k-1} \right) \mu^k(d \boldsymbol{u}),
\end{aligned}
\end{equation}
where   $\ln_{+} x:=\ln (x\vee 1)$, $x\in [0,\infty]$.
% \sy{the first relation above when $k=2$ seems weaker than the one below; see \cite[Theorem 4.4.2]{rosinski1983products}; Here is what I would propose. Combine the sufficient conditions into a single results. Unify the proofs.}

The following theorem establishes a set of sufficient conditions for the integrability of multiple $\alpha$-Fr\'echet extremal integrals.
% \suggestion{comparable to those in \cite[Theorem 5.3]{samorodnitsky1989asymptotic} for multiple symmetric $\alpha$-stable ($S\alpha S$) integrals}
% We note that
% in the extremal case, the index $\alpha$ can take any positive value, while in the $S\alpha S$ case, it is restricted to $\alpha \in(0,2)$. 
% By 
% Proposition \ref{sim inv}, we may assume $f$ is symmetric to simplify the proof when studying the distributional properties of $I_k^e(f)$.

\begin{thm}\label{suff thm}
For a function $f \in \cl{L}_k$,  $k\ge 2$,  a sufficient condition for integrability  $I_k^e(f)< \infty$ a.s., is that there exists a probability measure $m$ on $E$ equivalent to $\mu$ with  $\psi=d\mu/dm\in (0,\infty)$ $m$-a.e., such that 
\begin{equation}\label{suff 2}
    L^\alpha \ln^{k-1} L(f\cdot(\psi^{\otimes k})^{1/\alpha}, m) < \infty,
\end{equation}
where  the $k$-variate tensor product function $\psi^{\otimes k}$ is defined by $\psi^{\otimes k}(x_1, ,\ldots,x_k): = \psi(x_1)\otimes \ldots \otimes\psi(x_k)$.

\end{thm}

\begin{rem}
 We first note that the finiteness in \eqref{suff 2} depends on the choice of $m$: it can hold for some probability measure $m$ equivalent to $\mu$ but fail for another (see Example \ref{exam thm 3.2}). Moreover, in contrast to condition \eqref{suff 2}, the condition $L^\alpha \ln ^{k-1} L(f, \mu)<\infty$ is not sufficient (see Example \ref{not suff exam}). 
 Besides, we point out that when $k=2$, condition  \eqref{suff 2} is equivalent to
 \begin{equation}\label{I2}
\int_{E^{2}}f(s, t)^\alpha\left(1+\ln_{+}  \pp{\frac{f(s, t)}{\left(\int_E f(s, u)^\alpha  \mu(du)\right)^{1 / \alpha}\left(\int_E f(u, t)^\alpha \mu(du)\right)^{1 / \alpha}}}\right) \, \mu(ds)\mu(dt)<\infty,
\end{equation}
where the ratio inside $\ln_{+}$ is understood as $1$ in the case $0/0$, a celebrated necessary and sufficient condition for integrability of a double symmetric $\alpha$-stable integral \cite[]{rosifiski1986ito,kwapien1987double}. The proof of this equivalence will be deferred to Proposition \ref{prop k=2}.
\end{rem}

\begin{proof}[Proof of Theorem \ref{suff thm}]
%First, we consider the case where $k \geq3$. 
 We now apply the equivalent characterization of integrability given in \eqref{eq class}. 
Without loss of generality, one may assume that $\mu$ is a probability measure and 
$\psi \equiv 1$. More precisely, by the LePage representation \eqref{eq4}, the original 
multiple extremal integral $I_k^e(f)$ with respect to the random sup measure $M_\alpha$ 
is equal in distribution to another multiple extremal integral whose integrand is 
$f \cdot (\psi^{\otimes k})^{1/\alpha}$, and whose  random sup measure is 
$\bigvee_{i=1}^{\infty} \mathbf{1}_{\{T_i \in \cdot\}}\, \Gamma_i^{-1/\alpha}$, with control 
measure given by the probability measure $m = \mathbb{P} \circ T_1^{-1}$.
Furthermore, it suffices to restrict attention to  symmetric $f$.  
Indeed, this follows from \eqref{sim inv} and the fact that  the relations \eqref{I2} and \eqref{suff 2}  hold   if and only if they hold for $\wt{f}$; see Lemma \ref{lem eqv} for a proof of these equivalences.

Next, suppose $r>\alpha$. Note
%raising a nonnegative random variable to a positive power preserves finiteness, we have
$$
S_k^e(f)<\infty \text { a.s. } \Longleftrightarrow\left(S_k^e(f)\right)^r <\infty \text { a.s.. }
$$
Moreover,
\begin{equation}\label{eq:S_k^r<inf}
\left(S_k^e(f)\right)^{r} = \bigvee_{\boldsymbol{j} \in \mathcal{D}_{k,<}} f\left(T_{\boldsymbol{j}}\right)^{r} \left[\Gamma_{\boldsymbol{j}}\right]^{-1 / (\alpha/r)} \le \sum_{\boldsymbol{j} \in \mathcal{D}_{k,<}} f\left(T_{\boldsymbol{j}}\right)^{r} \left[\Gamma_{\boldsymbol{j}}\right]^{-1 / (\alpha/r)},
\end{equation}
where the assumed symmetry of $f$  allows us to reduce the index set from $\cl{D}_k$ in \eqref{no order} to $\cl{D}_{k,<}$ in \eqref{D<}.  Hence, it suffices to identify sufficient conditions to ensure the last expression in \eqref{eq:S_k^r<inf} is a.s.\ finite for $k \geq2$. 

Now we claim that, under the Borel assumption on $E$, it suffices to verify the conditions \eqref{I2} and \eqref{suff 2}  ensuring that  the last expression in \eqref{eq:S_k^r<inf} is a.s.\ finite for $k \geq2$ in the special case  $E=[0,1]$, $\mu=\lambda$, with $T_j = U_j, j \in \mathbb{Z}_+,$ where $(U_j)_{j \in \mathbb{Z}_+}$ are i.i.d.\ uniform random variables on $[0,1]$.  Suppose the claim holds in this special case. Then for a general Borel space $E$, it follows from \cite[Lemma 4.22]{kallenberg2021foundations} that there exists a measurable map $\rho:[0,1] \mapsto E$ such that $\mu=\lambda \circ \rho^{-1}$, and hence $(\rho\left(U_j\right))_{j \in \mathbb{Z}_+} \stackrel{d}{=} (T_j)_{j \in \mathbb{Z}_+}$. Therefore, for $k >2$,
\begin{equation}\label{l1}
    L^\alpha \ln ^{k-1} L\left( f_\rho, \lambda^k\right) =L^\alpha \ln ^{k-1} L\left( f, \mu^k\right). 
\end{equation}
where $f_\rho: E^k \mapsto [0,\infty] $ is defined by $f_\rho(u_1 ,\ldots,u_k) = f(\rho(u_1), \rho(u_2)\ldots, \rho(u_k))$.
Since, by the claim above, the finiteness of the left-hand side of \eqref{l1} implies a.s. finiteness of 
$$\bigvee_{\boldsymbol{j} \in \mathcal{D}_k} f_\rho\left(U_{j_1}, \ldots, U_{j_k}\right)\left[\prod_{\ell=1}^k \Gamma_{j_\ell}^{-1 / \alpha} \right],$$
which equals in distribution to  $ S_k^e(f)$, it follows that $L^\alpha \ln ^{k-1} L\left( f, \mu^k\right) < \infty$ ensures $S_k^e(f) < \infty$ a.s.. An analogous argument applies to condition \eqref{I2} for $k=2$.

Next, assume $E = [0,1]$, $\mu = \lambda$, and $T_j=U_j, j \in \mathbb{Z}_{+}$. Let $\widetilde{\alpha} = 2\alpha/r \in (0,2)$, and set  $g = f^{\alpha/\wt{\alpha}}$.   It remains to show the condition \eqref{suff 2} implies that for $k \geq 2$,
\begin{equation}\label{eq:series finite 2}
\sum_{\boldsymbol{j} \in \mathcal{D}_{k, <}}g\left(U_{\boldsymbol{j}}\right)^2 \left[\Gamma_{\boldsymbol{j}}\right]^{-2 / \widetilde{\alpha}} < \infty \quad \text{a.s..}
\end{equation}
For $k>2$, the moment condition in  \cite[Theorem 5.3]{samorodnitsky1989asymptotic} implies their relation (5.2),  which translated into our context means   $L^{\wt{\alpha}} \ln ^{k-1} L(g, \lambda) < \infty$ implies (\ref{eq:series finite 2}). 
For $k = 2$, in view of \cite[Theorem 1.3. (i) and Theorem 2.1]{samorodnitsky1989asymptotic},
the condition (\ref{eq:series finite 2}) is equivalent to 
the convergence of the series representation of a double symmetric $\widetilde{\alpha}$-stable integral of $g$. Moreover, $L^{\wt{\alpha}} \ln ^{k-1} L(g, \lambda) < \infty$ is known to be sufficient for the existence of such an integral (see \cite[]{breton2002multiple}). Hence, a straightforward calculation shows that $L^{\alpha} \ln ^{k-1} L(f,\lambda) < \infty$, $k \geq 2$, serves as a sufficient condition for (\ref{eq:series finite 2}).

\end{proof}
%\tcp{To be revised. It looks like from the literature: Surgailis, D. "On the multiple stable integral. "Zeitschrift für Wahrscheinlichkeitstheorie und verwandte Gebiete 70.4 (1985): 621-632.,  \cite{breton2002multiple}  we have   \eqref{suff 2} $\Rightarrow$ \eqref{I2}.}

\subsection{Necessary condition}\label{s dec}

In this section,  we shall establish  the following necessary condition for integrability.
\begin{thm}\label{thm:nec cond}
 For a  function $f \in \cl{L}_k$,  $k\in \bb{Z}_+$, a necessary condition for integrability $I_k^e(f)<\infty$ a.s., is $f \in L_{+}^\alpha\left(\mu^k\right)$.
\end{thm}

This necessary condition is established through a
 decoupling argument: We  first establish a stochastic order between the multiple integral $I^e_k(f)$ and its decoupled version constructed using i.i.d.\ copies of the random sup measure defining $I^e_k(f)$. The idea can be illustrated by the following formal derivation. Suppose for simplicity $k=2$, and let $M_\alpha^{(1)}$ and $M_\alpha^{(2)}$ be i.i.d.\ copies of $M_\alpha$. By the max-stability property $M_\alpha\EqD 2^{-1/\alpha} \pp{M_\alpha^{(1)}\vee M_\alpha^{(2)}}$, one may formally write
 \begin{align*}
  &\leftindex^e \int_{E^2} f(u_1,u_2) M_\alpha(du_1)M_\alpha(du_2)\\ \EqD & 2^{-2/\alpha} \leftindex^e \int_{E^2} f(u_1,u_2) \pp{M_\alpha^{(1)}(du_1)\vee  M_\alpha^{(1)}(du_1)}  \pp{M_\alpha^{(2)}(du_2)\vee  M_\alpha^{(2)}(du_2)} \\
  \ge  & 2^{-2/\alpha} \leftindex^e \int_{E^2} f(u_1,u_2) M_\alpha^{(1)}(du_1)  M_\alpha^{(2)}(du_2)  \text{ a.s..}
 \end{align*}
 In the actual proof, we  work with the decoupled LePage representation  \eqref{S_de} below. After establishing this stochastic order relation,  we then show that $f \in L_{+}^\alpha\left(\mu^k\right)$ is  a necessary condition for integrability for the decoupled version.  In particular, Theorem \ref{thm:nec cond} follows from Proposition \ref{prop 4.6} and Lemma \ref{lem 3.17} below.
 
 % By   \eqref{eq class} and Proposition \ref{01law},  it suffices to study necessary conditions for $S_k^e(f) < \infty$ a.s.
 Following the LePage representation notation as in Definition \ref{defn l}, let $\left(T_i^{(\ell)}\right)_{i\in \bb{Z}_+}$, $ \ell =1,\ldots, k$, be i.i.d.\ copies of $\left(T_i\right)_{i\in \bb{Z}_+}$, and let $ \pp{{\Gamma}_i^{(\ell)}}_{i\in \bb{Z}_+}$, $ \ell=1, \ldots, k$, be i.i.d.\ copies of $\pp{\Gamma_i}_{i\in \bb{Z}_+}$, and suppose the two collections of random variables are independent of each other.
For a function $f \in \cl{L}_k$, we set
\begin{equation}\label{S_de}
    S_k^{e,\text{de}}(f)  = \bigvee_{\boldsymbol{i} \in \bb{Z}_+^k}  f \left(T_{i_1}^{(1)}, \ldots ,T_{i_k}^{(k)}\right) \left(\prod_{\ell=1}^k \psi\left(T_{i_\ell}^{(\ell)}\right)\right)^{1 / \alpha}\left(\prod_{\ell=1}^k \Gamma_{i_\ell}^{(\ell)}\right)^{-1 / \alpha}.
\end{equation}
% which may be understood as the LePage representation of a decoupled version of $I^e_k(f)$. 
Below for two nonnegative random variables  $X$ and $Y$ (possibly taking value $\infty$),   we write $X \leq_{s t} Y$ to denote  $\mathbb{P}(X>x) \leq \mathbb{P}(Y>x)$ for all $x \in[0, \infty)$.  
The next result establishes a stochastic order relation between $I_k^e(f)$ and its decoupled version.
\begin{prop}\label{prop 4.6}
    Suppose  a function $f \in \cl{L}_k$, $k \geq 2$. Then we have $k^{-k/\alpha} S^{e,\text{de}}_k(f) \leq _{st}I^e_k(f)$.
\end{prop}
\begin{proof}
% \sy{still revising}
 Set $M_{\alpha,\ell}  (\cdot)= \bigvee_{i \geq 1} \mathbf{1}_{\left\{T_i^{(\ell)} \in \cdot\right\}} \psi\left(T_i^{(\ell)}\right)^{1 / \alpha} \left(\Gamma_i^{(\ell)}\right)^{-1 / \alpha}$, $\ell=1 \ldots k$, which are i.i.d.\ copies of $M_\alpha^L$ in Definition \ref{defn l}.
 Set $\widehat{M}_\alpha=k^{-1 / \alpha} \bigvee_{1 \leq \ell \leq k} M_{\alpha,\ell}$ and let $\widehat{I}_k^e(f)$ be the multiple extremal integral with respect to $\widehat{M}_\alpha$.  By max-stability, one has $\widehat{M}_\alpha \stackrel{d}{=} M_\alpha^L \stackrel{d}{=} M_\alpha$, and thus $\widehat{I}_k^e(f) \stackrel{d}{=} I_k^e(f)$. Next, 
 observe that a.s. we have
\begin{equation}\label{r4}
\begin{aligned}
   & k^{-k/\alpha} S^{e,\text{de}}_k(f) =   k^{-k/\alpha} \bigvee_{ \boldsymbol{i} \in \bb{Z}_+^k} f\left(T_{i_1}^{(1)},\ldots, T_{i_k}^{(k)} \right) \left( \prod_{\ell=1}^k\psi\left(T_{i_\ell}^{(\ell)}\right)\right)^{1 / \alpha} \left(\prod_{\ell=1}^k \Gamma_{i_\ell}^{(\ell)}\right)^{-1 / \alpha}\notag\\
   \leq   & k^{-k/\alpha} \bigvee_{\substack{ \boldsymbol{i}\in \bb{Z}_+^k,\,
   \boldsymbol{\ell}\in \{1,\ldots,k\}^k,\\(i_d,\ell_d) \neq (i_{d^\prime}, \ell_{d^\prime})  \\ \text{ for any } d \neq d^\prime,\\ d,d^\prime \in \{1,\ldots,k\}}} f\left(T_{i_1}^{(\ell_1)},\ldots, T_{i_k}^{(\ell_k)} \right) \left( \prod_{d=1}^k\psi\left(T_{i_d}^{(\ell_d)}\right)\right)^{1 / \alpha} \left(\prod_{d=1}^k \Gamma_{i_d}^{(\ell_d)}\right)^{-1 / \alpha}=:\widehat{S}_\alpha^{[1: k]}(f).
\end{aligned}
 \end{equation}
  % We claim that $\widehat{I}_k^e(f) \EqD \widehat{S}^{[1:k]}_\alpha (f)$ a.s., from which the  proof is concluded.  Indeed,  for $f=1_{B}$ with $B= A_1\times\ldots\times A_k\in \cl{C}_k$, it can be verified that $\wh{I}_k^e(1_B)= \wh{M}_\alpha(A_1)\ldots \wh{M}_\alpha(A_k)$ is the same as $\widehat{S}^{[1:k]}_\alpha (1_B)$. Then the claim  follows from an approximation argument similar as that for Theorems \ref{claim 5.3}  %\ref{thm sf case} 
  % and Definition \ref{def:mult extr int} \sy{how do we do this exactly?}.
Let $\widehat{M}^{(k)}_\alpha$ be defined in terms of $\widehat{M}_\alpha$ as in (\ref{prod M_a}). 
Using arguments similar as in  \cite[]{samorodnitsky1991construction}, it can be shown that 
    $\left(\widehat{S}_\alpha^{[1: k]}\left(\mathbf{1}_A\right)\right)_{A \in \mathcal{E}^{(k)}} \stackrel{d}{=}\left(\widehat{M}_\alpha^{(k)}(A)\right)_{A \in \mathcal{E}^{(k)}}$.  See Proposition \ref{prop B.5} in the Appendix for the details. It then follows from the monotone simple function approximation in Definition \ref{def:mult extr int}   that $\widehat{I}_k^e(f) \stackrel{d}{=} \widehat{S}^{[1:k]}_\alpha (f)$. The proof is concluded combining the above.

\end{proof}

Next, we shall establish  a stochastic lower bound for $S^{e,\text{de}}_k(f)$ which relates to the necessary condition $f \in L_{+}^\alpha\left(\mu^k\right)$. 
We first prepare a preliminary result.
The following lemma concerns a stochastic order relation on multivariate $\alpha$-Fr\'echet distributions.   It essentially follows from \cite[Corollary 4.3]{corradini2024stochastic}, while we still include a short proof.
\begin{lem}\label{jc lem}
Suppose $(E,\cl{E},m)$ is a probability measure space, and
  $f_i:E\mapsto [0,\infty]$ is a measurable function satisfying $\int_E f_i^\alpha(\nu)  m(d\nu)< \infty$, $i  \in \bb{Z}_+$.  Let $\left( \Gamma_j\right)_{j \in \bb{Z}_+}$ be the  arrival times of a standard Poisson process on $[0, \infty)$, and let $\left(T_j\right)_{j \in \bb{Z}_+}$ be a sequence of i.i.d.\ random elements taking value in $E$ with  distribution $m$, independent of $\left( \Gamma_j\right)_{j \in \bb{Z}_+}$.  Then for any $x_i>0$,  $i\in \bb{Z}_+$, we have 
\begin{equation}\label{jc}
\begin{aligned}
   \mathbb{P}\left( \bigvee_{ j \in \mathbb{Z}_+}f_i(T_j) \Gamma_j^{-1/\alpha} \leq x_i, \text{ for } i \in \bb{Z}_+ \right)  \leq  \mathbb{P}\left( \left(\int_E f^\alpha_{i}(v)m(dv) \right)^{1/\alpha}Z \leq x_i , \text{ for } i \in \bb{Z}_+ \right),
\end{aligned}
\end{equation}
where $Z$ is an $\alpha$-Fr\'echet random variable with unit scale coefficient. 
\end{lem}
\begin{proof}
 Set $N=\sum_{j=1}^\infty\delta_{\left(T_j, \Gamma_j \right)}$, which is a Poisson point process on $E\times [0, \infty)$ with intensity measure $  m(dv)\times dx$. Introduce $g: E \mapsto [0,\infty),   g(v)=  \bigvee_{i\in \bb{Z}_+} \frac{f_i^\alpha\left(v\right)}{x_i^\alpha}$, and let $G=\{(v,x)\in E\times [0,\infty) \mid   g(v) > x \}$.
The left hand side of (\ref{jc}) is
\begin{equation} \label{suff single}
\begin{aligned}
     & \mathbb{P}\left( \bigvee_{i \in \mathbb{Z}_+} \frac{f_i^\alpha\left(T_j\right)}{x_i^\alpha} \leq \Gamma_j \text{ for all } j \in \mathbb{Z}_+\right)  \\
     & =\mathbb{P}\left( N\pp{ G}=0 \right)  =    \exp\left\{ - \int_E \int_0^{g(v)} dx\, m(dv)\right\} \\
    = & \exp\left\{ - \int_E  \bigvee_{i \in \mathbb{Z}_+}\frac{f^\alpha_{i}(v)}{x_i^\alpha}   m(dv) \right\} 
     \le    \exp\left\{ - \bigvee_{i \in \mathbb{Z}_+}\int_E \frac{f^\alpha_{i}(v)}{x_i^\alpha}   m(dv) \right\},
\end{aligned}
\end{equation}
% The right hand side of (\ref{vf}) 
% is less than \suggestion{and}{or} equal to 
% $\exp\left\{ - \max_{ 1\leq i \leq r}\int_E \frac{f^\alpha_{i}(v)}{x_i^\alpha}   m(dv) \right\},$
where the last expression is precisely the right hand side of (\ref{jc}).
\end{proof}
  % To study a necessary condition for $I_k^e(f)<\infty$ a.s. is $f \in L_{+}^\alpha\left(\mu^k\right)$ for any $k \geq 2$, we rely on the following two results Lemma \ref{lem 3.17} by comparing the stochastic order of integrals.

\begin{lem}\label{lem 3.17}
Suppose $f \in \cl{L}_k$. Then,
\begin{align}\label{122}
 \int_{E^k} f^\alpha(\boldsymbol{u}) \mu(d \boldsymbol{u}) \pp{ \prod_{i=1}^k Z_i} \leq_{s t} S_k^{\text{e, de}}(f) ,
 \end{align}
 where $Z_i$, $i = 1, \ldots,k$, are i.i.d.\ $\alpha$-Fr\'{e}chet random variables with unit scale coefficient. 
\end{lem}
\begin{proof}
We may assume without loss of generality that $\mu$ is a probability measure. 
If $\mu$ is not a probability measure, then we replace $f$ by 
$f \cdot (\psi^{\otimes k})^{1/\alpha}$ and take $m$ as the underlying probability 
measure, where $\psi = d\mu/dm \in (0,\infty)$ $m$-a.e.. 
The desired relation \eqref{122} then follows by applying the result proved in the 
probability measure case. In addition,  for notational simplicity, we only treat $k=2$, and the argument easily  extends to the case where $k >2$.    
% Suppose $x >0$,  sequences of real numbers $\pp{ t_{i}}_{i \in \bb{Z}_+}$ , $\pp{ \gamma_{i}}_{i \in \bb{Z}_+ }$, a consequence  Lemma \ref{jc lem} is
% \begin{align*}
%      &\mathbb{P}\left( \bigvee_{1 \leq j< \infty}f( T_{j}^{(1)},t_{i}) \left(\Gamma_j^{(1)}\right)^{-1/\alpha} 
%      \leq x \gamma_{i}^{1/\alpha} , \text{ for } 1 \leq i \leq  M \right) \\
%      & \leq  \mathbb{P}\left( \left(\int_E f^\alpha( v,t_i)\mu(dv) \right)^{1/\alpha}Z_1 \leq x \gamma_{i}^{1/\alpha}, \text{ for } 1 \leq i \leq  M\right).
% \end{align*}

Suppose in addition that $Z_i$'s are  independent of everything else.
Since  $\left(T_j^{(1)}\right)_{j\in \bb{Z}_+}$  and  $\left(\Gamma_j^{(1)}\right)_{j\in \bb{Z}_+}$ are independent of $\left(T_j^{(2)}\right)_{j\in \bb{Z}_+}$ and $\left(\Gamma_j^{(2)}\right)_{j\in \bb{Z}_+}$,  conditioning on the latter two and applying  Lemma \ref{jc lem}, we have  for any $x>0$, 
%and $\tcp{r}\ge 2$ that
 \begin{equation*}\label{compare1}
 \begin{aligned}
  &\mathbb{P}\left( \bigvee_{i \in \mathbb{Z}_+}\bigvee_{j \in \mathbb{Z}_+}f(T_{j}^{(1)},T_{i}^{(2)}) \left(\Gamma_j^{(1)}\right)^{-1/\alpha}  \left(\Gamma_{i}^{(2)}\right)^{-1/\alpha} 
     \leq x \right) \\
     \leq  &  \mathbb{P}\left( \bigvee_{i \in \mathbb{Z}_+}\left(\int_E f^\alpha(v, T_i^{(2)})\mu(dv) \right)^{1/\alpha} \left(\Gamma_{i}^{(2)}\right)^{-1/\alpha}  Z_1 \leq x   \right).
\end{aligned}
 \end{equation*}
% The relation above continues to hold with $\tcp{r}$ replaced by $\infty$ through continuity of probability measure. 
Applying  Lemma \ref{jc lem}   again  conditioning on $Z_1$,
%\cite[Lemma 4.11]{kallenberg2021foundations}, 
the last expression is further bounded by
 \begin{equation*} 
 \begin{aligned}
  % &  \mathbb{P}\left( \bigvee_{1\leq i<\infty}\left(\int_E f^\alpha(v, T_i^{(2)})\mu(dv) \right)^{1/\alpha}\left(\Gamma_{i}^{(2)}\right)^{-1/\alpha} Z_1 \leq x \right)\\
  %& \leq   
  \mathbb{P}\left(\int_{E^2} f^\alpha(v, w) \mu(d v)\mu(d w) \, Z_1Z_2 \leq x \right).
\end{aligned}
 \end{equation*}
 % Further, letting  $M \rightarrow \infty$ on both sides of (\ref{compare1}), then combining the resulting expression with (\ref{compare2}), the conclusion follows. 
\end{proof}

\begin{rem}
We have established a sufficient condition for the integrability of multiple extremal integrals in 
Theorem~\ref{suff thm}, and a necessary condition in Theorem~\ref{thm:nec cond}. 
For $k\ge 2$, the necessary condition is in general not sufficient (see Example~\ref{S1} below). 
However, it remains open whether the sufficient condition is also necessary for $k\ge 2$.
\end{rem}

\subsection{A discussion on the moment-type conditions}\label{counter exam}

The previous sections involve several moment-type conditions. 
These include the conditions in \eqref{I2} and \eqref{suff 2} in Theorem~\ref{suff thm}, 
which serve as sufficient conditions for the integrability of multiple extremal integrals, 
as well as the assumption \( f \in L_{+}^\alpha(\mu^k) \) in Theorem~\ref{thm:nec cond}, 
which provides a necessary condition for integrability.
Following the notation of \cite[]{samorodnitsky1989asymptotic}, we introduce additionally 
\begin{equation}\label{defn loglog}
    L^\alpha \ln L \ln \ln L(f,\mu) 
    := \int_{E^2} f^\alpha(\boldsymbol{u})\,\bigl(1 + (\ln_+ f(\boldsymbol{u})) \cdot \ln_+|\ln f(\boldsymbol{u})|\bigr)\, \mu^k(d\boldsymbol{u}), \quad f \in \mathcal{L}_2.
\end{equation}
This quantity will be used to formulate an additional sufficient condition for integrability 
(see Proposition~\ref{prop k=2} below) for $k=2$, 
and it will play a key role from  Section~\ref{sec:tail}  onward. Since multiple moment-type conditions are involved in the case $k=2$, we summarize the relations among  them as below.
\begin{prop}\label{prop k=2}
Suppose $f\in \cl{L}_2$ and $\alpha\in (0,\infty)$.  Consider the following 4 statements:
\begin{enumerate} 
    \item For some probability measure $m$ on $E$ equivalent to $\mu$ with $\psi=d \mu / d m \in(0, \infty)$ $m$-a.e., we have $ L^\alpha \ln L \ln \ln L(f \cdot\left(\psi^{\otimes 2}\right)^{1 / \alpha}, m) < \infty$,  with the   notation defined   in \eqref{defn loglog}.
    \item  For some probability measure $m$ on $E$ equivalent to $\mu$ with  $\psi=d\mu/dm\in (0,\infty)$ $m$-a.e.,  we have $
    L^\alpha \ln  L(f\cdot(\psi^{\otimes 2})^{1/\alpha}, m) < \infty$, with the  notation defined in \eqref{defn l space}.
    \item  We have condition \eqref{I2} holds.
    \item  We have \( f \in L_{+}^\alpha(\mu^2) \).
\end{enumerate}
Then we have
$$\begin{aligned}
\text { (i) } \Rightarrow \text { (ii) } \Longleftrightarrow \text { (iii) } \Rightarrow \text { (iv). }
\end{aligned} $$
\end{prop}

\begin{proof}

 For (i) $\Rightarrow$ (ii), assume without loss of generality that $\mu$ is a probability measure and $\psi\equiv 1$.
We have 
$$L^\alpha \ln L \ln \ln L(f \mathbf{1}_{\{0\le f <e\}}, \mu) = \int_{\{ 0 \le f < e\}} f^\alpha d\mu^2,$$ 
while $$ L^\alpha \ln L (f \mathbf{1}_{\{0\le f <e\}}, \mu) = \int_{\{ 0 \le f < e\}} f^\alpha (1+ \ln_+( f\mathbf{1}_{\{1\le f <e\}}))d\mu^2\leq 2\int_{\{ 0 \le f < e\}} f^\alpha d\mu^2.$$
Furthermore, combining the above with the observation that  $$L^\alpha \ln L \ln \ln L(f \mathbf{1}_{\{f \geq e\}}, \mu) \geq L^\alpha \ln L (f \mathbf{1}_{\{f \geq e\}}, \mu),$$ the desired implication follows.

For (ii) $\Rightarrow$ (iii), first note that either of the conditions (ii) and (iii) holds if and only if it remains valid when $f$ is replaced by $f^{r}$ and $\alpha$ by $\alpha/r$ for any $r>0$. Hence, without loss of generality, we may assume $\alpha\in(0,2)$. Also, the double integral in \eqref{I2} is unchanged if $f$ is replaced by $f\cdot(\psi^{\otimes 2})^{1/\alpha}$ and $\mu$ by $m$, where $\psi$ and $m$ are as in (ii); thus we may further assume $\mu$ is a probability measure and $\psi\equiv 1$. 
Next, recall from the proof of Theorem~\ref{suff thm} that condition (ii) is sufficient for the existence of  a double symmetric $\alpha$-stable integral of $f$ (cf. \cite[]{breton2002multiple}) (one may follow the same arguments to reduce further to the case $(E,\mu)=([0,1],\lambda)$). On the other hand, condition \eqref{I2} is necessary and sufficient for the existence of such an integral (see \cite[]{rosifiski1986ito,kwapien1987double}). The conclusion follows.
 
For (iii) $\Rightarrow$ (ii), by Lemma \ref{lem eqv} in the Appendix, it suffices to assume that $f$ is symmetric. We may also assume $\mu^2(f>0)>0$. Let $\|f\|_\alpha^{\alpha}=\int_{E^2} f^\alpha(x, y) \mu(d x) \mu(d y) \in (0,\infty)$.
Define $$\psi^{-1}(x)=\frac{1}{\|f\|_\alpha^{\alpha}}\int_E f^\alpha(x, y) \mu(d y)$$ if $\int_E f^\alpha(x, y) \mu(d y) > 0$ for $\mu$-a.e. $x \in E$; otherwise, set
$$
\psi^{-1}(x)=\frac{1}{2\|f\|_\alpha^{\alpha}} \int_E f^\alpha(x, y) \mu(d y)\mathbf{1}_{\{\int_E f^\alpha(x, y) \mu(d y) >0\}} + \kappa_0(x)\mathbf{1}_{\{\int_E f^\alpha(x, y) \mu(d y) =0\}}, \quad x \in E,
$$
where  $\kappa_0$ is a positive measurable function chosen so that $\int_E \kappa_0(x)\mathbf{1}_{\{\int_E f^\alpha(x, y) \mu(d y) =0\}} \mu (dx)=\frac{1}{2}$. Such a choice is possible since $E$ is $\sigma$-finite. Note that if $x\in E$ satisfies $\int_E f^\alpha(x, y) \mu(d y)=0$, then  
$f(x,y) = 0$ for $\mu$-almost every $y \in E$. Thus in either of the cases above, $L^\alpha \ln  L(f\cdot(\psi^{\otimes 2})^{1/\alpha}, m)$, where $dm=\psi^{-1}d\mu$, equals 
 $$
\int_{E^2} f(s, t)^\alpha\left(1+\ln _{+}\left(\frac{ Cf(s, t)}{\left(\int_E f(s, u)^\alpha \mu(d u)\right)^{1 / \alpha}\left(\int_E f(u, t)^\alpha \mu(d u)\right)^{1 / \alpha}}\right)\right) \mu(d s) \mu(d t),
$$
for some positive constant $C$. The last integral is finite due to condition \eqref{I2} and the inequality $\ln_+(ab) \leq \ln_+ a + \ln_+ b$ for $a, b>0$.

 The implication  (iii) $\Rightarrow$ (iv) is obvious.
\end{proof}

% The proof is deferred to the appendix ***.

% \begin{prop}\label{prop k=2}
% Fix $f \in \mathcal{L}_2$ and some probability measure $m$ on $E$ equivalent to $\mu$ with $\psi=d \mu / d m \in(0, \infty)$.
% % \begin{enumerate}
% %     %\item The condition $L^\alpha \ln L \ln \ln L(f \cdot\left(\psi^{\otimes 2}\right)^{1 / \alpha}, m) < \infty$ is stronger than condition~\eqref{I2}.
% %      \item The condition $ L^\alpha \ln L \ln \ln L(f \cdot\left(\psi^{\otimes 2}\right)^{1 / \alpha}, m) < \infty$  implies the condition \eqref{suff 2} when setting $k = 2$, that is, $L^\alpha \ln L(f \cdot\left(\psi^{\otimes 2}\right)^{1 / \alpha}, m) < \infty$.
% %     \item The condition $L^\alpha \ln L(f \cdot\left(\psi^{\otimes 2}\right)^{1 / \alpha}, m) < \infty$  implies condition~\eqref{I2}.
% % \end{enumerate}
% The following relations hold
% \begin{align*}
%    L^\alpha \ln L \ln \ln L(f \cdot\left(\psi^{\otimes 2}\right)^{1 / \alpha}, m) < \infty  \Rightarrow  L^\alpha \ln L(f \cdot\left(\psi^{\otimes 2}\right)^{1 / \alpha}, m) < \infty
% \end{align*}

% \end{prop}

Next, we provide an example that is similar to the one in \cite[]{samorodnitsky1991construction}, demonstrating that the condition 
\( f \in L_{+}^\alpha(\mu^k) \), \( k \geq 2 \), 
which is necessary for the integrability by Theorem~\ref{thm:nec cond}, 
is not sufficient.
 In particular, we show that  it is possible to have   $A \in \mathcal{E}^{(k)}$  with $\mu^k(A)<\infty$, i.e., $\mathbf{1}_A \in  L_{+}^\alpha(\mu^k)$, yet $I_k^e(\mathbf{1}_A)=M_\alpha^{(k)}(A)=\infty$ a.s.. 
 % This example also shows that $$ a.s., while $\mu^{(k)}(A)<\infty$ for $k \geq 2$.
\begin{exam}\label{S1}
Fix $k \geq 2$. Suppose  $M_\alpha$ is defined on $\mathbb{R}$ with the Lebesgue control measure $\lambda$. Let $\pp{a_i}_{i \in \bb{Z}_+}$ be a sequence  satisfying
  \begin{enumerate}[label=(\roman*)]
        \item  $\sum_{i=1}^{\infty}a_i^k<\infty$ and  $0< a_i <1$ for each $i\in \mathbb{Z}_+$;
        \item $\sum_{i=1}^{\infty}a_i^k |\ln a_i|^{k-1}=\infty $.
    \end{enumerate}
     Define a sequence of disjoint off-diagonal cubes  
     \begin{equation}\label{Ai}
      A_i = \left[k i, k i+a_i\right] \times\left[k i+1, k i+1+a_i\right] \times \cdots \times\left[k i+k-1, k i+k-1+a_i\right], i\in \bb{Z}_+,   
     \end{equation}
    and set $A = \bigcup_{i \in \mathbb{Z}_+} A_i$.  
     Condition (i) ensures that $\lambda^k(A)<\infty$, whereas (i)-(ii) imply $\sum_i a_i=\infty$; hence, each coordinate projection of $A$ has infinite $\mu$-measure, so $A$ is not contained in any rectangle $B_1 \times B_2 \times \ldots \times B_k$ with $B_j \in \mathcal{E}$,  $\mu\left(B_j\right)<\infty, j=1,2\ldots, k$.
     
Let $\pp{M_i}_{i\in \bb{Z}_+}$ be an i.i.d.\ sequence of standard  $\alpha$-Fr\'echet random variables. By the scaling and independently scattered  properties of  $M_\alpha$, we obtain 
\begin{equation}\label{eq 28}
    M_\alpha^{(k)}(A) \stackrel{d}{=} \bigvee_{i \in \mathbb{Z}_{+}} a_i^{k/\alpha} M_{k i} M_{k i+1}\ldots M_{ki+k-1}.
\end{equation}
Recall that for $\delta_i \geq 0, i \in \mathbb{Z}_+$,  $\bigvee_{i=1}^\infty \delta_i M_i < \infty$ if and only if $\sum_{i=1}^\infty \delta_i^\alpha < \infty$; see, e.g., \cite[(4.1)]{stoev2005extremal}.  By this fact and conditioning on $\pp{ M_{ki}, M_{k i+1},\ldots, M_{ki+k-2}}_{i\in \bb{Z}_+} $,  we have
$$
\bigvee_{i=1}^{\infty} a_i^{k / \alpha} M_{k i} M_{k i+1}\ldots M_{ki+k-1}=\infty \text { a.s. if and only if } \sum_{i=1}^{\infty} a_i^k M_{k i}^\alpha\ldots M_{ki+k-2}^\alpha=\infty \text { a.s.. }
$$
The a.s.\ divergence of the last sum follows from Kolmogorov's three-series theorem, since condition (ii) above implies  $\sum_{i=1}^\infty\bb{E}\pp{a_i^k M_{k i}^\alpha \ldots M_{k i+k-2}^\alpha \mbf{1}_{\{a_i^k M_{k i}^\alpha \ldots M_{k i+k-2}^\alpha\le 1\}}}=\infty$. Indeed, this follows from the elementary calculation that $M_{k i}^\alpha, \ldots, M_{k i+k-2}^\alpha,$ $i \in \mathbb{Z}_+$, are standard i.i.d.\ $1$-Fr\'echet variables, and hence 
$$\mathbb{E}\pp{ M_{k i}^\alpha \ldots M_{k i+k-2}^\alpha  \mbf{1}_{\{M_{k i}^\alpha \ldots M_{k i+k-2}^\alpha \le x\}}}\sim \frac{(\ln x)^{k-1}}{(k-1)!} \text{ as } x\rightarrow\infty.$$

\end{exam}

Interestingly, combining Proposition \ref{prop k=2}  for the case $k=2$ and the following result, when the integrand is a product of   univariate functions, all the  aforementioned moment-type conditions and $f \in L_+^\alpha(\mu^k), k \geq 2,$ are equivalent.
\begin{prop}\label{Pro:tensor int cond}
    Suppose $f$ is a separable non-negative function, i.e. $$f(x_1, x_2\ldots, x_k) = \varphi_1(x_1) \ldots\varphi_k(x_k),$$
for some univariate measurable functions $\varphi_i:E\mapsto [0,\infty], 1 \leq i \leq k$. Then, there exists a probability measure $m$ on $E$ equivalent to $\mu$ with density $\psi=d \mu / d m \in(0, \infty)$ such that  the following equivalences hold:
\begin{enumerate}
\item When $k = 2$, 
$$
f \in L_{+}^\alpha\left(\mu^2\right) \quad \Longleftrightarrow \quad L^\alpha \ln L \ln \ln L\left(f \cdot\left(\psi^{\otimes 2}\right)^{1 / \alpha}, m\right)<\infty.
$$
\item When $k >2$,  
$$
f \in L_{+}^\alpha\left(\mu^k\right) \quad \Longleftrightarrow \quad L^\alpha \ln ^{k-1} L\left(f \cdot\left(\psi^{\otimes k}\right)^{1 / \alpha}, m\right)<\infty.
$$
\end{enumerate}
Hence,  $f \in L_+^\alpha(\mu^k)$  is a sufficient and necessary condition for $I_k^e(f) < \infty$ a.s..
\end{prop}
\begin{proof}Suppose $f \in L_+^\alpha(\mu^k)$, $k \geq 2$. We aim to show that this implies $L^\alpha \ln L \ln \ln L\left(f \cdot\left(\psi^{\otimes 2}\right)^{1 / \alpha}, m\right)<\infty$ when $k=2$, and implies $L^\alpha \ln ^{k-1} L\left(f \cdot\left(\psi^{\otimes k}\right)^{1 / \alpha}, m\right)<\infty$ when $k >2$, for some $m, \psi$.  Assume for simplicity $\alpha = 1$, and the more general case is similar.

Let $A:=\left\{s \in E\mid \sum_{i=1}^k \varphi_i(s)=0\right\}$.
Define
$$
\psi(s)= \begin{cases}\kappa(s)^{-1}, & s \in A, \\ \frac{C}{\sum_{i=1}^k \varphi_i(s)}, & s \notin A, \end{cases}
$$
where $C>0$ is a constant and $\kappa$ is a positive measurable function to be chosen so that $\int_E \psi^{-1} d \mu=1$. This is possible since $E$ is $\sigma$-finite and  $\int_{E} \varphi_i(s)\mu (ds) <\infty$, $1\le i\le k$.

% Assume $\mu(A) > 0$ (and possibly $\mu(A) = \infty$); otherwise, let $\kappa = 1$, $C = \int_E \sum_{i=1}^k \varphi_i(s) \mu(d s)<\infty$. By  $\sigma$-finiteness of $\mu$ on $A$, there exist disjoint measurable sets $B_n \subseteq A, n \in \mathbb{Z}_+$ with $0<\mu\left(B_n\right)<\infty$, satisfying $A=\bigcup_n B_n$. Define
% $\kappa=\sum_{n \geq 1} 2^{-(n+1)}\mu\left(B_n\right)^{-1} \mathbf{1}_{B_n}$. Then $\int_A \kappa d \mu=\sum_n 2^{-(n+1)}=1 / 2$. Meanwhile, let $C = 2 \int_E \sum_{i=1}^k \varphi_i(s) \mu(ds) < \infty$.  Combining the two parts of $\psi$ gives
% $\int_E \psi^{-1} d \mu=1$. 
Under such a choice, $f\cdot \psi^{\otimes k} $ is bounded by a constant $\mu$-a.e. (since $f$ vanishes on $A$). Thus the desired implication follows.  The converse directions are straightforward.

The conclusion that $f \in L_{+}^\alpha\left(\mu^k\right)$ is a sufficient and necessary condition for $I_k^e(f)<\infty$ a.s.\ follows from the equivalence relations in (i)–(ii), while a direct proof is also possible. Indeed, note that $f\in L^\alpha_+(\mu^k)$ if and only if $\varphi_i \in L_+^\alpha(\mu), 1\leq  i \leq k $, and thus $I^e_1(\varphi_i) < \infty$ a.s., or equivalently  $S_1^e(\varphi_i) < \infty$ a.s.\ for $1 \leq i \leq k$ (see \eqref{eq prep 2.4}). Observe that $S_k^e(f) < \prod_{i=1}^k S_1^e(\varphi_i)$ a.s.\ due to the absence of the diagonal terms in  $S_k^e(f)$. Then, the conclusion follows from Theorem \ref{thm:nec cond} and \eqref{eq4}. 
\end{proof}

%\begin{prop}\label{cor tensor sn}
 %   Suppose $f$ is a tensor product function in the form $f(x_1,\cdots,x_k) = \phi(x_1)\otimes \cdots\otimes \phi(x_k)$, $k \geq 2$,  for some univariate nonnegative measurable function $\phi:E\mapsto [0,\infty]$. Then, $f \in L_+^\alpha(\mu^k)$  is sufficient and necessary condition for $I_k^e(f) < \infty$ a.s.
%\end{prop}
%\begin{proof}
%Note that $f\in L^\alpha_+(\mu^k)$ if and only if $\phi \in L_+^\alpha(\mu)$, and that \cite[Proposition 2.7]{stoev2005extremal} implies that $I^e_1(\phi) < \infty$ a.s., or equivalently  $S_1^e(\phi) < \infty$ a.s.. Observe also that $S_k^e(f) < (S_1^e(\phi))^k$ a.s.\ due to the absence of the diagonal terms in  $S_k^e(f)$. Then, the conclusion follows from Theorem \ref{thm:nec cond} and Corollary \ref{cor:gen int}. 
%\end{proof}

\section{Convergence of multiple extremal integrals}\label{sec:limit}
 
In \cite[]{stoev2005extremal}, it has been shown that for $k=1$ and $f,f_n\in  L_+^\alpha(\mu)$, the convergence $I^e_1(f_n)\xrightarrow{\mathbb{P}} I^e_1(f)$ as $n\rightarrow\infty$ holds if and only if $\int |f_n-f|^\alpha d\mu\rightarrow 0$ as $n\rightarrow\infty$. Below for the case of $k\ge 2$, we provide a sufficient condition for the convergence $I^e_k(f_n)\xrightarrow{\mathbb{P}} I^e_k(f)$ as $n\rightarrow\infty$. %Recall   $L^\alpha \ln L \ln \ln L(f,\mu)$ as defined in \eqref{defn loglog} for $f \in \mathcal{L}_2$.

\begin{prop}\label{prop 5.44}
 Suppose  $f, f_n \in \cl{L}_k$ and $I_k^e(f_n)<\infty$ a.s.\ for $n\in\bb{Z}_+$ and $k \geq 2$. 
 If  there exists a probability measure $m$ on $E$ equivalent to $\mu$ with  $\psi=d\mu/dm\in (0,\infty)$ $m$-a.e., and as $n\rightarrow\infty$,
\begin{align*}
L^\alpha \ln ^{k-1} L(|f_n-f|\psi^{\otimes k}, m)\rightarrow 0, 
\end{align*}
 then $I_k(f)<\infty$ a.s., and $I_k^e(f_n)  \xrightarrow{\mathbb{P}} I_k^e(f)$ as $n \rightarrow \infty$.
\end{prop}
\begin{proof}
In view of the triangle inequality and monotonicity items in Proposition \ref{I^e_k up} for general measurable functions  stated in Corollary \ref{cor:gen int}, we   have $$I_k^e(f_n)-I_k^e(f)\le I_k^e((f_n-f)_+)\le I_k^e(|f_n-f|) \text{  a.s.. }$$ Combining this inequality  with an additional one obtained by switching the roles between $f_n$ and $f$, we conclude 
    \begin{equation}\label{eq:diff tri ineq}
    |I_k^e(f_n)-I_k^e(f)|\le I_k^e(|f_n-f|)
     \text{ a.s..}
    \end{equation}
    By the assumption and the sufficient condition for integrability from Theorem \ref{suff thm}, we know that $I_k^e(|f_n-f|)<\infty$ a.s.\ for sufficiently large $n$, and hence $I_k^e(f)<\infty$  a.s.\ by \eqref{eq:diff tri ineq}.   In view of \eqref{eq:diff tri ineq}, it suffices to assume $f \equiv 0$ in the remainder of the proof. Similarly as the proof of Theorem \ref{suff thm},  
without loss of generality, we may take $\mu = \lambda$, $ \psi\equiv 1$ and $E = [0,1]$, with $T_j=U_j, j \in \mathbb{Z}_{+}$, where $\left(U_j\right)_{j \in \mathbb{Z}_{+}}$, are i.i.d.\ uniform random variables on $[0,1]$. Moreover, by Lemma~\ref{lem eqv}, each $f_n$ may  be assumed symmetric.  Thus, to conclude the proof, it suffices to show that as $n \rightarrow \infty$,  
$$
L^\alpha \ln ^{k-1} L\left(f_n, \lambda\right) \rightarrow 0 \quad \Rightarrow \quad I_k^e\left(f_n\right) \xrightarrow{\mathbb{P}} 0 .
$$
Fix  $r>\alpha$. In view of Corollary \ref{cor:gen int}, it suffices to work with the LePage representation $S_k^e(f_n)$. We have 
\begin{equation*}
\left(S_k^e(f_n)\right)^r=\bigvee_{\boldsymbol{j} \in \mathcal{D}_{k,<}} f_n\left(U_{\boldsymbol{j}}\right)^r\left[\Gamma_{\boldsymbol{j}}\right]^{-1 /(\alpha / r)} \leq  \sum_{\boldsymbol{j} \in \mathcal{D}_{k,<}} f_n\left(U_{\boldsymbol{j}}\right)^r\left[\Gamma_{\boldsymbol{j}}\right]^{-1 /(\alpha / r)}.
\end{equation*}
By \cite[Section 4.1.2]{breton2002multiple}, the last term converges to zero in probability if $L^{\alpha/r} \ln ^{k-1} L\left(f_n^r, \lambda\right) \rightarrow 0$ as $n \rightarrow \infty$, which is further equivalent to $L^\alpha \ln ^{k-1} L\left(f_n, \lambda\right) \rightarrow 0$.  

\end{proof}

Next, we state a dominated-convergence-type result for multiple extremal integrals. 
\begin{prop}\label{prop 4.3}
Suppose $k \geq 2$ and $f, g,  f_n \in \cl{L}_k$ for $n \in \mathbb{Z}_+$. Assume that there exists a probability measure $m$ equivalent to $\mu$ with $\psi\in d\mu/dm\in (0,\infty)$ $m$-a.e. such that $L^\alpha \ln ^{k-1} L(g \cdot\left(\psi^{\otimes k}\right)^{1 / \alpha}, m)<\infty$.  Moreover, suppose $f_n \rightarrow f$ and $f_n \leq g$, $\mu^k$-a.e. for all $n \in \mathbb{Z}_+$. Then as $n \rightarrow \infty$, 
 \begin{equation}\label{eq:I_k conv as}
     I_k^e(f_n) \rightarrow   I_k^e(f) \quad\text{ a.s..}
 \end{equation}
\end{prop}
\begin{proof}
In view of Corollary \ref{cor:gen int}, it suffices to prove an analogous relation with the LePage representation.
Without loss of generality, suppose $\mu$ is a probability measure, $\psi\equiv 1$ in the  LePage representation $S_k^e(f)$  in \eqref{eq4}, and $f$ is symmetric. Fix $r > \alpha$, we have 
\begin{equation}\label{eq 41}
   \pp{ S_k^e(f_n)}^r = \bigvee_{\boldsymbol{j} \in \mathcal{D}_{k,<}} f_n\left(T_{\boldsymbol{j}}\right)^{ r}\left[\Gamma_{\boldsymbol{j}}\right]^{-1 /(\alpha / r)} \leq \sum_{\boldsymbol{j} \in \mathcal{D}_{k,<}} g\left(T_{\boldsymbol{j}}\right)^{ r}\left[\Gamma_{\boldsymbol{j}}\right]^{-1 /(\alpha / r)} < \infty \text{ a.s.,}
\end{equation}
where the last relation   follows from a similar argument to (\ref{eq:S_k^r<inf}), and  the fact that $L^\alpha \ln ^{k-1} L\left(g, \mu\right)<\infty$    and $L^{2\alpha/r} \ln ^{k-1} L\left(g^{r/2} , \mu\right)<\infty$, $ r>\alpha$, are equivalent.
%by assumption (see the argument for (\ref{eq:S_k^r<inf})). 
Also, relation (\ref{eq 41}) holds with $f_n$ replaced by $f$ since $f \leq g$ $\mu^k$-a.e.. Observe that
\begin{equation}\label{leq 42}
\begin{aligned}
  S_k^e(|f-f_n|) &= \pp{\bigvee_{\boldsymbol{j} \in  \mathcal{D}_{k,<}^M} \left|f\left(T_{\boldsymbol{j}}\right)-f_n\left(T_{\boldsymbol{j}}\right)\right|\left[\Gamma_{\boldsymbol{j}}\right]^{-1 / \alpha}} \bigvee \pp{\bigvee_{ \mathcal{D}_{k,<} \setminus \mathcal{D}_{k,<}^M} \left|f\left(T_{\boldsymbol{j}}\right)-f_n\left(T_{\boldsymbol{j}}\right)\right|\left[\Gamma_{\boldsymbol{j}}\right]^{-1 / \alpha}},\\
 & \leq \pp{\bigvee_{\boldsymbol{j} \in  \mathcal{D}_{k,<}^M} \left|f\left(T_{\boldsymbol{j}}\right)-f_n\left(T_{\boldsymbol{j}}\right)\right|\left[\Gamma_{\boldsymbol{j}}\right]^{-1 / \alpha}}  + 2\pp{ \sum_{ \mathcal{D}_{k,<} \setminus \mathcal{D}_{k,<}^M} g\left(T_{\boldsymbol{j}}\right)^r\left[\Gamma_{\boldsymbol{j}}\right]^{-1 / (\alpha/r)}}^{1/r},
\end{aligned}
\end{equation}
where $\mathcal{D}_{k,<}^M=\left\{\left(t_1, \ldots, t_k\right) \in \mathbb{Z}_+^k \mid t_1<\ldots<t_k \leq M \right\}$ for $M >0$ and $n \in \mathbb{Z}_+$. By first taking $n \rightarrow \infty$ and then $M \rightarrow \infty$ in (\ref{leq 42}), the last expression in (\ref{leq 42}) tends to zero a.s.. Further, combining the above results with the inequality $\left| S_k^e(f) -  S_k^e(f_n)\right|\leq  S_k^e(|f-f_n|)$, we complete the proof. 
\end{proof}
\begin{rem}
The condition   $L^\alpha \ln ^{k-1} L(g \cdot\left(\psi^{\otimes k}\right)^{1 / \alpha}, m)<\infty$  in Proposition \ref{prop 4.3} can be weakened to requiring only that the last relation in (\ref{eq 41}) holds for some $r>\alpha$. The convergence in \eqref{eq:I_k conv as} also holds in $L^r(\mathbb{P})$ for any $r\in (0,\alpha)$. See Corollary \ref{Cor:I_k conv in L^r} below.
\end{rem}
\section{Tail behavior}\label{sec:tail}
In this section, we examine the joint tail behavior of the random vector $(I_k^e(f_1),\ldots I_k^e(f_d))$, $d\in \bb{Z}_+$, for suitable  integrands $f_1,\ldots,f_d$.  The results will be established for the subclass of integrable functions as described in Definition \ref{def:L_{k,+}} below. 

  For describing the joint tail behavior, recall that 
a $d$-dimensional random vector $\mathbf{X} = (X_1, X_2\ldots,X_d)$  taking values in the nonnegative quadrant $[0, \infty)^d$, $d\in \bb{Z}_+$, is said to be multivariate (or jointly) regularly varying, if there exists a function $\nu:(0,\infty)^d \mapsto (0,\infty)$,
such that  
$$\lim _{t \rightarrow \infty}\mathbb{P}\left[\boldsymbol{X} \in[\mathbf{0}, t\boldsymbol{x}]^c\right] / \mathbb{P}\left[\boldsymbol{X} \in[\mathbf{0}, t\mathbf{1}]^c\right]=\nu\left(\mbf{x}\right),\quad \mbf{x}\in (0,\infty)^d,$$
where we have used the notation $[\mbf{a},\mbf{b}]=[a_1,b_1]\times \cdots \times [a_d,b_d]$  for $\mbf{a}=(a_1,\ldots,a_d)$, $\mbf{b}=(b_1,\ldots,b_d)$, $a_i\le b_i$, $i\in d$, $\mbf{0}:=(0,\ldots,0)$ and $\mbf{1}:=(1,\ldots,1)$. Moreover, the function $\nu$ is necessarily homogeneous: there exists an index $\alpha>0$, such that $\nu(c\mbf{x})=c^{-\alpha}\nu(\mbf{x})$ for all $c>0$ and $\mbf{x}\in (0,\infty)^d$. The convergence above may be equivalently formulated in terms of vague convergence of measures.
For more details, see, e.g., \cite[Section 5.4.2]{resnick2008extreme}.
We denote such class of multivariate regularly varying $\mbf{X}$ with index $\alpha$ as $\mrm{MRV}_d(\alpha)$.  When $d=1$, the notion reduces to  univariate regular variation of the  {distributional} tail.

% For more details on the notion of multivariate regular variation, we refer the reader to \cite{haan2006extreme} and \cite{resnick2007heavy}. 

% Then, we  present below Corollary \ref{cor Lambda space} without proofs, as the proofs are similar to those for \cite[Corollary 1.2]{samorodnitsky1989asymptotic}, where $X\vee Y$ below is replaced by $X+Y$.
\begin{lem}\label{lem 1.1}
Let  $X$ and $Y$ be random variables.  Suppose that the right tail of $X$ dominates the right tail of   $Y$ in the sense that $\lim _{t \rightarrow \infty} \mathbb{P}(Y>t)/\mathbb{P}(X>t)=0$. Then, for any  $x>0$,
\begin{equation*}
  \lim_{t \rightarrow \infty}  \mathbb{P}\left( X\vee Y  >t \right) / \mathbb{P}\left(X>t\right) = 1.
\end{equation*}
\end{lem}
\begin{proof}
Note that  $\mathbb{P}\left(X\vee Y  >t\right)/ \mathbb{P}\left(X>t\right)$ is bounded below by 1 and bounded above by 
 \begin{align*}
 \frac{ \mathbb{P}\left( X>t \right) + \mathbb{P}\left( Y>t\right) }{\mathbb{P}\left(X>t\right)} \leq  1+ \frac{ \mathbb{P}\left(Y > t  \right)}{\mathbb{P}\left(X > t\right)},
    \end{align*}
    where the last expression tends to 1 as $t \rightarrow \infty$ by assumption. 
\end{proof}
For $k \in \bb{Z}_+$, let $\Lambda_{k}$ denote the set of all nonnegative random variables for which the limit  
\begin{equation}\label{lambda_k}
 \lambda_{k}(X):=\lim _{x \rightarrow \infty} x^\alpha(\ln x)^{-k} \mathbb{P}(X>x)
\end{equation}
exists. Clearly, a random variable $X\in \Lambda_k$ with  $\lambda_k(X)>0$ belongs to $\mrm{MRV}_1(\alpha)$. The following result is an immediate consequence of the previous lemma.
\begin{cor}\label{cor Lambda space}
 Suppose $X$ and $Y$ are as in Lemma \ref{lem 1.1}, and additionally,  $X \in \Lambda_k$. Then $X\vee Y \in \Lambda_k$ and $\lambda_k(X\vee Y)=\lambda_k(X)$.
\end{cor}
%Recall the class $\mathcal{L}_k$ in Corollary \ref{cor:gen int}. 
We are now ready to state the main result of this section. To do so, we first introduce the following class of integrands, which will play an important role in later sections as
well. Also, recall the class $\mathcal{L}_k$ in \eqref{eq:cl L_k}. 
\begin{defn}\label{def:L_{k,+}}
On a product measure space $(E^k,\cl{E}^k,\mu^k)$,
we denote by  $\mathcal{L}_{k,+}^\alpha (\mu)$ the class of functions in $ \cl{L}_k$ that satisfy 
\begin{equation}
\begin{aligned}
 L^\alpha \ln ^{k-1} L(f,\mu) < \infty, \quad\quad &\text{ if } k>2, \\
 % L^\alpha \ln L \ln \ln L(f, \mu) 
L^\alpha \ln L \ln \ln L(f, \mu) < \infty, \quad\quad & \text{ if } k=2 \text {, }
\end{aligned}
\end{equation}
where $L^\alpha \ln ^{k-1} L(f, \mu)$ and $L^\alpha \ln L \ln \ln L(f, \mu) $ are as in (\ref{defn l space})  and (\ref{defn loglog}), respectively.
\end{defn}

\begin{thm}\label{lambda_{k-1}}
Suppose $f_i \in \mathcal{L}_k$,  $i = 1,\ldots,d$, with each $\mu^k(f_i>0)>0$, $k\ge 2$. 
Assume that there exists a probability measure $m$ equivalent to $\mu$ with $\psi =  d\mu/dm\in (0,\infty)$ $m$-a.e.,  so that $f_i\cdot\left(\psi^{\otimes k}\right)^{1 / \alpha} \in \mathcal{L}_{k,+}^\alpha(m)$  for $1\leq i \leq d$, where $d \in \bb{Z}_+$, $k \geq 2$, and $\mathcal{L}_{k,+}^\alpha(m)$ is as in Definition \ref{def:L_{k,+}}. Then, the random vector $(I_k^e(f_1),\ldots I_k^e(f_d))\in \mrm{MRV}_d(\alpha)$, and more specifically, for $\boldsymbol{x} = (x_1, \ldots, x_d)\in (0,\infty)^d$,
\begin{equation}\label{eq 42}
    \begin{aligned}
     &\lim_{t \rightarrow \infty} t^{\alpha} (\ln t)^{-(k-1)} \mathbb{P}\left( \left(  I^e_k(f_i)\right)_{i = 1,\ldots,d} \in [\boldsymbol{0},t\boldsymbol{x}]^c\right)  =  k \alpha^{k-1}(k !)^{-2}  \int_{E^k} \bigvee_{i=1}^d \frac{\widetilde{f}_i^\alpha\left(\boldsymbol{u}\right)}{x_i^\alpha} \mu^k(d\boldsymbol{u}),
 \end{aligned}
\end{equation}
where $\wt{f}_i$, $1\leq i \leq d$, are the max-symmetrization as in \eqref{sym def}.
\end{thm}

\begin{rem}\label{Rem:k=1 tail}
We note that \eqref{eq 42} holds in the case $k=1$ under the assumption that each 
$f_i \in L_+^\alpha(\mu)$, $i=1,\ldots,d$, i.e., the necessary and sufficient 
integrability condition for single integrals. This follows from standard facts on 
max-stable random vectors and their spectral (LePage-type) representations; 
see, e.g., \cite[Theorem~13.2.2]{kulik2020heavy}.
\end{rem}

\begin{proof}
 Without loss of generality,  suppose $\psi\equiv 1$ (and thus $\mu$ is a probability measure) in the LePage representation $S_k^e(f)$  in \eqref{eq4}, and $f_i$, $1 \leq i \leq d$, are symmetric.   
We begin with proving the case for $d = 1$, $f = f_1$ and $x_1 = 1$, i.e.,
\begin{equation}\label{case1}
\lambda_{k-1}\left(I_k^e(f)\right)=k \alpha^{k-1}(k!)^{-2} \int_{E^k} f^\alpha\left(\boldsymbol{u}\right)\mu^k(\boldsymbol{u}).
\end{equation}
The proof of (\ref{case1}) essentially follows the approach of \cite[Theorem 5.3]{samorodnitsky1989asymptotic}, and we only give a sketch. 

We first identify the contributing term in the  LePage representation $S_k^e(f)$: It follows from \cite[Corollary 3.2]{samorodnitsky1989asymptotic} that 
\begin{equation}\label{gen cor 3.2}
 \lambda_{k-1} \left(f(T_1,\ldots,T_k)(\Gamma_{1}\ldots\Gamma_k)^{-1/\alpha}\right)=k \alpha^{k-1}(k !)^{-2} \mathbb{E}f^\alpha(T_{\boldsymbol{j}}).
 \end{equation}
Recall the index set $\mathcal{D}_{k,<}$ defined in (\ref{D<}). In view of Corollary \ref{cor Lambda space}, it suffices to show 
 \begin{equation}\label{eq:remainder vanish}
 \lambda_{k-1}\pp{R}=0, \quad R:=\bigvee_{\boldsymbol{j} \in \mathcal{D}_{k,<}\setminus \{(1,\ldots,k)\}}f(T_{\boldsymbol{j}} ) \left[\Gamma_{\boldsymbol{j}}\right]^{-1 / \alpha}.
 \end{equation}
The proof of  \eqref{eq:remainder vanish}  follows  similar arguments as  that for \cite[Theorem 5.3]{samorodnitsky1989asymptotic}.   First, partition the index set $\mathcal{D}_{k, <}\setminus\{(1,\ldots,k)\}$ as in their relation \cite[(5.3)]{samorodnitsky1989asymptotic}, namely,
\begin{equation}\label{p1}
\mathcal{D}_{k, <}\setminus\{(1,\ldots,k)\} = \bigcup_{m=1}^k \mathcal{D}_{k, m, k-m+2}^{<},    
\end{equation}
where for $1 \leq m \leq k$ and $i \geq k-m+2$, we define $$\mathcal{D}_{k, m,i}^{<}=\left\{\left(1,2, \ldots, k-m, j_1, \ldots, j_m\right) \mid\left(j_1, \ldots, j_m\right) \in \mathcal{D}_{m,<}, \, j_1 \geq i\right\}.$$
Define $Y_{k, m, i}=\bigvee_{\boldsymbol{j} \in D_{k, m, i}^<} f(T_{\boldsymbol{j}} ) \left[\Gamma_{\boldsymbol{j}}\right]^{-1 / \alpha}, \, i \geq k-m+2.$
By Corollary \ref{cor Lambda space}, to show (\ref{eq:remainder vanish}), it suffices to prove
\begin{equation}\label{gen}
    \lambda_{k-1}(Y_{k, m, i})=0, m = 1,\ldots, k, \text{ for } i \geq k-m+2.
\end{equation}
%Notably, for each $m$, it is enough to verify \eqref{gen} holds at $i = k-m+2$, since  $Y_{k,m,i}$ is a.s.\ non-increasing in $i$; hence, the conclusion extends to all $i \geq k-m+2$. 

When $m=1$,  \eqref{gen}  follows from arguments analogous to those in \cite[Proposition 5.2]{samorodnitsky1989asymptotic}, so we omit the details. Then we proceed by induction. Suppose \eqref{gen} holds for some $1<m \leq k-1$. We now prove it for $m+1$. Partition $\mathcal{D}_{k, m+1, i}^{<}$ for $1 < m \leq k-1$ as
 $$
\mathcal{D}_{k, m+1, i}^{<} =\left(\bigcup_{d=i+1}^{2 k+1} \mathcal{D}_{k, m, d}^{*,<}\right) \bigcup \mathcal{D}_{k, m+1,2 k+1}^{<}, \quad i \geq k-m+1,
$$
where $\mathcal{D}_{k, m, d}^{*,<}=\left\{\left(1,2, \ldots, k-m-1, d-1, j_1, \ldots, j_m\right)\mid \left(j_1, \ldots, j_m\right) \in \mathcal{D}_m, j_1 \geq d\right\}$, $d \geq k -m+1$. In particular,  when $m = k-1$,  $\mathcal{D}_{k, k-1, d}^{*,<} = \left\{\left( d-1, j_1, \ldots, j_{k-1}\right)\mid\left(j_1, \ldots, j_{k-1}\right) \in \mathcal{D}_{k-1}, {j}_1 \geq d\right\}$.  Then we set $Y^*_{k, m, d}=\bigvee_{\boldsymbol{j} \in D_{k, m, d}^{*,<}} f(T_{\boldsymbol{j}} ) \left[\Gamma_{\boldsymbol{j}}\right]^{-1 / \alpha}$  and it follows that for $1 < m \leq k-1$,
$$
Y_{k, m+1, i}=\pp{\bigvee_{d=i+1}^{2 k+1} Y_{k, m, d}^*} \bigvee Y_{k, m+1,2 k+1},\quad i \geq k-m+1.
$$
Since $\mathcal{D}_{k, m, d}^{*,<} \subset \mathcal{D}_{k, m, d}^{<}$, we have   $Y_{k,m,d}^* \leq Y_{k,m,d}$ a.s.\ for all $d \geq k-m+2$, and hence the assumption of induction  implies  $\lambda_{k-1}(Y_{k, m, d+1}^*)=0$ for any $d \geq k-m+1$. By Corollary \ref{cor Lambda space}, it remains to prove $\lambda_{k-1}\left(Y_{k, m+1,2 k+1}\right)=0$. Then, we analyze $\mathbb{P}\left(Y_{k, m+1,2 k+1}>x\right)$ for $x >0$ , following the approach of  \cite[(5.9)]{samorodnitsky1989asymptotic}. This involves applying a Markov’s inequality and a moment bound in view of Proposition \ref{prop6.7} (which plays the role of \cite[Proposition 5.1]{samorodnitsky1989asymptotic}) and the assumption $f \in \mathcal{L}_{k,+}^\alpha(\mu)$.

For the case $d \geq2$, the result follows from max-linearity (Corollary \ref{cor:gen int}):
$$
\mathbb{P}\left( \left(  I^e_k(f_i)\right)_{i = 1,\ldots,d} \in [\boldsymbol{0},t\boldsymbol{x}]^c\right) =\mathbb{P}\left[\bigvee_{i=1}^d x_i^{-1} I_k^e\left(f_i\right)>t\right]=\mathbb{P}\left[I_k^e(g)>t\right],
$$
where $g:=\bigvee_{i=1}^d x_i^{-1} \cdot f_i$.
Applying the tail asymptotics of $\mathbb{P}\left[I_k^e(g)>t\right]$ from the case $d = 1$ yields the desired result.
\end{proof}
The tail estimate in Theorem~\ref{lambda_{k-1}} implies that 
\(I_k^e(f) \in L^r(\mathbb{P})\) for all \(r \in (0,\alpha)\), provided that \(f\) satisfies 
the conditions of the theorem. Moreover, when this integrability is combined 
with the convergence criterion in Proposition~\ref{prop 4.3}, we obtain the 
following result.

\begin{cor}\label{Cor:I_k conv in L^r}
%Suppose $f, g, f_n \in \mathcal{L}_k$ for $n \in \mathbb{Z}_{+}$. Assume that there exists a probability measure $m$ equivalent to $\mu$ with $\psi = d \mu / d m \in(0, \infty)$ such that $g \cdot\left(\psi^{\otimes k}\right)^{1 / a} \in \mathcal{L}_{k,+}^\alpha(m)$. Moreover, suppose $f_n \rightarrow f$ and $f_n \leq g\, \mu^k$-a.e. for all $n \in \mathbb{Z}_{+}$.
Suppose $k \geq 2$ and $f, g, f_n \in \mathcal{L}_k$ for $n \in \mathbb{Z}_{+}$. Assume that there exists a probability measure $m$ equivalent to $\mu$ with $\psi \in d \mu / d m \in(0, \infty)$ such that  $g\cdot\left(\psi^{\otimes k}\right)^{1 / \alpha} \in \mathcal{L}_{k,+}^\alpha(m)$. Moreover, suppose $f_n \rightarrow f$ and $f_n \leq g,\, \mu^k$-a.e. for all $n \in \mathbb{Z}_{+}$.  Then for any $0<r < \alpha$, as $n \rightarrow \infty$,
$$
I_k^e\left(f_n\right) \rightarrow I_k^e(f) \text{ in } L^r(\mathbb{P}).
$$
\end{cor}
\begin{proof}
By Theorem  \ref{lambda_{k-1}}, each of the random variables $I_k^e(f), I_k^e(g), I_k^e\left(f_n\right)$, $n \in \mathbb{Z}_{+}$,   belongs to $\mrm{MRV}_1(\alpha)$. Hence, each admits finite $r$-th moments for $0<r<\alpha$. Moreover, since $I_k^e\left(f_n\right) \leq I_k^e(g)$ a.s.\ by monotonicity, we may choose $\epsilon \in(0, \alpha-r)$ such that $\sup _n \mathbb{E}\left|I_k^e\left(f_n\right)\right|^{r+\epsilon} \leq \mathbb{E}\left|I_k^e(g)\right|^{r+\epsilon}<\infty$. Thus,  the family $\left\{\left|I_k^e\left(f_n\right)\right|^r: n \in \mathbb{Z}_{+}\right\}$ is uniformly integrable. Applying Proposition  \ref{prop 4.3} (i), the conclusion follows.
\end{proof}

The tail asymptotics also lead to a strengthened version of the monotonicity property stated in 
Corollary~\ref{cor:gen int}.
\begin{cor}
    Suppose $k \geq 2$ and $f,g \in \mathcal{L}_k$, with  $\mu^k\left(f>0\right)>0$ and $\mu^k\left(g>0\right)>0$. Assume that there exists a probability measure $m$ equivalent to $\mu$ with $\psi=d \mu / d m \in(0, \infty)$, so that $f \cdot\left(\psi^{\otimes k}\right)^{1 / \alpha} , g\cdot\left(\psi^{\otimes k}\right)^{1 / \alpha} \in \mathcal{L}_{k,+}^\alpha(m)$. Then,
    $$
    I_k^e(f) \leq I_k^e(g) \text{ a.s., if and only if  } \widetilde{f}(\boldsymbol{u}) \leq \widetilde{g}(\boldsymbol{u}), \mu^k\text{-a.e.,}
    $$
and the equality holds a.s. if and only if  $\widetilde{f}(\boldsymbol{u}) = \widetilde{g}(\boldsymbol{u}), \mu^k\text{-a.e.,}$
\end{cor}
\begin{proof}
The ``if'' part of the first conclusion follows from Corollary \ref{cor:gen int}. For the ``only if'' part,
suppose $ I_k^e(f) \leq I_k^e(g)$ a.s.. Then for $t>0$, 
    $$
    \mathbb{P}(I_k^e(g) > t) =  \mathbb{P}(I_k^e(g) \vee I_k^e(f)> t) = \mathbb{P}((I_k^e(g), I_k^e(f)) \in [\boldsymbol{0},t\boldsymbol{1}]^c).
    $$
Thus, by applying Theorem \ref{lambda_{k-1}} to the tail probability of $I_k^e(f)$ and  $(I_k^e(g), I_k^e(f))$ respectively, we obtain that 
$$
\int_{E^k} \widetilde{f}(\boldsymbol{u}) \mu^k(\boldsymbol{u}) =   \int_{E^k} \widetilde{g}(\boldsymbol{u}) \vee \widetilde{f}(\boldsymbol{u}) \mu^k(\boldsymbol{u}).
$$
Since  $ \widetilde{g}(\boldsymbol{u}) \leq \widetilde{g}(\boldsymbol{u}) \vee \widetilde{f}(\boldsymbol{u}), \forall \boldsymbol{u} \in E^k$, it follows that $ \widetilde{g}(\boldsymbol{u}) = \widetilde{g}(\boldsymbol{u}) \vee \widetilde{f}(\boldsymbol{u}), \mu^k$-a.e. and hence  $ \widetilde{g}(\boldsymbol{u}) \leq  \widetilde{f}(\boldsymbol{u}), \mu^k$-a.e.. 

The second conclusion regarding equality follows from the first conclusion.  
 
\end{proof}
The following criterion for pairwise extremal independence \cite[Definition 2.1.7]{kulik2020heavy} between multiple extremal integrals of the same order can be derived based on Theorem \ref{lambda_{k-1}}.  

\begin{cor}\label{asym indep} 
Under the assumption of Theorem \ref{lambda_{k-1}}, the multiple integrals $I^e_k(f_i),  i = 1, \ldots,d$, are pairwise extremally independent, iff $\wt{f}_i(\boldsymbol{u})\wt{f}_j(\boldsymbol{u}) = 0$ $\mu^k$-a.e.  for any $1\leq i<j \leq d$, i.e., the function $\wt{f}_i$'s   have disjoint supports  modulo  $\mu^k$.
\end{cor}
\begin{proof}
% Note first that $f_i(\boldsymbol{u})f_j(\boldsymbol{u}) = 0$ $\mu^k$-a.e.\ if and only if  $\widetilde{f}_i(\boldsymbol{u})\widetilde{f}_j(\boldsymbol{u}) = 0$ $\mu^k$-a.e., $1\leq i<j \leq k$. 

In view of \cite[Proposition 2.1.8]{kulik2020heavy},  the claimed extremal independence can be characterized by the right hand side of (\ref{eq 42}) taking  the form $\sum_{i=1}^d c_i x_i^{-\alpha}$  {for some constant $c_i\in(0,\infty)$, $i=1,\ldots,d$.} Sufficiency immediately follows since under the assumption, %$\widetilde{f}_i(\boldsymbol{u})\widetilde{f}_j(\boldsymbol{u}) = 0$ $\mu^k$-a.e., $1\leq i<j \leq k$,   
the limit in \eqref{eq 42} can be written as  $\frac{ \alpha^{k-1}}{(k-1)!k!} \left[ \sum_{i=1}^d x_i^{-\alpha} \int_{E^k}\widetilde{f}_i^\alpha\left(\boldsymbol{u}\right)  \mu^k(d\boldsymbol{u})\right]$.
% \begin{equation}\label{eq 48}
%     \begin{aligned}
%      &\lim_{t \rightarrow \infty} t^{\alpha} (\ln t)^{-(k-1)} \mathbb{P}\left( \left( \frac{1}{t} I^e_k(f_i),  i = 1,2,\ldots,d\right) \in [\boldsymbol{0},\boldsymbol{x}]^c\right) \\
%      & =  \frac{ \alpha^{k-1}}{(k-1)!k!} \Bigg[ \sum_{i=1}^d x_i^{-\alpha} \int_{E^k}\widetilde{f}_i^\alpha\left(\boldsymbol{u}\right)  \mu^k(d\boldsymbol{u})\Bigg].
%      \end{aligned}
%      \end{equation}
To see necessity,  letting $x_j$'s tend to infinity except for $x_i$ in the limit in (\ref{eq 42}), we obtain  $c_i = \frac{ \alpha^{k-1}}{(k-1)!k!} \int_{E^k} \widetilde{f}_i^\alpha\left(\boldsymbol{u}\right) \mu^k(d \boldsymbol{u})$.
Comparing   $\sum_{i=1}^d c_i x_i^{-\alpha}$ and \eqref{eq 42}, where we also set all $x_i=1$, we have   
$$    %\label{ext eq}
\int_{E^k} \bigvee_{i=1}^d 
 \widetilde{f}_i^\alpha\left(\boldsymbol{u}\right) \mu^k(d \boldsymbol{u}) =\int_{E^k}  \sum_{i=1}^d \widetilde{f}_i^\alpha \left(\boldsymbol{u}\right)\mu^k(d \boldsymbol{u}).$$
The last relation holds if and only if   $\widetilde{f}_i(\boldsymbol{u})\widetilde{f}_j(\boldsymbol{u}) = 0$ $\mu^k$-a.e., $1\leq i<j \leq k$.

\end{proof}

The multivariate regular variation result in Theorem~\ref{lambda_{k-1}} applies to multiple 
extremal integrals of the same order $k\in\mathbb{Z}_+$. Together with 
Remark~\ref{Rem:k=1 tail}, this shows that extremal integrals of different orders 
have tail asymptotics which, up to multiplicative constants, differ by an additional 
logarithmic factor each time the order $k$ increases by one.

Although the setting involving different orders does not fall within the classical framework 
of multivariate regular variation (which requires that all marginals be tail-equivalent up to 
constants), a meaningful description is still available via \emph{non-standard multivariate 
regular variation}; see \cite[Section~6.5.6]{resnick2007heavy}. 
In particular, in view of \cite[Theorem~6.5]{resnick2007heavy},  non-standard multivariate 
regular variation is equivalent to standard multivariate regular variation after applying 
appropriate monotone marginal transformations. To this end,
the following result shows, when  the orders of two multiple extremal integrals are different, their joint tail behavior forces  extremal independence. Below we adopt the following notation.

\begin{prop}\label{extremal diff order}
     Suppose $I_p^e(f)$ and $I_q^e(g)$ are two multiple extremal integrals with $1 \le p < q$ and 
$p,q \in \mathbb{Z}_+$. Assume that $g$ satisfies the assumption of Theorem~\ref{lambda_{k-1}} 
on a integrand. If $p=1$, we   assume that $f \in L_+^\alpha(\mu)$ and $\mu(f>0)>0$, 
while if $p\ge 2$, we assume that $f$ satisfies the assumption of Theorem~\ref{lambda_{k-1}}.  Suppose    $a,b:(0,\infty)\mapsto(0,\infty)$ satisfy
\begin{equation}\label{lim a b}
    \lim_{t\rightarrow\infty} t\bb{P}\pp{I_p^e(f)>a(t)}=1  \text{ and } \lim_{t\rightarrow\infty} t\bb{P}\pp{I_q^e(g)>b(t)}=1.
\end{equation}
Then 
\begin{equation}\label{doub lim}
    \lim_{t\rightarrow\infty} t\bb{P}\pp{I_p^e(f)>a(t)u, I_p^e(f)>b(t)v}=0
\end{equation}
for any $u,v>0$.
\end{prop}
\begin{proof}
As in the proof of Theorem \ref{lambda_{k-1}},   we assume $\psi = 1$ without loss of generality and work with the LePage representation $S_p^e(f), S_q^e(g)$. In view of  \eqref{gen cor 3.2}, \eqref{eq:remainder vanish}, \eqref{lim a b} and Corollary \ref{cor Lambda space},  $a(t)$ and $b(t)$ satisfy
$$
\mathbb{P}\left(f(T_1, \ldots, T_p)\prod_{r=1}^p \Gamma_r^{-1/\alpha}  > a(t)\right) \sim \frac{1}{t}, \quad \mathbb{P}\left(g(T_1, \ldots, T_q)\prod_{r=1}^q \Gamma_r^{-1/\alpha}  > b(t)\right) \sim \frac{1}{t},
$$
and hence $a(t) \sim  \pp{(p!(p-1)!)^{-1}\, t(\ln t)^{p-1}}^{1/\alpha}$ and $b(t) \sim \pp{(q!(q-1)!)^{-1}\, t(\ln t)^{q-1}}^{1/\alpha}$ as $t \rightarrow \infty$. Also, again by \eqref{eq:remainder vanish} and  Corollary \ref{cor Lambda space}, to show \eqref{doub lim},  it suffices to show for any fixed $u,v >0$, 
\begin{equation}\label{5.11}
    \lim_{t \rightarrow \infty} t\mathbb{P}\left( f\left(T_1, \ldots, T_p\right) \prod_{r=1}^p \Gamma_r^{-1 / \alpha}>a(t)u, \,g\left(T_1, \ldots, T_q\right) \prod_{r=1}^q \Gamma_r^{-1 / \alpha}>b(t)v\right) =0.
\end{equation}
Let $E_i, i \in \mathbb{Z}_+$ be i.i.d.\ standard exponentials.  Note that $\Gamma_k \stackrel{d}{=} E_1+ E_2\ldots+E_k, k \in \mathbb{Z}_+$ and we have
\begin{align}\label{las}
  &\mathbb{P}\left( f\left(T_1, \ldots, T_p\right) \prod_{r=1}^p \Gamma_r^{-1 / \alpha}>a(t) u, \ g\left(T_1, \ldots, T_q\right) \prod_{r=1}^q \Gamma_r^{-1 / \alpha}>b(t)v\right) \notag\\
  & = \mathbb{P}\left( \prod_{r=1}^p \Gamma_r< a(t)^{-\alpha} u^{-\alpha}f\left(T_1, \ldots, T_p\right)^\alpha, \  \prod_{r=1}^q \Gamma_r< b(t)^{-\alpha} v^{-\alpha} g\left(T_1, \ldots, T_q\right)^\alpha\right)\notag\\
  & \leq \mathbb{P}\left( \prod_{r=1}^p E_r< a(t)^{-\alpha} u^{-\alpha} f\left(T_1, \ldots, T_p\right)^\alpha,\  \prod_{r=1}^q E_r< b(t)^{-\alpha} v^{-\alpha} g\left(T_1, \ldots, T_q\right)^\alpha\right).
\end{align}
Define  $$F_t(x, y):=t \mathbb{P}\left(E_1 \cdots E_p \leq a(t)^{-\alpha} x, E_{1} \cdots E_q \leq b(t)^{-\alpha} y\right)$$ for $x,y >0$, $X := u^{-\alpha} f\left(T_1, \ldots, T_p\right)^\alpha $ and $Y: = v^{-\alpha} g\left(T_1, \ldots, T_q\right)^\alpha$, which are finite a.s.\ based on our assumption. 
In view of \eqref{5.11} and \eqref{las}, it suffices to show that $\mathbb{E}F_t(X,Y) \rightarrow 0$ as $t\rightarrow \infty$. 

Since $\mathbb{P}(E_1 \leq s) \sim s$ as $s \downarrow 0$, Lemma \ref{lem F.1} yields
\begin{equation}\label{5.13}
    \mathbb{P}\left(E_1 \cdots E_p \leq a(t)^{-\alpha} \right) \sim   \frac{p!}{t}, \text{ as } t\rightarrow \infty,
\end{equation}
and thus for all sufficiently large $t$, we have
\begin{equation}\label{5.14}
   t \mathbb{P}\left(E_1 \cdots E_p \leq a(t)^{-\alpha} \right) \leq  2p!.
\end{equation}
Let
$H_t(x):=t \mathbb{P}\left(E_1 \cdots E_p \leq a(t)^{-\alpha} x\right)$ and $L_t(x): = \frac{\mathbb{P}\left(E_1 \cdots E_p \leq a(t)^{-\alpha} x\right)}{\mathbb{P}\left(E_1 \cdots E_p \leq a(t)^{-\alpha} \right)}$, for any $x>0$. It follows from \eqref{5.14} that for all sufficiently large $t$,  we have
\begin{equation}\label{eq:H_t(x)}
H_t(x) = L_t(x) t \mathbb{P}\left(E_1 \cdots E_p \leq a(t)^{-\alpha} \right) \leq 2p! L_t(x).
\end{equation}
Further, Lemma \ref{lem F.3} and \eqref{5.13} imply that for any fixed $x >0$,
\begin{equation}\label{eq:L_t(x)}
\lim_{t \rightarrow \infty} L_t(x) = \lim_{t \rightarrow \infty} (p!)^{-1} H_t(x) = x.
\end{equation}
Moreover, Lemma \ref{lem F.3} and the fact that $f \in L_+^{\alpha}(\mu^p)$ imply $$\lim_{t \rightarrow \infty} \mathbb{E}L_t(X) = \mathbb{E}X < \infty.$$ Combining this with \eqref{eq:H_t(x)} and \eqref{eq:L_t(x)}, it then follows from Pratt's lemma \cite[Theorem 1.23]{kallenberg2021foundations} that 

\begin{equation}\label{eq:E H_t limit}
\lim _{t \rightarrow \infty} \mathbb{E} H_t(X)=(p!)^{-1} \mathbb{E} X,
\end{equation}
Next, by Lemma \ref{lem F.2},  for any fixed  $x, y \in[0, \infty)$, there exists a constant $C>0$, such that we have  $F_t(x, y) \le C  y \pp{\frac{\ln \ln t}{\ln t}}^{q-p}$ for all $t$ sufficiently large,  and hence $F_t(x, y) \rightarrow 0$ as $t \rightarrow \infty$. Observe also that $F_t(x, y) \leq H_t(x)$ for any $x, y >0$.
Combining these with  \eqref{eq:E H_t limit}, we obtain
$$
\lim _{t \rightarrow \infty} \mathbb{E}\left[F_t(X, Y)\right]=\mathbb{E}\left[\lim _{t \rightarrow \infty} F_t(X, Y)\right]=0
$$
by Pratt's lemma.
\end{proof}
Although extremal independence holds for multiple extremal integrals of different orders, 
we shall see in the next section that full independence between two such integrals requires 
additional assumptions.

\section{Independence between multiple extremal integrals}\label{s7}
% In this section, we develop a product formula for multiple extremal integrals and provide necessary and sufficient conditions for the independence of extremal integrals of arbitrary orders. We shall study product formula for multiple stochastic integrals to establish independence, similar to Wiener-Itô integrals and $\alpha$-stable integrals \cite{rosinski1999product}. 

In this section, we develop a criterion for the independence between two multiple extremal integrals, 
possibly of different orders. As will be shown in  Section~\ref{max sec}, pairwise independence 
of multiple extremal integrals in fact implies their mutual independence.
A key step in our argument is a product formula for multiple extremal integrals, which may be of 
independent interest. A similar idea was used in \cite[]{rosinski1999product} in the context of 
multiple stable integrals. For convenience, throughout this section we work with symmetric 
integrands (see Proposition~\ref{sim inv}).

\subsection{Product formula}
Recall that $\mathcal{D}_p=\left\{\boldsymbol{i} \in \bb{Z}_+^p\mid \text { all } i_1, \ldots, i_p \text { are distinct}\right\}$, $p\in \bb{Z}_+$. 
%Let $\mathcal{D}_p^n=\mathcal{D}_p \cap\{1, \ldots, n\}^p$. 
The following lemma can be established directly in an elementary manner, whose proof is omitted.
\begin{lem}
\label{pre prod}
Let $F: \mathcal{D}_p \mapsto [0,\infty]$ and $G: \mathcal{D}_q \mapsto [0,\infty]$ be symmetric functions, where $p, q \in \bb{Z}_+$. Then, the following identity holds:
$$
\left(\bigvee_{\mathbf{i} \in \mathcal{D}_p} F(\mathbf{i})\right) \left(\bigvee_{\boldsymbol{j} \in \mathcal{D}_q} G(\boldsymbol{j})\right) =\bigvee_{r=0}^{p \wedge q} \bigvee_{\mathbf{k} \in \mathcal{D}_{p+q-r}} U_r(\mathbf{k}),
$$
where
$U_r(\mathbf{k})=F\left(k_1, \ldots, k_r, k_{r+1}, \ldots, k_p\right) G\left(k_1, \ldots, k_r, k_{p+1}, \ldots, k_{p+q-r}\right)$ (note that ${U}_r$ is not necessarily symmetric). 
\end{lem}
The following formula directly follows from LePage representation \eqref{eq4} and  the lemma above.  
\begin{prop}\label{prod formula}
(Product formula).
Suppose $f\in \mathcal{L}_p$ and $g\in \mathcal{L}_q$ are symmetric functions, where $p, q \in \bb{Z}_+$.  Then 
\begin{equation}
\begin{aligned}\label{prodf}
  I_p^e(f) I_q^e(g)& \stackrel{d}{=} \bigvee_{r=0}^{p \wedge q} \left[ \bigvee_{\mathbf{s} \in \mathcal{D}_{p+q-2r}}  \bigvee_{\mathbf{k} \in \mathcal{D}_r} h_r\left(T_{\mathbf{k}}, T_{\mathbf{s}}\right) [\Gamma_{\mathbf{k}}]^{-2/\alpha}  [\psi(T_{\mathbf{k}})]^{2/\alpha} [\Gamma_{\mathbf{s}}]^{-1/\alpha} [\psi(T_{\mathbf{s}})]^{1/\alpha}\right] \\
  & = :\bigvee_{r=0}^{p \wedge q} S_{p,q}^{(r)}(f \otimes g),
\end{aligned}
\end{equation}
where  $h_r:E^{p+q-r}\mapsto [0,\infty]$, $r=0,\ldots,p\wedge q$ is defined as follows: When $p\le q$, we set $$h_r\left(x_1, \ldots, x_{p+q-r}\right)= f \left(x_1, \ldots, x_r,  \ldots, x_p\right) g\left(x_1, \ldots, x_r, x_{p+1}, \ldots, x_{p+q-r}\right),$$ which in terms of the index sets $\mathbf{k} \in \mathcal{D}_r, \mathbf{s} \in \mathcal{D}_{p+q-2 r}$ can be written as
$$
h_r\left(T_{\mathbf{k}}, T_{\mathbf{s}}\right)=f\left(T_{k_1}, \ldots, T_{k_r}, T_{s_1}, \ldots, T_{s_{p-r}}\right) g\left(T_{k_1}, \ldots, T_{k_r}, T_{s_{p-r+1}}, \ldots, T_{s_{p+q-2 r}}\right) .
$$
When $p>q$, $h_r$ is defined similarly switching the roles between $f$ and $g$ above, and $h_0$ is understood as $f\otimes g$.
\end{prop}
 It is  possible to develop an a.s.\ formula for $I_p^e(f)I_q^e(f)$ involving $M_\alpha$. We choose  not to pursue this here since the distributional representation above suffices for our  purpose below.

\subsection{Tail of the product and pairwise independence of extremal integrals}

In order to study the tail behavior of $I_p^e(f) I_q^e(g)$ for suitable integrands $f, g$,  we first describe the tail behavior of each $S_{p,q}^{(r)}(f \otimes g)$ in \eqref{prodf}. Recall the class $\mathcal{L}_{p,+}^\alpha (\mu)$ from Definition \ref{def:L_{k,+}}.
\begin{lem}\label{thm: tail of tensor}
Let  $f\in \mathcal{L}_p$ and $g\in \mathcal{L}_q$ be symmetric functions, where $p, q \in \bb{Z}_+$.
Suppose that there exists a probability measure $m$ equivalent to $\mu$ with $\psi\in d\mu/dm\in (0,\infty)$ $m$-a.e.,  so that
 $f \cdot\left(\psi^{\otimes p}\right)^{1 / \alpha} \in \mathcal{L}_{p,+}^\alpha (m)$ and $g\cdot\left(\psi^{\otimes q}\right)^{1 / \alpha}\in \mathcal{L}_{q,+}^\alpha(m)$. We have for $r=1,\ldots,p\wedge q$,
\begin{equation}\label{eq 50}
\lim _{\lambda \rightarrow \infty} \lambda^{\alpha / 2}(\ln \lambda)^{-(r-1)} \mathbb{P}\left( S_{p,q}^{(r)}(f \otimes g) >\lambda\right)  = k_{r, \alpha} C_r(f, g),
\end{equation}
 where  
%$k_{0, \alpha} = \alpha^{p+q-1}/[(p+q-1)!(p+q)!]$,  
$k_{r,\alpha} = (\alpha / 2)^{r-1} [r!(r-1) !]^{-1}$, and 
\begin{align}\label{C_r(f, g)}
  C_r(f, g)=
  % \begin{cases}
  % \|\widetilde{h}_0\|_{L^\alpha\left(\mu^{p+q}\right)}^\alpha      & r=0,\\
\int_{E^r} \mathbb{E}\left|I_{p+q-2 r}^e\left(h_r\left(s_1, \ldots, s_r, \cdot\right)\right)\right|^{\alpha / 2} \mu\left(d s_1\right) \ldots \mu\left(d s_r\right),  \quad r=1,\ldots,p\wedge q.   
%\end{cases}
\end{align} 
Here, $I^e_{p+q-2 r}\left(h_r\left(s_1, \ldots, s_r, \cdot\right)\right)$ is understood as a multiple extremal integral of order $p+q-2 r$ with respect to the random sup measure $M_\alpha$ of $h_r$ regarded as a function of its last $p+q-2 r$ coordinates (when $r=p=q$, $I^e_{0}\left(h_r\left(s_1, \ldots, s_r, \cdot\right)\right)$ is understood as  $f\left(s_1, \ldots, s_r\right) g\left(s_1, \ldots,s_r\right)$).
%and (\ref{eq 50}) reduces to the special case as in (\ref{case1}).}
\end{lem}
\begin{proof}
  The proof is similar to that of \cite[Theorem 3.1]{rosinski1999product}.  We leave the details to Appendix \ref{app lem 6.3}.
\end{proof}
 For   convenience of stating the next result,  we   set  additionally $k_{0, \alpha} = \alpha^{p+q-1}/[(p+q-1)!(p+q)!]$ and $C_0(f,g)=\|\widetilde{h}_0\|_{L^\alpha\left(\mu^{p+q}\right)}^\alpha$.

\begin{thm}\label{Tail product thm}
  % Assume that $f \cdot\left(\psi^{\otimes p}\right)^{1 / \alpha} \in \mathcal{L}_{p,+}^\alpha (m)$ and $g\cdot\left(\psi^{\otimes q}\right)^{1 / \alpha}\in \mathcal{L}_{q,+}^\alpha(m)$, where $p, q \in \bb{Z}_+$ and $k \geq 2$.
  Under the assumption of Lemma \ref{thm: tail of tensor}, suppose in addition that $\mu^p(f>0)>0$ and $\mu^q(g>0)>0$. Let $r=\max \left\{i \mid C_i(f, g) \neq 0\right\}$.
  %and 
  % $C_0(f, g)=\|\widetilde{h}_0\|_{L^\alpha\left(E^{p+q}\right)}^\alpha$. 
  Then, as $\lambda \rightarrow \infty$, 
\[
\mathbb{P}\left( I_p^e(f) I_q^e(g) >\lambda\right) \sim \begin{cases}k_{0, \alpha} \lambda^{-\alpha}(\ln \lambda)^{p+q-1} C_0(f, g), & \text { if } r=0, \\ k_{r, \alpha} \lambda^{-\alpha / 2}(\ln \lambda)^{r-1} C_r(f, g), & \text { if } r=1, \ldots, p \wedge q.\end{cases}
\]
% where $k_{0, \alpha} = \alpha^{p+q-1}/[(p+q-1)!(p+q)!]$, when $r \neq 0$, $k_{r,\alpha} = (\alpha / 2)^{r-1} [r!(r-1) !]^{-1}$.
\end{thm}
\begin{proof}

When $r = 0$, we have $C_1(f, g)=\cdots=C_{p \wedge q}(f, g)=0$, which implies that for $r = 1, \ldots, p \wedge q,$ that
$h_r=0$ $\mu^{p+q-r}$-a.e..
Thus, $S^{(r)}_{p,q} (f\otimes g)= 0$ a.s. for $r = 1, \ldots, p \wedge q$. In view of \eqref{prodf} and \eqref{eq4}, it follows that  
$$
 I_p^e(f) I_q^e(g) \EqD S_{p,q}^{(0)}(f\otimes g)= S_{p+q}^e(f\otimes g)\EqD I_{p+q}^e(f\otimes g),
$$
where we recall $f\otimes g=h_0$.  Then the conclusion follows from Lemma \ref{thm: tail of tensor}.

When $r \ge 1$, the conclusion follows 
from  Lemma \ref{thm: tail of tensor}, and  Corollary \ref{cor Lambda space}.

\end{proof}

\begin{thm}\label{independence}
 % Assume $f \cdot\left(\psi^{\otimes p}\right)^{1 / \alpha} \in \mathcal{L}_{p,+}^\alpha (m)$ and $g\cdot\left(\psi^{\otimes q}\right)^{1 / \alpha}\in \mathcal{L}_{q,+}^\alpha(m)$,  where $p, q \in \bb{Z}_+$ and $k \geq 2$.
 Under the assumption of Lemma \ref{thm: tail of tensor},
 the multiple extremal integrals $I_p^e(f)$ and $I_q^e(g)$ are independent if and only if there exist disjoint measurable sets $A, B \in \mathcal{E}$ such that $\operatorname{supp}\{f\} \subset A^p \operatorname{modulo} \mu^p$ and $\operatorname{supp}\{g\} \subset B^{ q} \operatorname{modulo} \mu^{q}$.
\end{thm}
\begin{proof}
We can assume without loss of generality that $\mu^p(f>0)>0$ and $\mu^q(g>0)>0$. Otherwise, the independence trivially holds since at least one of the multiple extremal integrals is $0$ a.s..

To prove sufficiency, note that, in view of definition \ref{defn s}, $\pp{M_\alpha(D \cap A)}_{D \in \mathcal{E}}$ and $\pp{M_\alpha(D \cap B)}_{D \in \mathcal{E}}$ are independent. This implies that
$\pp{M_\alpha^{(k)}(D \cap A^p)}_{D \in \mathcal{E}^{(k)}}$  and  $\pp{M_\alpha^{(k)}(D \cap B^q)}_{D \in \mathcal{E}^{(k)}}$ are independent, which can be seen by restricting the approximation sets in Theorems \ref{claim 5.3} and \ref{thm sf case}  to $A$ and $B$, respectively. Then observe that 
 $I_p^e(f) =I_p^e(f \mathbf{1}_{A^p})$  and $I_q^e(g)=I_q^e(g \mathbf{1}_{B^q})$ a.s.\ by assumption. So the conclusion follows once we approximate $f \mathbf{1}_{A^p}$ and $g \mathbf{1}_{B^q}$ via Definition \ref{def:mult extr int}   restricting  to $A^p$ and  $B^q$, respectively.

Now we show necessity. Since $I_p^e(f)$ and $I_q^e(g)$ are independent, and we know the tail asymptotic of each from Theorem \ref{lambda_{k-1}},  by  \cite[Proposition 3.1 (ii)]{kifer2017tails} (their result was stated for $\alpha\in (0,2)$ but it extends readily to $\alpha\in (0,\infty)$), we can derive $\mathbb{P}\pp{I_p^e(f)I_q^e(g)>\lambda} \sim C \lambda^{-\alpha}(\ln \lambda)^{p+q-1}$ for some constant $C>0$ as $\lambda\rightarrow\infty$.  From Theorem \ref{Tail product thm}, we have $C_1(f, g)=\cdots=$ $C_{p \wedge q}(f, g)=0$. Hence for $r=1, \ldots, p \wedge q$, we have
$
h_r=0 
$
$\mu^{p+q-r}$-a.e..The conclusion then follows from similar arguments as in the proof of \cite[Theorem 4.3]{rosinski1999product}.
\end{proof}
\begin{rem}
 %For max-stable random vectors, extremal independence coincides with independence. However, this equivalence does not hold for the multiple extremal integrals. Note that independence implies extremal independence. Moreover, when $p = q$, the condition in Theorem~\ref{independence} implies that of Corollary~\ref{asym indep} but not vice versa. Indeed, the former condition requires non-overlapping coordinate projections of the supports of symmetric integrands. 
For single extremal integrals, which are max-stable random vectors, extremal independence is equivalent to independence \cite[Proposition 5(b)]{papastathopoulos2016conditional} \cite[p. 192, p. 195]{resnick2008extreme}. However, this equivalence does not extend to multiple extremal integrals. Clearly, independence always implies extremal independence. On the other hand, when $p = q$, the condition stated in Theorem~\ref{independence} implies that of Corollary~\ref{asym indep}, but not conversely. In particular, the condition for independence in Theorem~\ref{independence} requires that the symmetrized integrands have supports with disjoint coordinate projections, which is strictly stronger than the requirement in Corollary~\ref{asym indep} that the supports be disjoint.
\end{rem}

\subsection{Max-infinitely divisibility and mutual independence of extremal integrals} \label{max sec}
   Recall that a random vector $\boldsymbol{\zeta}=\left(\zeta_i\right)_{i=1}^d\in [0,\infty]^d$ is said to be max-infinitely divisible if, for any integer $n \geq 1$, there exist i.i.d.\ random vectors $\boldsymbol{\zeta}^{(1)}, \ldots, \boldsymbol{\zeta}^{(n)}$ such that $\boldsymbol{\zeta} \stackrel{d}{=} \bigvee_{j=1}^n \boldsymbol{\zeta}^{(j)}$, where the maximum $\bigvee$ is taken componentwise. The law of such a random vector $\boldsymbol{\zeta}$ is characterized by 
\begin{equation}\label{max-id}
\mathbb{P}[\boldsymbol{\zeta} \leq \boldsymbol{x}]=\exp \left(-\nu\left([\mathbf{0}, \boldsymbol{x}]^c\right)\right), \quad \boldsymbol{x} \in [0,\infty]^d,    
\end{equation}
for some  Borel measure $\nu$ on $[0,\infty]^d \backslash \{\mathbf{0}\}$ satisfying $\nu(A)<\infty$ for every Borel set $A$ bounded away from the origin. We refer to \cite[Section 5.1]{resnick2008extreme} for more details.
Now we show that the random vectors of extremal integrals are max-infinitely divisible.

\begin{thm}\label{max-id them}
Consider general integrands $f_i \in \mathcal{L}_{k_i}$ as in \eqref{eq:cl L_k},   $k_i \in \mathbb{Z}_{+}$ (not necessarily identical across $i$), $i=1,\ldots,d$, $d\in \bb{Z}_+$.
%Suppose there exists a probability measure $m$ equivalent to $\mu$ with $\psi=  d\mu/dm\in (0,\infty)$.  
Then the law of random vector $\left(I_{k_i}^e\left(f_i\right)\right)_{i=1}^d$ is max-infinitely divisible.
\end{thm}

\begin{proof}

Let $\left(\Gamma_i \right)_{i \in \bb{Z}_+}$ be the arrival times of a standard Poisson process on $[0, \infty)$, and $\pp{T_i}_{i \in \bb{Z}_+}$ be  a sequence of i.i.d.\ random variables with distribution $m$,  a probability measure   equivalent to $\mu$ with $\psi=d\mu/dm\in (0,\infty)$ $m$-a.e., independent of $\pp{\Gamma_i}_{i \in \bb{Z}_+}$. Then $N := \sum_{i=1}^\infty\delta_{(T_i, \,  \Gamma_i^{-1/\alpha})}$ is a Poisson point process on $S = E \times (0,\infty)$ with mean measure $m\times \nu_\alpha$ where $\nu_\alpha(dx) = \alpha x^{-\alpha-1}dx$, $x > 0$.   Let $K$ be the uniform probability measure on $\{1,2 \ldots, n \}$, i.e., $K(\{r\})=\frac{1}{n}$ for $ r=1, \ldots, n $. Suppose $Y_i, i \in \mathbb{Z}_+$ are i.i.d.\ random variables with law $K$, independent of $N$. Define the $K$-marking of $N$ by 
$\widehat{N}=\sum_{i=1}^\infty \delta_{\left(\left(T_i, \,  \Gamma_i^{-1/\alpha}\right), Y_i\right)}$  (see \cite[Definition 5.3]{last2018lectures}).  By the marking theorem \cite[Theorem 5.6]{last2018lectures}, $\widehat{N}$ is a Poisson point process on $S \times \{1,2, \ldots,n\}$ with intensity $\left(m \times \nu_\alpha\right) \times K$. Consequently,  it follows from \cite[Theorem 5.8]{last2018lectures} that the processes $\widehat{N}^{(r)} := \widehat{N}(\cdot \times\{ r\})$,  $1 \leq r \leq n$, are i.i.d.\ Poisson point processes on $S$, each with intensity $\frac{1}{n} m \times \nu_\alpha $.  

  % Note that the Borel space $(E,\cl{E})$ is  localized by a localizing sequence $E_n\in \cl{E}$, $n\in \bb{Z}_+$, satisfying $E_n\subset E_{n+1}$ and $\cup_n E_n=E$ by  \cite[P.15]{kallenberg2021foundations} under the assumption that $(E, \mathcal{E}, \mu)$ is a $\sigma$-finite measure space.  Write $\mathcal{M}_S$ for the space of locally finite point measures on $S$.
  % Also, $m$ is a locally finite measure on $(E,\cl{E})$ as required.
   Next, to describe  measurable enumeration of the point processes, we follow the framework of localized Borel space (see \cite[P.15]{kallenberg2021foundations}). In particular, we localize the space $S=E\times [0,\infty)$ by $S_n:=E\times [1/n,\infty)$, $n\in \bb{Z}_+$ . A Borel subset of $S$ is said to be bounded if it is a subset of some $S_n$. A Borel measure on $S$ is said to be locally finite, if it is finite on any bounded Borel subset.
   Write $\mathcal{M}_S$ for the space of locally finite  measures on $S$.  Then, each $\widehat{N}^{(r)}$, with a modification on a zero-probability event if necessary, is a random element that takes value in $\mathcal{M}_S$, $r=1,\ldots,n$.
  Hence, in view of \cite[Theorem 2.19 (i)]{kallenberg2021foundations}, there exist measurable maps $\rho_r: \mathcal{M}_S \mapsto S^{\infty}$, $r = 1, \ldots, n$, such that $\left(T_i, \,  \Gamma_i^{-1/\alpha}\right)_{i : Y_i = r}:=\rho_r(\widehat{N}^{(r)})$. 
For $i = 1,\ldots,d$, define $g_i: S^\infty \mapsto [0,\infty]$ by
$$
g_i\left(\left(x_j, y_j\right)_{j \geq 1}\right):=\bigvee _{\substack{(j_1, \ldots, j_{k_i}): \\ j_\ell\text {'s distinct }}}\left[f_i\left(x_{j_1}, \ldots, x_{j_{k_i}}\right) \prod_{\ell=1}^{k_i} \psi\left(x_{j_{\ell}}\right)^{1 / \alpha} y_{j_{\ell}}\right], 
$$
and define $F: S^{\infty} \rightarrow [0,\infty]^d$ as the map $ \left(x_j, y_j\right)_{j \geq 1}\mapsto\left(g_1\left(\left(x_j, y_j\right)_{j \geq 1}\right), \ldots, g_d\left(\left(x_j, y_j\right)_{j \geq 1}\right)\right)$.
For each $r \in \{ 1,\ldots,n\}$, define the $d$-dimensional vector
$$
\boldsymbol{\zeta}^{(r)} := F(\rho_r(\widehat{N}^{(r)})) = \left(\bigvee_{\substack{\left(j_1, \ldots, j_{k_i}\right):\\ j_\ell\text{'s distinct } \\ Y_{j_{\ell}}=r, \,\text{ for all }1 \leq \ell\leq k_i}} f_i\left(T_{j_1}, \ldots, T_{j_{k_i}}\right)
\prod_{\ell=1}^{k_i} \psi\left(T_{j_{\ell}}\right)^{1 / \alpha} \Gamma_{j_{\ell}}^{-1/\alpha}
\right)_{i = 1, \ldots, d}.
$$
Observe that since $\widehat{N}^{(r)}$, $r = 1\ldots, n$ are i.i.d., we have that $\boldsymbol{\zeta}^{(r)}$, $r = 1\ldots, n$ are also i.i.d.. Note also that
\begin{equation}
  \bigvee_{r = 1}^n \boldsymbol{\zeta}^{(r)}  \overset{\text{a.s.}}{=} \pp{S_{k_i}^e(f_i)}_{i=1,\ldots, d},
\end{equation}
where each of the components on right-hand side above is a LePage representation  as in \eqref{eq4}.
 We thus conclude that  $\left(I_{k_i}^e\left(f_i\right)\right)_{i=1}^d \EqD  \bigvee_{r = 1}^n \boldsymbol{\zeta}^{(r)}  $, which establishes that $\left(I_{k_i}^e\left(f_i\right)\right)_{i=1}^d$ is max-infinitely divisible. The characterization stated as \eqref{max-id} then follows directly from \cite[Proposition 5.8]{resnick2008extreme}.
\end{proof}
\begin{cor}
Under the assumption of Theorem \ref{max-id them}, pairwise independence of the components of the random vector $\left(I_{k_i}^e\left(f_i\right)\right)_{i=1}^d$ implies mutual independence.
\end{cor}
\begin{proof}
    This follows directly from the fact that the random vector is max-infinitely divisible and  \cite[Proposition 5.24]{resnick2008extreme}.
\end{proof}

\section{A multiple regenerative model} \label{sec:mult reg}

%expressed by multiple extremal integrals}
In the recent works \cite[]{bai2023tail,bai2024phase,bai2024remarkable}, a class of regularly varying stationary processes was introduced and shown to exhibit interesting extreme value behavior. In particular, the extreme limits of this class of models exhibit a phase transition from short-range dependence to long-range dependence as a model parameter crosses a critical value. Here, short-range dependence refers to extreme value behavior comparable to the i.i.d.\ case, up to a multiplicative scale factor in the limit theorem, whereas long-range dependence corresponds to qualitatively different behavior, potentially involving different normalization rates and limit distributions; see, e.g., \cite[]{samorodnitsky2016stochastic,dombry2017ergodic} for further background on these notions. Moreover, even in the short-range dependence regime, the model displays an unusual phenomenon in which the so-called candidate extremal index fails to coincide with the true extremal index, the parameter that determines the aforementioned multiplicative scale factor under short-range dependence.

In these works, this class of processes was defined through multiple stable integrals. In this section, we discuss extensions of these processes, as well as the associated results, to models defined via multiple extremal integrals.

We start by  recalling some preliminaries on renewal processes.  More details can be found in  \cite[]{bai2023tail,bai2024phase}. Consider a discrete-time renewal process starting at the origin with the consecutive renewal times denoted by 
\begin{equation}\label{eq:tau no delay}
\boldsymbol{\tau}:=\left\{\tau_0=0, \tau_1, \tau_2, \ldots\right\},
\end{equation}
where $\tau_{i+1}-\tau_i$ are i.i.d.\ $\bb{Z}_+$-valued random variables with the distribution function $F$, that is, $F(x)=\mathbb{P}\left(\tau_{i+1}-\tau_i \leq x\right)$, $i \in \bb{N}_0:= \{0,1,2, \ldots\}$, $x \in \bb{R}$. We assume for some constant $\mathrm{C}_F>0$ that
\begin{equation}\label{eq:C_F}
\bar{F}(x)=1-F(x) \sim \mathrm{C}_F x^{-\beta} \text { as } x \rightarrow \infty \quad \text { with } \quad \beta \in(0,1),
\end{equation}
and the following technical assumption: with $\Delta F(n):=F(n)-F(n-1)$ denoting the probability mass function of the renewal distribution,
$
\sup _{n \in \bb{Z}_+} n \Delta F(n)/\bar{F}(n)<\infty.
$
The condition allows us to equivalently express the condition \eqref{eq:C_F} by the  renewal mass function  $\bb{P}(n\in \boldsymbol{\tau})\sim  (\Gamma(\beta)\Gamma(1-\beta)\mathrm{C}_F)^{-1}  n^{\beta-1}$ as $n\rightarrow\infty$.
Next, we introduce an independent ``random'' shift to the renewal process, so that the starting point may not be the origin. In particular, we recall the stationary shift distribution of the renewal process on $\bb{N}_0$ denoted by $\pi$. Since the renewal distribution has an infinite mean, the stationary shift distribution $\pi$ is a $\sigma$-finite and infinite measure on $\bb{N}_0$ unique up to a multiplicative constant. We shall work with
$
\pi(\{k\}):=\bar{F}(k),
$
$k \in \bb{N}_0 $.
Suppose $d$ is a random shift on $\mathbb{N}_0$ with infinite law $\pi$, independent of the renewal process $\boldsymbol{\tau}$. Then it is  known that the law of the shifted renewal process $d+\boldsymbol{\tau}:=\left\{d, d+\tau_1, d+\tau_2, \ldots\right\}$ is shift-invariant (see, e.g., \cite[Section 2.3]{bai2023tail}). Here by the (infinite) law $\mu$ of
\begin{equation}\label{eq:tau*}
\boldsymbol{\tau}^*:=d+\boldsymbol{\tau} \subset \bb{N}_0,
\end{equation}
we mean an infinite measure   on $(E := \mathcal{P}(\mathbb{N}_0)\setminus \{\emptyset\}, \mathcal{F})$, where $\mathcal{F}$ denotes the cylindrical $\sigma$-field obtained by identifying the power set $\mathcal{P}(\mathbb{N}_0)$ with $\{0,1\}^{\mathbb{N}_0}$ via indicator functions.  In particular, $\mu$ is the pushforward of the product measure of $\pi$ and the distribution of $\boldsymbol{\tau}$, under the map in~\eqref{eq:tau*}. Recall that  $\{0,1\}^{\mathbb{N}_0}$  equipped with the  product discrete topology is a Polish space, and hence $E$ is a Borel space with Borel $\sigma$-field $\cl{F}$. In addition, $E$ admits a localizing sequence 
\begin{equation}\label{eq:E_n}
E_n:=\{ T \subset \bb{N}_0:\ \inf\pp{ T}\le n \},
\end{equation}
$n\in \bb{N}_0$, where $\mu(E_n)=\pi(\{0,\ldots,n\})<\infty$.

We are now ready to introduce the stationary model of interest.  
Consider
$$
\sum_{i=1}^{\infty} \delta_{\left(\eta_i, \boldsymbol{\tau}_i^*\right)} \stackrel{d}{=} \operatorname{PPP}\left((0, \infty) \times \bb{Z}_+,\, \alpha x^{-\alpha-1} d x d \mu\right), \quad \alpha\in(0,\infty),
$$
where the notation $\mathrm{PPP}(A,\nu)$ stands for a Poisson point process on space $A$ with intensity measure $\nu$.
% Suppose $\left\{\boldsymbol{\tau}^{(i)}\right\}_{i \in \bb{Z}_+}$ are i.i.d.\ copies of the non-shifted renewal process $\boldsymbol{\tau}$ which are independent of the point process above. Set $\boldsymbol{\tau}^{\left(i, d_i\right)}=d_i+\boldsymbol{\tau}^{(i)}, i \in \bb{Z}_+$. 
Then, our new $k$-tuple regenerative model, $k\in \bb{Z}_+$,   is defined as the stationary process
\begin{equation}\label{eq:mult reg mod}
\pp{X_t}_{t \in \bb{N}_0}=\pp{\bigvee_{\boldsymbol{i}\in \cl{D}_k}\left[\eta_{\boldsymbol{i}}\right] \boldsymbol{1}_{\left\{t \in \bigcap_{r=1}^k \boldsymbol{\tau}_{i_r}^*\right\}}}_{t \in \bb{N}_0}=\pp{\bigvee_{\boldsymbol{i}\in \cl{D}_k}\left(\eta_{i_1}\ldots\eta_{i_k}\right) \boldsymbol{1}_{\left\{t \in   \boldsymbol{\tau}_{i_1}^*\right\}}\times \ldots \times \boldsymbol{1}_{\left\{t \in   \boldsymbol{\tau}_{i_k}^*\right\}} }_{t \in \bb{N}_0}.
\end{equation}
Note that in view of  Section \ref{alt ppp}, the process $\pp{X_t}_{t\in \bb{Z}_+}$ may be viewed as a multiple extremal integral with   $t$-indexed integrands. Since by shift-invariance, we have 
\begin{equation}\label{eq:mu(t in tau*)}
\mu(t\in \boldsymbol{\tau}^*)=\mu(0\in \boldsymbol{\tau}^*)=\pi(\{0\})=1,
\end{equation}
the integrability of the multiple extremal integral follows from Proposition \ref{Pro:tensor int cond}.
% In particular, in terms of  \eqref{eq:ppp rep},  one may regard each point $\xi_i$ there as ${\tau}^{(\tau_i,d_i)}$, $i\in \bb{N}$, and the integrand $f$ there as $f_t(\xi_{\boldsymbol{i}})=\boldsymbol{1}_{\left\{t \in \xi_{i_1}  \right\}}\times \ldots  \times \boldsymbol{1}_{\left\{t \in \xi_{i_k}  \right\}}$, $\boldsymbol{i}=(i_1,\ldots,i_k)\in \cl{D}_k$. Here,  the space $E$ may be identified with $\{0,1\}^{\bb{N}_0}$ equipped with product topology (of disecrete topology), and $\mu$ is the infinite law of $\boldsymbol{\tau}^*$ in \eqref{eq:tau*}.  
We mention that the representation may be alternatively formulated in terms of a conservative infinite-measure-preserving dynamical system; see \cite[]{bai2020functional}.

% on a suitable space $(E,\cl{E},\mu)$ with an $\alpha$-Fr\'echet random sup measure $M_\alpha$ as in Definition \ref{defn s}, one may represent $\pp{X_t}_{t\in \bb{N}_0}\EqD \pp{I_k^e(f_t)}_{t\in \bb{N}_0}$, for some suitable indicator functions $f_t\in \mathcal{L}_k$.

\begin{thm}
Let $\boldsymbol{\tau}^{(1)}, \ldots, \boldsymbol{\tau}^{(k)}$ be i.i.d.\ copies of the non-delayed renewal process $\boldsymbol{\tau}$ in \eqref{eq:tau no delay}, and $P_\alpha$ is a standard Pareto law (i.e. $\bb{P}\pp{P_\alpha>x}=x^{-\alpha}$, $x\ge 1$) independent of  $\Theta_s$'s. 
The conditional joint law $\cl{L}\pp{X_0/c,X_1/c,\ldots, X_t/c \mid X_0>c }$ converges weakly to the joint law  $\cl{L}\pp{P_\alpha \Theta_0,\ldots,P_\alpha\Theta_t}$ as $c\rightarrow\infty$, for each fixed $t\in \bb{N}_0$, where $\Theta_s=\mbf{1}_{\{ s\in  \boldsymbol{\tau}^{(1)}\cap \ldots \cap \boldsymbol{\tau}^{(k)} \}}$, $s\in \bb{N}_0$. 
\end{thm}
\begin{proof}
The result can be obtained by arguments similar to those of \cite[Theorem 1.1]{bai2023tail}. Here, we sketch an alternative proof based on the characterization of the multivariate regular variation of multiple extremal integrals in Theorem \ref{lambda_{k-1}}.  

First, by arguing in the same way as below  \eqref{eq:mult reg mod} based on  Proposition \ref{Pro:tensor int cond}, each integrand 
$f_t(u_1,\ldots,u_k)=\boldsymbol{1}_{\left\{t \in   u_1\right\}}\times \ldots \times \boldsymbol{1}_{\left\{t \in   u_k\right\}} 
$   satisfies   the assumption in Theorem \ref{lambda_{k-1}}.  Based on the conclusion of Theorem \ref{lambda_{k-1}} and an elementary calculation via inclusion-exclusion, for $x_0,\ldots,x_t>0$, one has 
\[
\lim_{c\rightarrow\infty}\frac{\bb{P}\pp{\{X_0>c\}\cap  \pp{\bigcup_{i=0}^r\pc{ X_i>cx_i  } } }}
{\bb{P}\pp{X_0>c   }}=\frac{\int_{E^k}   f_0^\alpha(\sbf{u}) \wedge  \pp{\vee_{i=0}^t \frac{{f}_i^\alpha\left(\boldsymbol{u}\right)}{x_i^\alpha}   }\mu^k(d\sbf{u}) }{\int_{E^k} f_0^\alpha(\sbf{u}) \mu^k (d\sbf{u})  }.
\]
The denominator in the last expression  is $1$ in view of \eqref{eq:mu(t in tau*)}.  Define a probability measure $\mathrm{P}_0$ on $E^k$ by
$\mathrm{P}_0(A) = \mu^k(A \cap  \{0\in u_1,\ldots, 0\in u_k\})$. Note that in view of the definition of $\mu$, under the probability measure $\mathrm{P}_0$, the joint law of $(f_i)_{i=0,\ldots,t}$ is equal to that of $ (\Theta_i)_{i=0,\ldots,t}$. Hence 
the numerator above is equal to 
\begin{align*}
\bb{E}\pb{1\wedge \pp{\bigvee_{i=0}^t \Theta_i  x_i^{-\alpha}}}= \bb{P}\pp{ \bigcup_{i=0}^t \pc{  P_\alpha\Theta_i>x_i}}.
\end{align*}

\end{proof}

The process $\pp{\Theta_s}_{s\in \bb{N}_0}$ above is known as the spectral tail process of $\pp{X_t}_{t\in \bb{N}_0}$, while $\pp{P_\alpha\Theta_s}_{s\in \bb{N}_0}$ is known as its tail process \cite[]{basrak2009regularly}. 
Since the time index $t$ is fixed  in the limit theorem above,  the spectral tail process reflects microscopic  characteristics of extreme values of the process $(X_t)_{t\in \bb{N}_0}$. As pointed out in \cite[]{bai2023tail}, the spectral tail process $\pp{\Theta_s}_{s\in \bb{N}_0}$ undergoes a phase transition as $\beta$ increases past the critical value $(k-1)/k$: Setting $\beta_k=k\beta-k+1$, we have $\sum_{s=0}^\infty \Theta_s<\infty$ a.s.\, that is, the intersected renewal $\tau^{(1)}\cap \ldots \cap \tau^{(k)}$ is transient, if and only if $\beta_k<0$. The decay of the spectral tail process as the time index tends to infinity is typically interpreted as an indication of short-range dependence in extremes; see \cite[Section~6.1]{kulik2020heavy}.
Moreover, when $\beta_k<0$, following the argument in \cite[Section 1.3]{bai2023tail}, one can obtain the so-called candidate extremal index as (note that $\Theta_0=1$ always) 
\begin{equation}\label{eq:cei}
\vartheta=\bb{P}(\Theta_s=0  \ \text{for all } s\in \bb{Z}_+)\in (0,1),
\end{equation} 
i.e., the probability of no future intersected renewal after time $0$. In the literature, the candidate extremal index $\vartheta$ is often viewed as  a theoretical prediction of the extremal index $\theta$, a quantity that plays a crucial role in the description of macroscopic (i.e., involving temporal scaling) characteristics of  extreme values of $\pp{X_t}_{t\in \bb{N}_0}$ described below.  See \cite[Section 7.5]{kulik2020heavy} for more information about these extremal indices.

To describe the next macroscopic extreme limit result, let $\widetilde{\mathcal{R}}_\beta$ denote a $\beta$-stable regenerative set, which is the closed range of a $\beta$-stable increasing L\'evy process (subordinator) starting at the origin. Set
$\mathcal{R}_\beta:=\widetilde{\mathcal{R}}_\beta+B_{1-\beta, 1}$, where $B_{1-\beta, 1}$ is a $\operatorname{Beta}(1-\beta, 1)$ distributed variable,  i.e.,  $\mathbb{P}\left(B_{1-\beta, 1} \leq x\right)=x^{1-\beta}, x \in[0,1]$, which is independent from $\widetilde{\mathcal{R}}_\beta$.
Let $\pp{\mathcal{R}_{\beta, i}}_{i \in \bb{Z}_+}$ denote i.i.d.\ copies of $\mathcal{R}_\beta$, and write $\mathcal{R}_{\beta, \boldsymbol{i}}= \mathcal{R}_{\beta, i_1}\cap \ldots \cap \mathcal{R}_{\beta, i_k}$ for $\boldsymbol{i}=(i_1,\ldots,i_k)\in \cl{D}_k$.  Let $\left(\Gamma_i \right)_{i \in \bb{Z}_+}$ be the arrival times of a standard Poisson process on $[0, \infty)$ independent of everything else.
We are now ready to state the following macroscopic extreme limit result which exhibits a phase transition at $\beta_k=0$ as well.

\begin{thm}\label{thm:EVT}
Viewed as  processes indexed by open sets $G$ in $[0,1]$, we have the following convergence of finite-dimensional distributions as $n\rightarrow\infty$:
 \begin{align}\label{ex1}
 \frac{1}{c_n} \bigvee_{t/n\in G} X_t  \overset{f.d.d.}{\longrightarrow} \begin{cases} \mathfrak{C}_{F, k}^{1 / \alpha} \mathcal{M}_{\alpha, \beta, k}(G), & \text { if } \beta_k>0, \\ \mathfrak{C}_{F, k}^{1 / \alpha} \mathcal{M}_\alpha(G), & \text { if } \beta_k \leq 0,\end{cases}
 \end{align}
 where $\mathcal{M}_\alpha$ is an $\alpha$-Fr\'echet random sup measure with   Lebesgue control measure,    $\mathcal{M}_{\alpha, \beta, k}$ is a random sup measure defined by the LePage representation of a multiple extremal integral
 \begin{equation}\label{eq:M alpha beta k}
 \mathcal{M}_{\alpha, \beta, k}(G):=\bigvee_{\boldsymbol{i} \in \mathcal{D}_k}\left[\Gamma_{\boldsymbol{i}}\right]^{-1 / \alpha} \boldsymbol{1}_{\left\{\mathcal{R}_{\beta, \boldsymbol{i}} \cap G \neq \emptyset\right\}},
 \end{equation} 
 and 
 $$
c_n  := \begin{cases} n^{\left(1-\beta_k\right) / \alpha} & \text {if }\beta_k >0,\\ \left(\frac{n(\ln \ln n)^{k-1}}{\ln n}\right)^{1 / \alpha}, & \text {if }\beta_k = 0, \\ \left(n \ln ^{k-1} n\right)^{1 / \alpha}, & \text {if } \beta_k< 0, \end{cases}\qquad \mathfrak{C}_{F, k}  := \begin{cases} \left(\frac{\mathrm{C}_F }{1-\beta}\right)^k, & \text {if }\beta_k >0,\\ \frac{\left(C_F \Gamma(\beta) \Gamma(1-\beta)\right)^k}{k!(k-1)!}, & \text {if }\beta_k = 0, \\ \frac{\vartheta \mathrm{D}_{\beta, k}}{k!(k-1)!}  & \text {if } \beta_k< 0,
\end{cases}
$$ 
where $\mathrm{D}_{\beta,k} = \sum_{q=q_{\beta,k}}^{\infty} (-1)^q \binom{k}{q} (-\beta)^{q-k-1}\in (0,1)$ with $q_{\beta,p}= \min \{ q \in \mathbb{Z}_+ \mid \beta_q < 0 \}$, the constant $\mathrm{C}_F $ is as in \eqref{eq:C_F} and $\vartheta$ is as in \eqref{eq:cei}.
\end{thm}
\begin{proof}[Proof sketch]
The theorem can be proved exactly in a way similar to \cite[Theorem 1.8]{bai2024phase}. We hence provide only a sketch.

  First, it turns out more convenient to work with a triangular-array LePage representation of   $(X_t)_{0=1,\ldots,n}$, $n\in \bb{Z}_+$, described as follows. Recall again that $\mu$ is the law of the delayed stationary renewal $\boldsymbol{\tau}^*$ in \eqref{eq:tau*}. Recall the subspace  $E_n$ in \eqref{eq:E_n},  and denote the restriction of $\mu$ to $E_n$ by $\mu_n$.
  We let the derivative function $\psi=\psi_n$ that appears in the LePage representation \eqref{eq4} be given by
\[
 \psi_n(u)  = w_n, \quad u\in E_n,
\]
where 
\begin{equation}\label{eq:w_n}
w_n:=\mu(E_n) \sim \frac{\mathrm{C}_F }{1-\beta} n^{1-\beta}
\end{equation}
as $n\rightarrow\infty$ (see \cite[Eq.\ (1.7)]{bai2024phase}).
Then the probability measure $m_n$ on $E_n$ given by $ dm_n= \psi_n^{-1}  d\mu_n$ can be regarded as the law of stationary renewal $\boldsymbol{\tau}^*$ conditioned to renew during the time window $\{0,\ldots,n\}$.  Let $\pp{\mathcal{R}_{n,i}}_{i\in \bb{Z}_+}$  be i.i.d.\ renewal sets following the law of $m_n$ which are independent of $\left(\Gamma_i \right)_{i \in \bb{Z}_+}$.

Then, in view of  Corollary \ref{cor:gen int},  relations \eqref{eq:ppp rep}, \eqref{D.2} as well as the symmetry of the integrand, one has
\begin{align}\label{eq:mult reg LePage}
    (X_t)_{t=0,\ldots,n}\EqD \pp{ w_n^{k/\alpha} \bigvee_{\boldsymbol{i} \in \mathcal{D}_{k,<}}\left[\Gamma_{\boldsymbol{i}}\right]^{-1 / \alpha} \boldsymbol{1}_{ \left\{t\in \mathcal{R}_{n, \boldsymbol{i}}   \right\}} },
\end{align}
where $\mathcal{R}_{n, \boldsymbol{i}}:= \mathcal{R}_{n, i_1}\cap \ldots \cap \mathcal{R}_{n, i_k}$ for $\boldsymbol{i}=(i_1,\ldots,i_k)\in \cl{D}_{k,<}$, where the last index set is as in \eqref{D<}.

For the convergence in the super-critical case $\beta_k > 0$, by comparing the representation in \eqref{eq:mult reg LePage} with \eqref{eq:M alpha beta k} and applying a truncation argument, the key   is to establish the convergence of the finite-dimensional joint distributions of the intersected renewals $\{\mathcal{R}_{n,\boldsymbol{i}}\}_{\boldsymbol{i}\in \mathcal{D}_{k,<}}$ toward those of the limiting intersected regenerative sets $\{\mathcal{R}_{\beta,\boldsymbol{i}}\}_{\boldsymbol{i}\in \mathcal{D}_{k,<}}$, where the weak convergence is understood in the Fell topology on the space of closed subsets of $[0,1]$. This is guaranteed by \cite[Theorem~5.4]{samorodnitsky2019extremal}.

% For the convergence in the critical and sub-critical cases that correspond to $\beta_k\le 0$, the key   is to introduce a carefully chosen finite set $\cl{H}(n,K)$ ($K>0$ is a truncation parameter) defined in \cite[Eq.\ (4.5)]{bai2024phase} to replace the infinite set $\cl{D}_k$ in \eqref{eq:mult reg LePage}. In addition, one also approximates the random product $[\Gamma_{\sbf{i}}]=\Gamma_{i_1}\ldots \Gamma_{i_k}$ by the deterministic product $i_1\ldots i_k$ in view of law of large numbers. 
% It is worth pointing out that although the model in \cite{bai2024phase} involves a sum rather than a supremum as in \eqref{eq:mult reg mod}, the Poisson approximation arguments  will not be affected since the proof  explores a sparsity property where only a single term in the $\cl{H}(n,K)$-truncated sum  contributes to the limit; see \cite[Lemma 4.8]{bai2024phase}.
% There are modifications in the form of $\mathcal{M}_{\alpha, \beta, k}$  and the constant $\mathfrak{C}_{F, k} $ when $\beta_k=0$ and $\beta_k<0$  (dropping a factor $1/2$) due to change in the model of $(X_t)_{t\in \bb{N}_0}$ compared to that of \cite{bai2024phase}. The latter can take negative values, while $X_t$ in \eqref{eq:mult reg mod} is always positive. 
% We omit the details of the proof. 

The convergences in the critical and sub-critical cases $\beta_k \le 0$ are established via Poisson convergence of point processes, exploiting the classical correspondence between threshold exceedances and hits of the associated exceedance point process. In the regime $\beta_k \le 0$, the model $(X_t)$ is anticipated to exhibit short-range dependence, in contrast to the long-range dependence regime $\beta_k > 0$.
A typical approach in the literature for handling extremes under short-range dependence is to form blocks of time indices
$
\mathcal{I}_{n,j} = \{(j-1)d_n+1,\ldots, jd_n\}$,
where the block size $d_n \to \infty$ and $\left\lfloor n / d_n\right\rfloor \to \infty$, for $1 \le j \le \left\lfloor n / d_n\right\rfloor$. Limit theorems for the maximum then follow once one establishes weak convergence of the block-level point process
\[
\sum_{1 \le j \le \left\lfloor n / d_n\right\rfloor} \delta_{\, (m_{n,j},\, j/ \left\lfloor n / d_n\right\rfloor)}, \qquad
m_{n,j} := c_n^{-1} \bigvee_{t  \in \mathcal{I}_{n,j}} X_t,
\]
as $n \to \infty$, towards a suitable Poisson point process on $(0,\infty) \times [0,1]$, based on mixing-type conditions to control dependence between blocks; see, e.g., \cite[]{kulik2020heavy,mikoschwintenberger2024}.
However, the model \eqref{eq:mult reg LePage} exhibits a special type of temporal dependence, rendering classical approaches based on mixing-type conditions difficult to apply. Instead, the following approximation strategies are employed.

\begin{enumerate}

\item Use a carefully chosen finite set to replace the infinite set $\mathcal{D}_{k,<}$ in \eqref{eq:mult reg LePage}. In particular, following \cite[Eq.\ (4.5)]{bai2024phase}, we introduce
\begin{equation*}
\mathcal{H}(n,K)
:= \Bigl\{
\boldsymbol{i} = (i_1,\ldots,i_k) \in \mathcal{D}_{k,<}
:\ [\boldsymbol{i}] = i_1 \cdots i_k \le K w_n^k / c_n^{\alpha},\ i_k \le w_n
\Bigr\},
\end{equation*}
where $K > 0$ is a truncation parameter that will eventually be taken to approach $\infty$. This yields the truncated LePage representation
\begin{equation}\label{eq:X_n,t^K}
X_{n,t}^{(K)}
:= w_n^{k/\alpha} \bigvee_{\boldsymbol{i} \in \mathcal{H}(n,K)}
\left[\Gamma_{\boldsymbol{i}}\right]^{-1/\alpha}
\mathbf{1}_{\{\, t \in \mathcal{R}_{n,\boldsymbol{i}}\,\}},
\qquad t = 0,1,\ldots,n.
\end{equation}

\item One can show that, on an event whose probability tends to $1$ as $n \to \infty$, the following set cardinalities satisfy
\[
\Bigl|
\bigl\{
\boldsymbol{i} \in \mathcal{H}(n,K)
:\ \mathcal{R}_{n,\boldsymbol{i}} \cap \mathcal{I}_{n,j} \neq \emptyset
\bigr\}
\Bigr|
\le 1,
\qquad 1 \le j \le n/d_n.
\]
This implies, on the same event, that for $1 \le j \le n/d_n$,
\[
\Bigl\{
c_n^{-1} \bigvee_{t \in \mathcal{I}_{n,j}} X_{n,t}^{(K)}
\Bigr\} \cap (0,\infty)
=
\bigl\{
(w_n^{k/\alpha}/c_n)\,[\Gamma_{\boldsymbol{i}}]^{-1/\alpha}
:\ \boldsymbol{i} \in \mathcal{H}(n,K),\
\mathcal{R}_{n,\boldsymbol{i}} \cap \mathcal{I}_{n,j} \neq \emptyset
\bigr\},
\]
since both sides are simultaneously empty or consist of the same singleton.

\item Replace the random product $[\Gamma_{\boldsymbol{i}}] = \Gamma_{i_1} \cdots \Gamma_{i_k}$ by the deterministic product $i_1 \cdots i_k$, motivated by the law of large numbers $\Gamma_i / i \to 1$ a.s.\ as $i \to \infty$.
\end{enumerate}

One can show that these considerations reduce the proof of convergence in the critical and sub-critical cases  $\beta_k \leq 0$  to establishing the Poisson convergence of the point processes
\[
\overline{\eta}_{n,K}
:= \sum_{1 \le j \le \left\lfloor n / d_n\right\rfloor} \;
   \sum_{\boldsymbol{i} \in \mathcal{H}(n,K)}
   \delta_{\bigl(\,(w_n^{k/\alpha}/c_n)[\boldsymbol{i}]^{-1/\alpha},\; j/\left\lfloor n / d_n\right\rfloor \,\bigr)}
   \,\mathbf{1}_{\{\mathcal{R}_{n,\boldsymbol{i}} \cap \mathcal{I}_{n,j} \neq \emptyset\}}.
\]
When evaluating $\overline{\eta}_{n,K}$ on test sets, one encounters sums of dependent Bernoulli random variables, with the dependence structure governed by properties of the renewal sets $\{\mathcal{R}_{n,\boldsymbol{i}}\}$. The two-moment method of \cite[]{arratia1989two} is then employed to establish the desired Poisson convergence.

It is worth pointing out that although the model in \cite[]{bai2024phase} involves a sum (as well as additional Rademacher variables), rather than a supremum as in \eqref{eq:mult reg LePage}, the arguments can be directly adapted. For instance, the justification of truncation in Item~(i) above, corresponding to \cite[Lemma~4.9]{bai2024phase} in the case $\beta_k < 0$ and \cite[Lemma~5.11]{bai2024phase} in the case $\beta_k = 0$, can be adapted by restricting $\alpha \in (0,1)$ via a suitable power transform, and subsequently bounding the supremum of nonnegative quantities by their sum. Moreover, the second-moment estimates used in the proofs of \cite[Lemmas~4.9 and 5.11]{bai2024phase} to handle Rademacher variables can be replaced by first-moment estimates. In addition, the argument in Item~(ii) above remains unchanged regardless of whether a supremum or a sum is used in the definition of \eqref{eq:X_n,t^K}.

There are also corresponding modifications in the form of $\mathcal{M}_{\alpha,\beta,k}$ and in the constant $\mathfrak{C}_{F,k}$ when $\beta_k = 0$ and $\beta_k < 0$ (dropping a factor $1/2$), due to changes in the model $(X_t)_{t \in \mathbb{N}_0}$ compared to \cite[]{bai2024phase}, where the Rademacher variables are now absent.

\end{proof}

It should be noted that since the limit $\mathcal{M}_{\alpha, \beta, k}$ in \eqref{eq:M alpha beta k} obtained in the case $\beta_k>0$   takes the form of a LePage representation of a multiple extremal integral, leading to a non-Fréchet limit law.

We now pay particular attention to the case $\beta_k<0$, where, following the comments after \cite[Corollary 1.9]{bai2024phase}, we can derive based on Theorem \ref{thm:EVT}  an extremal index
\[
\theta= \mathrm{D}_{\beta,k} \vartheta<\vartheta
\]
for the model, where $\vartheta$ is the candidate extremal index in \eqref{eq:cei}. This discrepancy between the candidate extremal index $\vartheta$ and extremal index $\theta$ reveals an unusual behavior of the model. While it exhibits traits of short-range dependence, such as the vanishing tail spectral process and a Fr\'echet limit in \eqref{ex1}, the discrepancy implies the absence of certain aniti-clustering and mixing conditions \cite[Remark 1.4]{bai2023tail}, which are traits of
long-range dependence.  It is worth noting that \cite[]{bai2024remarkable} provides a more refined analysis on the discrepancy for the double  (i.e., $k=2$ ) stable integral case, and the results obtained here naturally extend to the double extremal integral counterpart.

%\begin{supplement}
%\stitle{Supplement to “Multiple extremal integrals”}
%\sdescription{We provide different definitions of random sup measure and additional discussion of the difference between the definitions. Technical proofs of Theorem 2.6, Proposition 3.5, Lemma 6.3 and some counterexamples are included.}
%\end{supplement}
\begin{center}
 \large\textbf{Acknowledgement}
\end{center}
 
 We sincerely thank the anonymous referees for their valuable comments and suggestions, which have led to significant improvements of the paper. In particular,   the max–infinite divisibility of extremal integrals discussed in Section~\ref{max sec}   was kindly pointed out by a referee.
% and the questions regarding mutiple sufficient moment conditions motivate us to form a new Section \ref{counter exam}. 

\bibliography{reference}
\bibliographystyle{abbrvnat}

\newpage
\appendix
\begin{center}
    \Large\textbf{Appendix}
\end{center}

\section{Different definitions of random sup measure}\label{dd rsp}

We first recall some background    of sup measure theory of \cite[]{vervaat1988random}. Let $E$  be a topological space, $\mathcal{G}=\mathcal{G}(E)$ the collection of open subsets of $E$, $\mathcal{F}=\mathcal{F}(E)$ the collection of closed subsets of $E$, and  $\mathcal{K}=\mathcal{K}(E)$ the collection of compact subsets of $E$. A map $m: \mathcal{G} \rightarrow[0, \infty]$ is a sup measure, if
\begin{equation}\label{sig alg}
    m\left(\bigcup_\alpha G_\alpha\right)=\bigvee_\alpha m\left(G_\alpha\right)
\end{equation}
for an arbitrary collection of open sets $\left(G_\alpha\right)_\alpha$.\ Given a sup measure $m$, its sup derivative, denoted by $d^{\vee} m: E \rightarrow[0, \infty]$, is defined as
$$
d^{\vee} m(t):=\bigwedge_{G \ni t} m(G), \quad t \in E ,
$$
where the infimum $\bigwedge$ is taken over all open set $G$ containing $t$.
The sup derivative of a sup measure is an upper semicontinuous function, that is a function $f$ such that $\{f<t\}$ is open for all $t>0$. Let $\operatorname{SM}(E)$  be the set of all sup measures on $E$.  Every $m \in$ SM has a \emph{canonical extension} to all subsets of $E$, given by
\begin{equation}\label{can ext}
 m(B)=\bigvee_{t \in B}\left(d^{\vee} m\right)(t) = \bigwedge_{G \supset B} m(G), \quad B \subset E.   
\end{equation}
The maxitivity property \eqref{sig alg} holds for the canonical extension when $G_\alpha$'s are replaced by arbitrary subsets of $E$.
The so-called sup vague  topology on $\mathrm{SM}(E)$ is generated by the subbase consisting of the subsets of the forms
\begin{equation}\label{eq7}
  \{m \in \operatorname{SM}(E): m(G)>z\}, \quad G \in \mathcal{G},\, z \in[0, \infty),  
\end{equation}
and
\begin{equation}\label{eq8}
\{m \in \operatorname{SM}(E): m(K)<z\}, \quad K \in \mathcal{K}, \, z \in(0, \infty].    
\end{equation}
  
 A random sup measure in the sense of \cite[]{vervaat1988random} is a random element taking value in $(\mathrm{SM}, \mathcal{B}(\mathrm{SM}))$ with $\mathcal{B}(\mathrm{SM})$ the Borel $\sigma$-algebra induced by the sup-vague topology.  We are now ready to present the definition of the $\alpha$-Fr\'{e}chet random sup measure in this framework  as follows; see \cite[Definition 11.2]{vervaat1988random}.
\begin{defn}\label{def v}
A $\mathrm{SM}(E)$-valued random element $M_\alpha^V$ is said to be an (independently scattered) $\alpha$-Fr\'{e}chet random sup measure on $(E, \mathcal{G})$ with a control measure $\mu$, if the following conditions are satisfied. 
\begin{enumerate}[label=(\roman*)]
\item (independently scattered) For any collection of disjoint sets $A_j \in \mathcal{G}, 1 \leq j \leq n$, $n \in \bb{Z}_+$, the random variables $M_\alpha^V\left(A_j\right), 1 \leq j \leq n$, are independent.
\item ($\alpha$-Fr\'{e}chet) For any $A \in \mathcal{G}$, we have
$$
\mathbb{P}\left\{M_\alpha^V (A) \leq x\right\}=\exp \left\{-\mu(A) x^{-\alpha}\right\},\ x\in (0,\infty),
$$
namely, $M_\alpha^V(A)$ is $\alpha$-Fr\'{e}chet with scale coefficient $(\mu(A))^{1 / \alpha}$, and we understand $M_\alpha^V(A)$ as a random variable taking value $\infty$ with probability one when $\mu(A)=\infty$.
% \item ($\sigma$-maxitive) For arbitrary collection of open sets $\{A_j \}_{j \in J}$, we have that
%     \begin{align}\label{eq stov}
%         M^V_\alpha \left(\bigcup_{j  \in J} A_i\right) = \bigvee_{j  \in J} M^V_\alpha \left( A_i\right).
%     \end{align}
\end{enumerate}
\end{defn}
%\begin{rem}
% Observe that property (iii) remains true if we replace $\mathcal{G}$ with $\mathcal{E}$. This is because the canonical extension given by (\ref{can ext}) leads to the following equality:
%   $$\bigvee_{t \in \cup_{j =1}^\infty A_j} d^{\vee} M^V_\alpha(t) =\bigvee_{j \geq 1} \bigvee_{t \in A_j} d^{\vee} M^V_\alpha(t).$$
% The measurability of $M^V_\alpha(A)$ is not guaranteed for all subsets $A \subset E$. 
 % Thus, $M^V_\alpha$, when defined on  $\mathcal{E}$ , does not necessarily exhibit properties (i) and (ii).\ 

%\end{rem}

Note that with $M_\alpha^V$ given in Definition \ref{def v}, the quantity $M_\alpha^V(A)$ for any $A\subset E$   can be formally defined through the canonical extension \eqref{can ext}, although $M_\alpha^V(A)$ is not guaranteed to be a measurable random variable a priori. Below we show that under some regularity conditions on the space $(E,\cl{E})$, the  measurability of $M_\alpha^V(A)$ follows automatically, and admits a pathwise LePage representation Definition \ref{def v}. We assume the underlying probability space is rich enough to accommodate a randomization as described before  \cite[Theorem 8.17]{kallenberg2021foundations}.

\begin{lem}\label{lem ab}
\cite[Theorem 13.2]{vervaat1988random} Suppose $(E, \mathcal{G})$ is a locally compact second countable Hausdorff (lcscH; see \cite[Appendix A.2]{kallenberg2021foundations}) space. Assume $M_\alpha$ is as in Definition \ref{defn s}. Then  a random sup measure $M^V_\alpha$ in the sense of Definition \ref{def v} exists on the same probability space defining $M_\alpha$, such that for each open set $G \in \mathcal{G}$ (separately), we have \begin{equation}
        M^V_\alpha(G) = M_\alpha(G) \text{ a.s..  }
    \end{equation} 
\end{lem}

Recall a measure defined on the Borel $\sigma$-field of a topological space is said to be locally finite, if every point of the space has a neighborhood with finite measure.
\begin{prop}\label{prop 2 to 3}
 Assume $(E, \mathcal{G})$ is a lcscH  topological space, and $(E, \mathcal{E}, \mu)$ is a locally finite measure space where the $\sigma$-field $\cl{E}$ is the Borel $\sigma$-field generated by $\cl{G}$. Suppose $M^V_\alpha$ is as described in Definition \ref{def v}. Then $M^V_\alpha (A)$ for each $A\subset E$ is a measurable random variable. Furthermore,   for any probability measure $m$ on $(E,\cl{E})$ equivalent to $\mu$ with $\psi=d\mu/dm\in (0,\infty)$ $m$-a.e., on the same probability space that defines $M_\alpha^V$, there exist random variables $(\Gamma_j)_{j\in \bb{Z}_+}$ and $(T_j)_{j\in \bb{Z}_+}$    as described in Definition \ref{defn l}, such that with
  $M_\alpha^L(A):= \bigvee_{j\geq 1} \mathbf{1}_{\left\{T_j \in A\right\}} \psi(T_j)^{1/\alpha}\Gamma_j^{-1 / \alpha} $,   $A\subset E$, on an event of probability $1$, we have
       \begin{align}\label{eqA}
      M^V_\alpha (A) =M^L_\alpha (A) \text{  for any }A \subset E.
     \end{align}

    \end{prop}
 \begin{proof}
  We apply \cite[Corollary 8.18]{kallenberg2021foundations}. Since $E$ is a lcscH space, it is Polish \cite[Lemma A2.4]{kallenberg2021foundations}.  It follows from \cite[Corollary 4.4, Theorem 5.5]{vervaat1988random} that $\text{SM}(E)$ is a compact  metrizable space. Therefore, both $E^{\bb{Z}_+} \times \mathbb{R}^{\bb{Z}_+}$ and $\text{SM}(E)$ are Polish spaces, and thus Borel as well according to \cite[Theorem 1.8]{kallenberg2021foundations}. So the regularity requirement of \cite[Corollary 8.18]{kallenberg2021foundations} on the spaces involved is satisfied.  Note that  the assumptions made also imply that $\mu$ is $\sigma$-finite.

 For any $(t_1, t_2, \ldots) \in E^{\bb{Z}_+}$, $(\gamma_1, \gamma_2, \ldots) \in \mathbb{R}^{\bb{Z}_+}$, one can construct a sup measure in the sense of \cite[Definition 11.2]{vervaat1988random} through the map $g: E^{\bb{Z}_+} \times \mathbb{R}^{\bb{Z}_+} \mapsto \text{SM}(E)$, $(t_j, \gamma_j )_{j \geq 1} \mapsto \bigvee_{j\geq 1}\boldsymbol{1}_{\{t_j \in \cdot\}}\psi(t_j)^{1/\alpha} \gamma_{j}^{-1/\alpha}$.
 Below we check that $g$ is a measurable map as required by \cite[Corollary 8.18]{kallenberg2021foundations}.  Indeed, for any $G \in \mathcal{G}, z \in[0, \infty)$,  the preimage   
 \begin{equation*} 
 \begin{aligned}
  % & g^{-1}\left(\left\{  \bigvee_{j\geq 1}\boldsymbol{1}_{\{t_j \in \cdot\}} \psi(t_j)\gamma_{j}^{-1/\alpha} \in \text{SM}(E)\big| \bigvee_{j\geq 1}\boldsymbol{1}_{\{t_j \in G\}} \psi(t_j)^{1/\alpha} \gamma_{j}^{-1/\alpha} > z \right \}\right)\\
   g^{-1}\pc{ \ \bigvee_{j\geq 1}\boldsymbol{1}_{\{t_j \in G\}} \psi(t_j)^{1/\alpha} \gamma_{j}^{-1/\alpha} > z} 
    =\bigcup_{j\in \bb{Z}_+}\left\{   t_j \in G, \gamma_{j}^{-1/\alpha} > z \right\} \in \mathcal{B}(E)^{\bb{Z}_+} \times \mathcal{B}(\mathbb{R})^{\bb{Z}_+}.
 \end{aligned}
 \end{equation*}
Similarly, we can check for any $K \in \mathcal{K}, z \in(0, \infty]$, the preimage of $\pc{   \bigvee_{j\geq 1}\boldsymbol{1}_{\{t_j \in K\}} \psi(t_j)^{1/\alpha} \gamma_{j}^{-1/\alpha} < z}$ belongs to $ \mathcal{B}(\mathbb{R})^{\bb{Z}_+} \times \mathcal{B}(\mathbb{R})^{\bb{Z}_+}$.  The meausrability of $g$ then follows since sets of the form in \eqref{eq7} and  (\ref{eq8}) generate the Borel $\sigma$-field $\cl{B}(\text{SM})$ \cite[Theorem 11.1]{vervaat1988random}.
 
Next, by \eqref{m and ml1} and Lemma \ref{lem ab}, we have  $$ \pp{ M^V_\alpha (A) }_{A \in \mathcal{G}}\stackrel{d}{=}  \pp{  \wt{M}_\alpha^L(A)  }_{A \in \mathcal{G}}:=  \pp{ \bigvee_{j\geq 1} \mathbf{1}_{\left\{\wt{T}_j \in A\right\}} \psi(\wt{T}_j)^{1/\alpha}\wt{\Gamma}_j^{-1 / \alpha}}_{A \in \mathcal{G}},
$$
where $(\wt{T}_j,\wt{\Gamma}_j)_{j\ge 1}\EqD ({T}_j,{\Gamma}_j)_{j\ge 1}$.
Hence in view of \cite[Theorem 11.5]{vervaat1988random}, both $M^V_\alpha$ and $\wt{M}^L_\alpha=g\pp{(\wt{T}_j,\wt{\Gamma}_j)_{j\ge 1}}$ share the same law as random elements taking value in $\text{SM}(E)$.
As a result of  \cite[Corollary 8.18]{kallenberg2021foundations},  one may  construct $  \pp{ {T}_j, {\Gamma}_j}_{j \ge 1} \stackrel{d}{=}\pp{ \widetilde{T}_j, \widetilde{\Gamma}_j}_{j \ge 1}$ on the same probability space defining $M_\alpha^V$, such that $M^V_\alpha=g\pp{({T}_j,{\Gamma}_j)_{j\ge 1}}$ a.s.\ as random element taking value in $\text{SM}(E)$. This implies on an event of probability $1$, we have
 \begin{align}\label{eq tilde RG}
  M^V_\alpha (G) = M^L_\alpha(G):= \bigvee_{j\geq1} \boldsymbol{1}_{\{ {T}_j \in G\}} \psi( {T}_j)^{1/\alpha} {\Gamma}_j^{-1 / \alpha}\  \text{for all } G \in \mathcal{G}. 
 \end{align}
Set the random function $\varphi(x): = \bigvee_{i\geq 1} \boldsymbol{1}_{\left\{{T}_i= x \right\}} \psi({T}_i)^{1/\alpha} {\Gamma}_i^{-1 / \alpha}$, $x\in E$, so that $M^L_\alpha(A) =\bigvee_{x \in A}\varphi(x)$ for any $A \subset E$.  Combining  (\ref{eq tilde RG}) and the canonical extension relation (\ref{can ext}), the conclusions immediately follow if one shows on an event of probability $1$, we have 
 \begin{equation}\label{key}
\bigwedge_{\mathcal{G}\ni G \supset A} 
M^L_\alpha(G) =M^L_\alpha(A)  \ \text{for all $A \subset E$},
 \end{equation}
noting also that each $M^L_\alpha(A)$ is a measurable random variable.  Indeed, the relation  (\ref{key}) can be expressed as 
 \begin{equation}\label{eq sat} 
 \bigwedge_{\mathcal{G}\ni G \supset A} \bigvee_{x\in G} \varphi(x) = \bigvee_{x\in A} \varphi(x) \text{ for all } A \subset E,
 \end{equation}
  which follows  if $\varphi$ is an upper semicontinuous (usc) function on an event of probability $1$; see \cite[Definition 1.1, Theorem 2.5 (c)]{vervaat1988random}.   
 Note that for each $m\in\bb{Z}_+$, one can write
 \begin{align*}
   \varphi(x) & = \left( \bigvee_{1\leq i\leq m} \boldsymbol{1}_{\left\{{T}_i =x \right\}} \psi({T}_i)^{1/\alpha} {\Gamma}_i^{-1 / \alpha}\right) \bigvee \left(\bigvee_{i\geq m+1} \boldsymbol{1}_{\left\{{T}_i =x\right\}} \psi({T}_i)^{1/\alpha} {\Gamma}_i^{-1 / \alpha} \right) \\
   &=:\varphi_{1, m}(x)\bigvee \varphi_{2, m}(x),
\end{align*}
where the random function $\varphi_{1, m}$ is usc as a finite supremum of usc functions (recall that a singleton set $\{T_i\}$ is a closed set  in a Hausdorff space) for every $m\in \bb{Z}_+$. 
To establish the   upper semicontinuity of $\varphi$,  it suffices to show that on an event of probability 1, for any $x \in E$, there exists an open neighborhood $U$ of $x$ such that $\sup_{x \in U} |\varphi_{1,m}(x) - \varphi(x)| \rightarrow 0 $ as $m \rightarrow \infty$. Since $E$ is  second countable and $\mu$ is locally finite, one can express $E = \bigcup_{n=1}^\infty U_n$, where $U_n \in \mathcal{G}$  and $\mu(U_n)<\infty$ for each $n \in \mathbb{Z}_+$. Moreover, since $|\varphi_{1,m}(x) - \varphi(x)| \leq \varphi_{2, m}(x)$, it suffices to show that for each $n \in \mathbb{N}$, on an event of probability 1, we have
\begin{equation}\label{eq 10}
    \sup_{x\in U_n}\varphi_{2, m}(x) = \bigvee_{i \geq m+1} \mathbf{1}_{\left\{T_i\in U_n\right\}} \psi\left(T_i\right)^{1 / \alpha} \Gamma_i^{-1 / \alpha} \rightarrow 0,  
\end{equation}
as $m \rightarrow \infty$. 
Choose any $r > \alpha$, and note that 
\begin{equation}\label{leq sum}
   \left( \bigvee_{i \geq m+1} \mathbf{1}_{\left\{T_i\in U_n\right\}} \psi\left(T_i\right)^{1 /\alpha } \Gamma_i^{-1 / \alpha} \right)^r\leq \sum_{i = m+1}^\infty \mathbf{1}_{\left\{T_i\in U_n\right\}} \psi\left(T_i\right)^{1 /(\alpha/r)} \Gamma_i^{-1 / (\alpha/r)}.
\end{equation}
By \cite[Theorem 3.10.1]{samorodnitsky1991construction},   the condition 
 $$
\int_E\left(\mathbf{1}_{\left\{x \in U_n\right\}} \psi(x)^{1 /( \alpha / r)}\right)^{\alpha / r} m(d x)=\mu\left(U_n\right)<\infty
$$
ensures that as the LePage representation of a positive $(\alpha/r)$-stable random variable, the series
 $$
\sum_{i = 1}^\infty \mathbf{1}_{\left\{T_i \in U_n\right\}} \psi\left(T_i\right)^{1 /(\alpha / r)} \Gamma_i^{-1 /(\alpha / r)}<\infty \quad \text { a.s.. }
$$
Hence, the right hand side of (\ref{leq sum}) converges to zero a.s.\ as $m \rightarrow \infty$ and the desired convergence in (\ref{eq 10}) follows.

\end{proof}

Combining Lemma \ref{lem ab} and Proposition \ref{prop 2 to 3} yields the following result.
\begin{cor}
    Under the assumptions %of  Lemma \ref{lem ab} and 
    Proposition \ref{prop 2 to 3},  suppose   $M_\alpha$ is as in Definition \ref{defn s}. Then for any probability measure $m$ on $(E,\cl{E})$ equivalent to $\mu$, on the same probability space defining $M_\alpha$, there exists a LePage random sup measure $M_\alpha^L$ as described in Proposition \ref{prop 2 to 3}, such that
    for each open set $G \in \mathcal{G}$ (separately), we have
    \begin{equation}\label{as}
        M_\alpha(G) = M_\alpha^L(G) \text{ a.s..  }
    \end{equation} 
\end{cor}

% \begin{rem}
%       In view of \eqref{eq7} and \eqref{eq8}, the evaluated random sup measure $M^V_\alpha(A)$ is a measurable random variable when $A$ is an open or a compact set in $E$.   
%   Proposition \ref{prop 2 to 3} below  locally compact second countable Hausdorff (lcscH) show that $M^V_\alpha(A)$ is measurable for any $A \in \mathcal{E}$ if $(E, \mathcal{G})$ is a  lcscH  space.
% \end{rem}

\section{Auxiliary results for product random sup measures}
\subsection{Product random sup measure on unions of rectangles}
 % Suppose $E$ is a Borel space, meaning there exists a bijection $\iota:E \leftrightarrow S$ such that both $\iota$ and $\iota^{-1}$ are measurable and $S$ is a Borel subset in $[0,1]$.  
 In this section, we follow the notation in Section \ref{S 2.2}. 

\begin{lem}\label{d ind}
  Suppose $A_i, B_j \in \mathcal{C}_k$, $1 \leq i \leq m$, $1 \leq j \leq n$, $m,n\in \bb{Z}_+$, and  $\bigcup_{i=1}^m A_i = \bigcup_{j = 1}^n B_j$. Then we have $\bigvee_{i=1}^m M_\alpha^{(k)}(A_i) = \bigvee_{j=1}^n M_\alpha^{(k)}(B_j)$ a.s..
\end{lem}
 \begin{proof}
First, we show that for any finite collection of disjoint sets $S^{(1)}, S^{(2)}, \ldots, S^{(N)}\in \mathcal{C}_k, N \in \mathbb{Z}_{+}$, satisfying $\bigcup_{j=1}^N S^{(j)} \in \mathcal{C}_k$, that is, $\{S^{(1)}, S^{(2)}, \ldots, S^{(N)} \}$ forming a partition of an off-diagonal rectangle, the following holds: 
\begin{equation}\label{eq 11}
    M_\alpha^{(k)}\left(\bigcup_{j=1}^N S^{(j)}\right)=\bigvee_{j=1}^N M_\alpha^{(k)}\left(S^{(j)}\right) \text{ a.s..}
\end{equation}
Define $S^{(j)}=S_1^{(j)} \times \cdots \times S_k^{(j)}$, for $1 \leq j \leq N$, where $S_t^{(j)} \in \mathcal{E}$ are  disjoint with respect to $t=1,\ldots,k$. We claim that $\bigcup_{j=1}^N S^{(j)} \in \mathcal{C}_k$ can be decomposed into a  collection of off-diagonal rectangles $L_t$,  $t=1,\ldots,\widetilde{N}$, $\widetilde{N}\in \mathbb{Z}_+$, such that $\bigcup_{j=1}^N S^{(j)}=\bigcup_{t=1}^{\widetilde{N}} L^{(t)}$, and $ L^{(t)} \in \mathcal{C}_k$ are   disjoint, $ 1\leq t\leq \widetilde{N}$.  Furthermore,  for any  pair  $L^{(t_1)}= L_1^{(t_1)}\times  \cdots \times L_k^{(t_1)}  \in \mathcal{C}_k$ and $L^{(t_2)}=L_1^{(t_2)} \times \cdots \times L_k^{(t_2)} \in \mathcal{C}_k$, $t_1 \neq t_2$, either one of the following conditions holds:

\begin{enumerate}[label=(\alph*)]
\item The sets $L_1^{(t_1)},  \cdots, L_k^{(t_1)}, L_1^{(t_2)},  \cdots, L_k^{(t_2)} $ are all disjoint. 
\item There exists exactly one $u\in\{1,\ldots,k\}$ such that $L_u^{(t_1)} \cap L_u^{(t_2)} = \emptyset $, and $L_d^{(t_1)} = L_d^{(t_2)}  $ for all $d \neq u, 1 \leq d \leq k$.
\end{enumerate}

We now describe the procedure for forming the collection $\left\{L^{(t)}, 1 \leq  t \leq \widetilde{N}\right\}$. 
%This involves the following two iterative steps for each $\left\{S_{\ell}^{(j)}, 1 \leq j \leq N\right\}, 1 \leq \ell \leq k$ :
For each coordinate index $\ell\in \{1,\ldots,k\}$, do the following: 
\begin{enumerate}
    \item     For any subset $I \subset\{1,2, \ldots, N\}$, set
$$
R_{\ell, I}=\pp{\bigcap_{i \in I} S_\ell^{(j)}} \setminus  \pp{\bigcup_{j \notin I} S_\ell^{(j)}} .
$$
\item   All non-empty $R_{\ell, I}$ 's form a partition of the union $\bigcup_{j=1}^N S_{\ell}^{(j)}$. The resulting partition is denoted as
$$
\mathcal{P}_\ell = \left\{R_{\ell, I}\mid I \subset\{1,2, \ldots, N\}, R_{\ell,I} \neq \emptyset\right\}.
$$
\end{enumerate}
Once these steps are completed for each coordinate index $\ell\in \{1,\ldots,k\}$, construct the collection:
$$
\mathcal{J}=\left\{R_1 \times \cdots \times R_k \mid R_{\ell} \in \mathcal{P}_{\ell}, 1 \leq \ell \leq k\right\}.
$$
Lastly, enumerate the elements of $\mathcal{J}$ as $\mathcal{J}=\left\{L^{(t)}\mid 1 \leq t \leq \widetilde{N}\right\}$, where $\widetilde{N} \in \mathbb{Z}_{+}$.

Since   $\bigcup_{i=1}^N S^{(j)} \in \mathcal{C}_k$ by  assumption,  we can write $\bigcup_{j=1}^N S^{(j)}=U_1 \times \cdots \times U_k$, where $U_{\ell}=\bigcup_{R_{\ell} \in \mathcal{P}_{\ell}} R_{\ell}$, $1\leq \ell \leq k $, are disjoint. By $\sigma$-maxitivity of $M_\alpha$, we have a.s. 
\begin{align}\label{eq 13}
M_\alpha^{(k)}\left(\bigcup_{j=1}^N S^{(j)}\right)& =M_\alpha\left(U_1\right) \times \cdots \times M_\alpha\left(U_k\right) \notag \\
   & = \bigvee_{R_{\ell} \in \mathcal{P}_{\ell},1 \leq \ell \leq k} M_\alpha\left(R_1\right) \times \cdots \times M_\alpha\left(R_k\right).
   \end{align}
On the other hand, we can express $S_\ell^{(j)} = \bigcup_{R_\ell \in A(\ell,j)}  R_\ell$, $1 \leq \ell \leq k$, $1 \leq j \leq N$, for some unique $A(\ell,j) \subset\mathcal{P}_\ell, \ell=1, \ldots, k$, where $\cup_{1\le j\le N}A(\ell,j)=\cl{P}_\ell$.  Consequently, we have $$S^{(j)}=\bigcup_{R_{\ell} \in A(\ell, j), 1 \leq \ell \leq k} R_1 \times \cdots \times R_k.$$  Again by  $\sigma$-maxitivity of $M_\alpha$, we have a.s.
\begin{align}\label{eq 14}
    \bigvee_{j=1}^N M_\alpha^{(k)}\left(  S^{(j)}\right) & =  \bigvee_{j=1}^N M_\alpha( S_1^{(j)}) \times \cdots \times M_\alpha( S_k^{(j)}) \notag \\
    & =    \bigvee_{j=1}^N M_\alpha\pp{\bigcup_{R_1 \in A(1,j)}  R_1} \times \cdots \times  M_\alpha \pp{\bigcup_{R_k \in A(k,j)}  R_k} \notag \\
    & = \bigvee_{\substack{R_\ell \in A(\ell,j),1 \leq \ell \leq k,\\ 1\leq  j \leq N}} M_\alpha(R_1)\times \cdots \times M_\alpha(R_k). 
\end{align}
Comparing \eqref{eq 13} and \eqref{eq 14},  the conclusion \eqref{eq 11} follows by noticing $ \cl{J}=\{R_1\times \ldots \times R_k \mid R_\ell\in A(\ell,j),  \ 1\le\ell\le k,  \ 1\le j\le N  \}$.
% Due to the disjointness of $S^{(j)}$'s and $R_{1}\times \cdots \times R_{k}$'s,  each $R_{1} \times \cdots \times R_{k}$ in $\mathcal{J}$ belongs to the collection $\{R_{1} \times \cdots \times R_{k} \mid R_{\ell} \in A(\ell, j), 1 \leq \ell \leq k \}$ for a unique $j$, $1\leq j \leq N$. Hence, we obtain $M_\alpha^{(k)}\left(\bigcup_{j=1}^N S^{(j)}\right) =  \bigvee_{j=1}^N M_\alpha^{(k)}\left( S^{(j)}\right)$ a.s.

We now proceed to prove the desirable relation: $\bigvee_{i=1}^m M_\alpha^{(k)}\left(A_i\right)=\bigvee_{j=1}^n M_\alpha^{(k)}\left(B_j\right)$ a.s., where $\bigcup_{i=1}^m A_i = \bigcup_{j = 1}^n B_j\in \cl{F}_k$.  One may assume without loss of generality that $A_i$'s are disjoint and $B_j$'s are disjoint; otherwise, properly partition each $A_i$ and $B_j$  further into disjoint rectangles and apply  \eqref{eq 11}.  Then, note that
$$
A_i=\bigcup_{j=1}^n\left(A_i \cap B_j\right), \quad 1 \leq i \leq m, \quad \text { and } \quad B_j=\bigcup_{i=1}^m\left(A_i \cap B_j\right), \quad 1 \leq j \leq n,
$$
where each $A_i \cap B_j \in \mathcal{C}_k$. Applying \eqref{eq 11}, we have for any $i\in \{1,\ldots,m\}$ that
$$
M_\alpha^{(k)}\left(A_i\right)=M_\alpha^{(k)}\left(\bigcup_{j=1}^n\left(A_i \cap B_j\right)\right)=\bigvee_{j=1}^n M_\alpha^{(k)}\left(A_i \cap B_j\right)\text { a.s..}
$$
Similarly, for any $j\in \{1,\ldots,n\}$,
$$
M_\alpha^{(k)}\left(B_j\right)=M_\alpha^{(k)}\left(\bigcup_{i=1}^m\left(A_i \cap B_j\right)\right)=\bigvee_{i=1}^m M_\alpha^{(k)}\left(A_i \cap B_j\right)  \text { a.s..}
$$
Combining these results completes the proof.
\end{proof}

\subsection{Proof of Theorem \ref{claim 5.3} }\label{Pf 2.6}

The next results will be useful in the proof of  Theorem \ref{thm sf case}. Recall we have assumed throughout that $(E,\cl{E})$ is a Borel space.

\begin{lem}\label{lem2.1}
 Suppose $\mu$ is a finite measure on $\cl{E}$. For any $A\in \mathcal{E}^{(k)}$, there exists a sequence of sets $(A_n)_{n \in \bb{Z}_+}$ with $A_n \in \mathcal{F}_k$ such that $\mu^k(A_n \Delta A) \rightarrow 0$ as $n \rightarrow \infty$.
\end{lem}
\begin{proof}
We first show that the conclusion holds for $E = \mathbb{R}$. %Observe that the off-diagonal space $\pp{D^{(k)}}^c = \left\{\left(x_1, \ldots, x_k\right) \in E^k: x_i\neq x_j \text{ for all } 1 \leq i < j \leq  k  \right \}$ is an open set when $E = \mathbb{R}$. 
Since $\mu^k$ is a finite measure on the metric space $\mathbb{R}^k$, it is outer regular \cite[Lemma 1.36]{kallenberg2021foundations}. Furthermore, since $\mathbb{R}^{(k)} $   is an open subset of $\mathbb{R}^k$, for any $A \in \cl{E}^{(k)}$  and any $n\in \bb{Z}_+$, there is    $O_n \in \cl{E}^{(k)}$ open in $\bb{R}^{k}$ such that $A \subset O_n$ and
$$
\mu^k\left(O_n \backslash A\right)<1 / n.
$$
 Now since one can express $O_n  = \bigcup_{j\in \bb{Z}_+} R_{j,n}$, where each $R_{j,n} \in \mathcal{C}_0$ (it suffices to consider open rectangles),   for each $n\in\bb{Z}_+$, there exists $N_n\in \bb{Z}_+$,  such that $$\mu^k \pp{ O_n \backslash A_n}  < 1/n,$$ where $A_n= \bigcup_j^{N_n} R_j\in \cl{F}_k$. 
It follows from  a triangle inequality for set symmetric differences  and the inequalities above that $$\mu^k\pp{A \Delta A_n} \leq \mu^k\left(A \Delta O_n\right) +\mu^k\left( O_n \Delta A_n\right) < 2/n.$$ So the conclusion holds for $E=\bb{R}$.  

Now, suppose $E$ is a Borel space. By definition, there exists a bijection $\iota: E\leftrightarrow S$ for a Borel subset  $S$ of $[0,1]$ equipped with   $\sigma$-field $\cl{S}=\cl{B}(\bb{R})\cap S$, such that both $\iota$ and $\iota^{-1}$ are measurable with respect to $\cl{E}$ and $\cl{S}$. Define a map $I: E^{(k)} \rightarrow S^{(k)}$ as $I(x_1, \ldots, x_k) = (\iota(x_1), \ldots, \iota(x_k))$, $x_1,\ldots,x_k\in E$. One can verify that $I$ is a  bijection between $\pp{E^{(k)},\mathcal{E}^k} $  and  $\pp{S^{(k)}, \cl{S}^{(k)}}$ such that both $I$ and $I^{-1}$ are measurable, where $\cl{S}^{(k)}$ is the off-diagonal $\sigma$-field  defined similarly as $\cl{E}^{(k)}$. Next, define the measure $\mu_S^k(\cdot): = \mu^k(I^{-1}(\cdot))$ on $\cl{S}^{(k)}$, which is also a finite measure given $\mu$ is finite.  Fix a set $A \in \mathcal{E}^{(k)}$, and note that $I(A)$ is in  $\cl{S}^{(k)}$, and hence also in the off-diagonal $\sigma$-field of $\bb{R}^k$.  So
by the result for $E = \mathbb{R}$ above,  there exists a sequence of $(\widetilde{A}_n)_{n \in \bb{Z}_+}$, where each $\widetilde{A}_n$ is a finite union of off-diagonal rectangles in $\cl{S}^{(k)}$ (obtained by restricting  the off-diagonal rectangles found in $\bb{R}^{(k)}$ to $S^{(k)}$), such that $\mu_S^k(I(A) \Delta \widetilde{A}_n) \rightarrow 0$ as $n \rightarrow \infty$. Hence, it follows that $\mu^k(A \Delta A_n)\rightarrow 0$ as $n \rightarrow \infty$, where each $A_n = I^{-1}(\widetilde{A}_n)$. To conclude the proof, it suffices to note that the measurable inverse $I^{-1}: S^{(k)} \rightarrow E^{(k)}$ is given by $I^{-1}(s_1, \ldots, s_k) = (\iota^{-1}(s_1), \ldots, \iota^{-1}(s_k))$ which preserves rectangles, and therefore $A_n\in \mathcal{F}_k$ for each $n\in \bb{Z}_+$.  
\end{proof}
Recall that $ \mathcal{D}_{k}=\left\{\boldsymbol{j} = (j_1,\ldots
 ,j_k)\in \bb{Z}_+^k\mid \text { all } j_1, \ldots, j_k \text { are distinct}\right\}$. Throughout the sequences $(\Gamma_j)_{j\in \bb{Z}_+}$, $(T_j)_{j\in \bb{Z}_+}$, the probability measure $m$ and $\psi=d\mu/dm$ are as in Definition \ref{defn l}.

\begin{lem}\label{lem 2.3}
Suppose $\alpha >0$. 
For any  $\gamma\in (0,\alpha)$ and $r>\alpha$, there is a constant $c>0$  which depends on $(r,\gamma, k)$, such that 
\begin{equation}\label{ta 2.3}
\left\| \bigvee_{\boldsymbol{j} \in \mathcal{D}_k}\left[\Gamma_{\boldsymbol{j}}\right]^{-1 / \alpha}  g\pp{T_{\boldsymbol{j}} }  \right\|_{\gamma }    \leq c \|g(T_1,\ldots,T_k)\|_r
\end{equation}
 for any measurable function  $g:E^{k}\mapsto[0,\infty]$ that vanishes on the diagonals. Here,  $\|\cdot\|_\gamma=\left(\mathbb{E}|\cdot|^\gamma\right)^{1 / \gamma}$ is interpreted as the usual $L^\gamma$ quasi-norm when $0<\gamma<1$, and the (quasi-)norm  value  may be $+\infty$.
\end{lem}
\begin{proof}

 % Let $g:E^{k}\mapsto[0,\infty]$ be a measurable function that vanishes on the diagonals. 
 We want to apply \cite[Corollary 2.1]{samorodnitsky1991construction}, which states a similar inequality  with the maximum $\bigvee$ replaced by a summation,  the index set $\cl{D}_k$ replaced by $\cl{D}_{k,<}$, and $\alpha$ restricted to $(0,1)$.   To do so, we introduce an additional auxiliary parameter $\beta>\alpha$.
 Define also $\mathcal{D}_{k,<}^\pi=\left\{\left(t_{\pi(1)}, \ldots, t_{\pi(k)}\right) \in \bb{Z}_+^k\mid t_1<\ldots<t_k\right\}$, where $\pi \in \Theta_k$, and $\Theta_k$ denotes the set of all permutations of $\{1,2,\ldots,k\}$. Note that $\mathcal{D}_{k} = \bigcup_{\pi \in \Theta_k}\mathcal{D}_{k,<}^\pi$.  Then, 
\begin{align*}
   \mathbb{E} \pb{\pp{  \bigvee_{\boldsymbol{j} \in \mathcal{D}_k}\left[\Gamma_{\boldsymbol{j}}\right]^{-1 / \alpha}  g\pp{T_{\boldsymbol{j}} } }^\gamma}  &= 
   \bb{E}\pb{\pp{ \bigvee_{\pi \in \Theta_k}\bigvee_{\boldsymbol{j} \in \mathcal{D}^\pi_{k, <}} \left[\Gamma_{\boldsymbol{j}}\right]^{-1 / (\alpha/ \beta)}    g^\beta\pp{T_{\boldsymbol{j}}}}^{\gamma/\beta}}  
   \\
   &\leq 
   \bb{E}\pb{\pp{ \sum_{\pi \in \Theta_k}\sum_{\boldsymbol{j} \in \mathcal{D}^\pi_{k, <}} \left[\Gamma_{\boldsymbol{j}}\right]^{-1 / (\alpha/ \beta)}    g^\beta\pp{T_{\boldsymbol{j}}}}^{\gamma/\beta}}\\
   & \leq \sum_{\pi \in \Theta_k} \mathbb{E}\pb{\left( \sum_{\boldsymbol{j} \in \mathcal{D}^\pi_{k, <}}\left[\Gamma_{\boldsymbol{j}}\right]^{-1 / (\alpha/\beta)}  g^\beta\pp{T_{\boldsymbol{j}} }  \right)^{\gamma/\beta }} 
\end{align*} 
where the last inequality uses the quasi-triangle inequality since $0< \gamma/\beta<1$. 
Then, for each $\pi$-term, it follows from \cite[Corollary 2.1]{samorodnitsky1991construction} that there is a constant $c>0$ which depends only on $(\gamma,r, k)$ such that
\begin{equation*} 
  % \mathbb{E} \left[\left(\bigvee_{\boldsymbol{j} \in \mathcal{D}_{k,<}^\pi}\left[\Gamma_{\boldsymbol{j}}\right]^{-1 / \alpha}  g\pp{T_{\boldsymbol{j}} }  \right)^{\gamma }\right] & = \mathbb{E}\pb{\left( \bigvee_{\boldsymbol{j} \in \mathcal{D}^\pi_{k, <}}\left[\Gamma_{\boldsymbol{j}}\right]^{-1 / (\alpha/\beta)}  g^\beta\pp{T_{\boldsymbol{j}} }  \right)^{\gamma/\beta }} \\
  \mathbb{E}\pb{\left( \sum_{\boldsymbol{j} \in \mathcal{D}^\pi_{k, <}}\left[\Gamma_{\boldsymbol{j}}\right]^{-1 / (\alpha/\beta)}  g^\beta\pp{T_{\boldsymbol{j}} }  \right)^{\gamma/\beta }}  
     \leq c  \pc{\mathbb{E}\pb{g(T_1,\ldots,T_k)^r}}^{\gamma/r}.
\end{equation*}
     % where $g^\beta$ corresponds to their function $g$, our $2 \alpha / \beta$ takes the role of their parameter $\alpha$, our $2 \gamma / \beta$ corresponds to their $\gamma$, and $2 r / \beta$ to their $r$. Note that the parameter $\beta$ is introduced to ensure $2\alpha/\beta\in (0,2)$, which is necessary for applying their result.
The proof is then concluded by combining the above.
\end{proof}

% The next lemma is deduced from the proof (the claim in Part B) of \cite[Theorem 3.1]{samorodnitsky1991construction} and Minkowski's inequality.

\begin{lem}\label{lem 2.4}
    Suppose $\mu$ is a finite measure on $\cl{E}$, and $\alpha >0$ and $0<\gamma  < \alpha /k$. Suppose $A \in \mathcal{E}^{(k)}$, $k\in \bb{Z}_+$,  and a sequence  $A_n \in \mathcal{F}_k$, $n\in \bb{Z}_+$,  satisfy $\mu^k(A_n \Delta A) \rightarrow 0$ as $n\rightarrow\infty$ (see Lemma \ref{lem2.1}).  For $k \geq 1$, let 
    \begin{equation}\label{s_ka}
        \widehat{S}_{k,\alpha} (A) = \bigvee_{\boldsymbol{j} \in \mathcal{D}_{k}}\left[\Gamma_{\boldsymbol{j}}\right]^{-1 / \alpha} \left[\psi(T_{\boldsymbol{j}})\right]^{1/\alpha}   \mathbf{1}_{\left\{T_{\boldsymbol{j}} \in A \right\}}, \quad A \in \mathcal{E}^{(k)}.
    \end{equation}  
    Then we have  
\begin{equation} \label{sbound}
\begin{aligned}
\mathbb{E} \left| \widehat{S}_{k,\alpha} (A_n)  - \widehat{S}_{k,\alpha} (A) \right |^{\gamma} \leq  \mathbb{E} \left| \widehat{S}_{k,\alpha} (A_n\Delta A) \right|^{\gamma} \rightarrow 0,
\end{aligned}
\end{equation}
 as $n \rightarrow \infty$.
\end{lem}
\begin{proof}
We first point out that under the assumption $\mu(E)<\infty$, both $\widehat{S}_{k,\alpha} (A) < \infty$ and  $\widehat{S}_{k,\alpha} (A_n) < \infty$  a.s., $n \in \mathbb{Z}_+$.  
Indeed, taking the first relation as an example, we have for $r>\alpha$ that
 
\begin{align}
\pp{\widehat{S}_{k,\alpha} (A)}^{r}&\le \sum_{\boldsymbol{j} \in \mathcal{D}_{k}}\left[\Gamma_{\boldsymbol{j}}\right]^{-1 / (\alpha/r)} \left[\psi(T_{\boldsymbol{j}})\right]^{1/(\alpha/r)}   \mathbf{1}_{\left\{T_{\boldsymbol{j}} \in A \right\}}\label{eq:bound S_k alpha}=:S_{k,\alpha/r}^*(A)\\
&\le \pp{\sum_{j\in\bb{Z}_+} \Gamma_j^{1/(\alpha/r)} \psi(T_j)^{\alpha/r}}^k. \notag
\end{align}
The last expression is a positive $(\alpha/r)$-stable random variable raised to power $k$, and hence finite a.s. (see the argument below \eqref{leq sum}).
Then, the first inequality in \eqref{sbound} follows from the inequality $|\bigvee_{i \in \bb{Z}_+} a_i -  \bigvee_{i \in \bb{Z}_+} b_i| \leq \bigvee_{i \in \bb{Z}_+}|a_i - b_i|$ for real-valued sequences $\left\{a_i\right\}_{i \in \bb{Z}_+}$ and $\left\{b_i\right\}_{i \in \bb{Z}_+}$ and the relation $|\mathbf{1}_{\left\{T_{\boldsymbol{j}} \in A \right\}}-\mathbf{1}_{\left\{T_{\boldsymbol{j}} \in A_n \right\}}|=\mathbf{1}_{\left\{T_{\boldsymbol{j}} \in A\Delta A_n \right\}}$.

Now with $S_{k,\alpha/r}^*$ in \eqref{eq:bound S_k alpha} above, together with \cite[part B of the proof of Theorem 3.1]{samorodnitsky1991construction}
and writing $\mathcal{D}_k=\bigcup_{\pi \in \Theta_k} \mathcal{D}_{k,<}^\pi$, we apply quasi-triangle inequality for $\|\cdot\|_{\gamma/r}, 0<\gamma/r <1$ as in the proof of Lemma \ref{lem 2.3}. This yields  $\mathbb{E}|S_{k,\alpha/r}^*(A\Delta A_n) |^{\gamma/r} \rightarrow 0$  as $n \rightarrow \infty$. So the convergence in  \eqref{sbound} follows  from \eqref{eq:bound S_k alpha} with $A$ replaced by $A\Delta A_n$.

\end{proof}

\begin{proof}[Proof of Theorem \ref{claim 5.3}]
 By Lemma \ref{lem2.1}, for any $A \in \mathcal{E}^{(k)}$, $k \geq 1$,  there is a sequence of sets $(A_n)_{n \in  \bb{Z}_+}$, where $A_n \in \mathcal{F}_k$ for all $n \in \bb{Z}_+$, such that $\mu^k(A_n \Delta A) \rightarrow 0$ as $n \rightarrow \infty$.

 \medskip
 
\noindent\textbf{Step 1}: Suppose first $\gamma \in (0,\alpha/k)$. We shall show that $M_\alpha^{(k)}(A):=\lim _{n \rightarrow \infty} M_\alpha^{(k)}\left(A_n\right)$ exists in $L^\gamma$, and $M_\alpha^{(k)}$ is $\sigma$-maxitive. 

In view of \eqref{m and ml1} and the definition of $M^{(k)}$ on $\cl{F}_k$, we have 
\begin{equation}\label{eq:eqd on F_k}
\left(M_\alpha^{(k)}(A)\right)_{A \in \cl{F}_k,\, {k \in \mathbb{Z}_+}} \stackrel{d}{=} \left(\widehat{S}_{k, \alpha}(A)\right)_{A \in \cl{F}_k,\, {k \in \mathbb{Z}_+}},
\end{equation}
where  $\widehat{S}_{k,\alpha}$ is as defined in (\ref{s_ka}). 
Hence if $\left(\widehat{S}_{k, \alpha}\left(A_n\right)\right)_{n \in \bb{Z}_+}$ forms a Cauchy sequence in $L^\gamma$, so does $\left(M_\alpha^{(k)}\left(A_n\right)\right)_{n \in \bb{Z}_+}$. By Lemma \ref{lem 2.4}, 
 %the right hand side of inequality (\ref{sbound}) tends to zero as $n \rightarrow \infty$. Consequently, 
the sequence $\widehat{S}_{k,\alpha} (A_n)  \rightarrow \widehat{S}_{k,\alpha} (A)$ in $L^\gamma$ as $n \rightarrow \infty$,
  and hence forms a Cauchy sequence in $L^\gamma$.  Thus,  $M_\alpha^{(k)}(A) : = \lim _{n \rightarrow \infty} M_\alpha^{(k)}\left(A_n\right)$ exists in $L^\gamma$  (uniquely a.s.). 
%for $\gamma \in (0,\alpha/k)$. 
Moreover, as a consequence of the $L^\gamma$ approximation and \eqref{eq:eqd on F_k}, we have 
\begin{equation}\label{eq:M^k S_k eq d}
\left(M_\alpha^{(k)}(A)\right)_{A \in \mathcal{E}^{(k)},{ k \in \mathbb{Z}_+}} \stackrel{d}{=} \left(\widehat{S}_{k, \alpha}(A)\right)_{A \in \mathcal{E}^{(k)},{ k \in \mathbb{Z}_+}}.
\end{equation}
It also follows that     $M_\alpha^{(k)}$ is $\sigma$-maxitive because $\widehat{S}_{k, \alpha}$ is so in view of its definition \eqref{s_ka}.

\medskip
\noindent\textbf{Step 2}: Suppose now $\gamma \in (0,\alpha)$; note that the range of $\gamma$ here is larger than that in Step 1. We shall show that for any  $r > \alpha$, there is a constant $c >0$  depending on $(r, \gamma)$ such that
\begin{equation}\label{imp}
\left\|\widehat{S}_{k, \alpha}\left(A\right)\right\|_\gamma  = \left\|M_\alpha^{(k)}\left(A\right)\right\|_\gamma  \leq c \left( \mu^k\left(A\right)\right)^{1/r},
\end{equation}
for any $A \in \mathcal{E}^{(k)}$. If this holds, we can conclude by uniform integrability that $M_\alpha^{(k)}(A)=\lim _{n \rightarrow \infty} M_\alpha^{(k)}\left(A_n\right)$ in $L^\gamma$ for $\gamma \in (0,\alpha)$ and $\alpha >0$.

Using \eqref{eq:M^k S_k eq d}, we have $\| M_\alpha^{(k)}( A)\|_\gamma= \| \widehat{S}_{k,\alpha} ( A) \|_\gamma$.  Viewing $\left(\widehat{S}_{k,\alpha} ( A)\right)_{A \in \mathcal{E}^{(k)}}$ as a set-indexed process, we note that different choices of the probability measure $m$ and the associated $\psi$  result in different modifications of the process. Nevertheless, all modifications have identical finite-dimensional distributions due to \eqref{eq:M^k S_k eq d}. 
% Furthermore, the finite dimensional distributions of $\left(M_\alpha^{(k)}(A), A \in \mathcal{E}^{(k)}\right)$ do not depend on $\psi$ because the finite-dimensional distributions of $\left(M_\alpha^{(k)}(A), A \in \mathcal{F}_k\right)$ do not depend on $\psi$, and because of the results established in \textbf{Step 1} and  \textbf{Step 2}. 
% Viewing \suggestion{$\left(M_\alpha^{(k)}(A),A \in \mathcal{F}_k\right)$}{} as a stochastic process, we note that different choices of $\psi$ result in different versions of the process. Nevertheless, all versions have identical finite-dimensional distributions, as the distribution of $M_\alpha(A)$, for $A \in \mathcal{E}$, does not depend on $\psi$. Furthermore, the finite dimensional distributions of $\left(M_\alpha^{(k)}(A), A \in \mathcal{E}^{(k)}\right)$ do not depend on $\psi$ because the finite-dimensional distributions of $\left(M_\alpha^{(k)}(A), A \in \mathcal{F}_k\right)$ do not depend on $\psi$, and because of the results established in \textbf{Step 1} and  \textbf{Step 2}. 
Hence, we can without loss of generality set $\psi(x)=\mu(E)$, $x\in E$, i.e., $m=\mu / \mu(E)$, which gives
\begin{equation}\label{term}
 \| M_\alpha^{(k)}( A)\|_\gamma  = \| \widehat{S}_{k,\alpha} ( A) \|_\gamma\le  \mu(E)^{k/\alpha} \left\|\sum_{\boldsymbol{j} \in \mathcal{D}_k}\left[\Gamma_{\boldsymbol{j}}\right]^{- 1 / \alpha}  \mathbf{1}_{\left\{T_{\boldsymbol{j}} \in A\right\}}\right\|_{\gamma}.
\end{equation}
At last, applying Lemma \ref{lem 2.3} by letting $g(x_1,\ldots,x_k) = \mathbf{1}_{\{(x_1, \cdots, x_k) \in A\}}$, the bound in  (\ref{term}) is further bounded above by $\left(\mu^k\left(A \right)\right)^{1 / r}$ with $r > \alpha$, up to a constant that does not depend on $A$. 

\medskip
\noindent\textbf{Step 3:}   We shall show the definition of $M_\alpha^{(k)}(A)$, $A \in \mathcal{E}^{(k)}$, is invariant to the choice of the approximating sequence of sets $A_n\in \cl{F}_k$ such that $\mu^k(A_n\Delta A)\rightarrow 0$ as $n\rightarrow\infty$.

Apart from $\left(A_n\right)_{n \in \bb{Z}_+}$, consider another sequence $\left(B_n\right)_{n \in \bb{Z}_+}$, where $B_n \in \mathcal{F}_k$ for each $n \in \bb{Z}_+$, satisfying $\mu^k\left(B_n \Delta A\right) \rightarrow 0$ as $n \rightarrow \infty$. 

Suppose $r > \alpha$ and $\gamma \in (0,\alpha)$. Using the first relation in \eqref{sbound} and the conclusions in \eqref{imp}, we derive 
$$\|\widehat{S}_{k, \alpha}\left(A_n\right)-\widehat{S}_{k, \alpha}\left( B_n\right)\|_\gamma \leq \|\widehat{S}_{k, \alpha}\left(A_n \Delta B_n\right)\|_\gamma \leq c \left(\mu^k(A_n \Delta B_n)\right)^{1 / r},   \ n\in\bb{Z}_+,
$$ 
where $c$ is a positive constant. As $n \rightarrow \infty$, we have $\mu^k\left(A_n \Delta B_n\right)\le  \mu^k\left(A_n \Delta A\right)+\mu^k\left(A \Delta B_n\right) \rightarrow 0$, and therefore,
$$\left\| M_\alpha^{(k)}\left(A_n\right)-M_\alpha^{(k)}\left(B_n\right) \right\|_\gamma \rightarrow 0.$$
\end{proof}

\subsection{Auxiliary results for Proposition \ref{prop 4.6}}
Recall the $\sigma$-finite measure $\mu$, the probability measure $m$ equivalent to $\mu$, the derivative $\psi=d\mu/dm$,  $\left(T_j\right)_{j \in \mathbb{Z}_{+}}$ and $\left(\Gamma_j\right)_{j \in \mathbb{Z}_{+}}$ in Definition \ref{defn l}. Recall also that $\left(T_j^{(\ell)}\right)_{j \in \mathbb{Z}_{+}}, \ell=1, \ldots, k$, are i.i.d. copies of $\left(T_j\right)_{j \in \mathbb{Z}_{+}}$, and $\left(\Gamma_j^{(\ell)}\right)_{j \in \mathbb{Z}_{+}}, \ell=1, \ldots, k$, are i.i.d. copies of $\left(\Gamma_j\right)_{j \in \mathbb{Z}_{+}}$, and the two collections are independent. 

For a measurable $f: E^k \mapsto[0, \infty]$ that vanishes on the diagonals, $k \geq 2$, we define 
$$
\widehat{S}_\alpha^{[1: k]}(f) = k^{-k / \alpha} \bigvee_{(\boldsymbol{i}, \boldsymbol{\ell}) \in 
 \widehat{\mathcal{D}}_k} f\left(T_{i_1}^{\left(\ell_1\right)}, \ldots, T_{i_k}^{\left(\ell_k\right)}\right)\left(\prod_{d=1}^k \psi\left(T_{i_d}^{\left(\ell_d\right)}\right)\right)^{1 / \alpha}\left(\prod_{d=1}^k \Gamma_{i_d}^{\left(\ell_d\right)}\right)^{-1 / \alpha},
$$
where
$$
\widehat{\mathcal{D}}_k = \left\{ (\boldsymbol{i}, \boldsymbol{\ell})\mid \boldsymbol{i}\in \bb{Z}_+^k,
 \boldsymbol{\ell}\in \{1,\ldots,k\}^k, (i_d,\ell_d) \neq (i_{d^\prime}, \ell_{d^\prime})  \text{ for any } d \neq d^\prime,\, d,d^\prime \in \{1,\ldots,k\} \right\}.
$$
\begin{prop} \label{prop B.5}
Set $$
M_{\alpha, \ell}(\cdot)=\bigvee_{i \geq 1} \mathbf{1}_{\left\{T_i^{(\ell)} \in \cdot\right\}} \psi\left(T_i^{(\ell)}\right)^{1 / \alpha}\left(\Gamma_i^{(\ell)}\right)^{-1 / \alpha}, \ell=1 \ldots k,$$
and $\widehat{M}_\alpha=k^{-1 / \alpha} \bigvee_{1 \leq \ell \leq k} M_{\alpha, \ell}$. Let $\widehat{M}_\alpha^{(k)}$ be constructed by $\widehat{M}_\alpha$ as in Section \ref{S 2.2}. Then we have
\begin{equation}\label{r26}
    \left(\widehat{S}_\alpha^{[1: k]}\left(\mathbf{1}_A\right)\right)_{A \in \mathcal{E}^{(k)}} \stackrel{d}{=}\left(\widehat{M}_\alpha^{(k)}(A)\right)_{A \in \mathcal{E}^{(k)}}.
\end{equation}
\end{prop}
\begin{proof}
Consider first the case $\mu(E)<\infty$. Suppose $A \in \mathcal{E}^{(k)}$, $k\in \bb{Z}_+$,  and a sequence  $A_n \in \mathcal{F}_k$, $n\in \bb{Z}_+$,  satisfy $\mu^k(A_n \Delta A) \rightarrow 0$ as $n\rightarrow\infty$ (see Lemma \ref{lem2.1}). Then both $\widehat{S}_\alpha^{[1: k]}\left(\mathbf{1}_A\right)<\infty$ and $\widehat{S}_\alpha^{[1: k]}\left(\mathbf{1}_{A_n}\right)<\infty$ a.s., $n \in \mathbb{Z}_+$. Indeed, taking the first relation as an example, we have for $r>\alpha$ that
\begin{equation}\label{eq 27}
\begin{aligned}
&\left(\widehat{S}_\alpha^{[1: k]}\left(\mathbf{1}_A\right)\right)^r \\
& \leq  k^{-k /  (\alpha/r)}\sum_{(\boldsymbol{i}, \boldsymbol{\ell}) \in \widehat{\mathcal{D}}_k}\mathbf{1}_{\{(T_{i_1}^{\left(\ell_1\right)}, \ldots, T_{i_k}^{\left(\ell_k\right)}) \in A\}}\left(\prod_{d=1}^k \psi\left(T_{i_d}^{\left(\ell_d\right)}\right)\right)^{1 /  (\alpha/r)}\left(\prod_{d=1}^k \Gamma_{i_d}^{\left(\ell_d\right)}\right)^{-1 / (\alpha/r)} \notag \\
& =: \widehat{S}_{k, \alpha / r}^*(A) 
 \leq\left( k^{-r / \alpha} \sum_{m=1}^k\sum_{j \in \mathbb{Z}_{+}} (\Gamma_j^{(m)})^{-1 /(\alpha / r)} \psi\left(T_j^{(m)}\right)^{1/(\alpha / r)}\right)^k.
\end{aligned}
\end{equation}
The last expression is a positive $(\alpha / r)$-stable random variable raised to power $k$, and hence finite a.s. (see the argument \eqref{leq sum}). 

Next, for  $\gamma \in (0,\alpha/k)$, note that
\begin{align*}
 \mathbb{E}\left|   \widehat{S}_\alpha^{[1: k]}\left(\mathbf{1}_{A}\right) - \widehat{S}_\alpha^{[1: k]}\left(\mathbf{1}_{A_n}\right)\right|^\gamma & \leq  \mathbb{E}\left|   \widehat{S}_\alpha^{[1: k]}\left(\mathbf{1}_{A \Delta A_n}\right) \right|^\gamma\leq \mathbb{E}\left|\widehat{S}_{k, \alpha / r}^*(A \Delta A_n) \right|^{\gamma/r}\rightarrow 0,
\end{align*}
as $n\rightarrow\infty$, where the last relation  follows similarly as the proof of Lemma \ref{lem 2.4}  
 using arguments as in \cite[part B of the proof of Theorem 3.1]{samorodnitsky1991construction}. Therefore, the relation \eqref{r26} follows.

For the case where $\mu$ is $\sigma$-finite, let $E_n$ and $\widehat{M}^{(k)}_{\alpha, E_n}$, $n \in \mathbb{Z}_+$, be as described prior to Theorem \ref{thm sf case}, and recall by definition $\widehat{M}_\alpha^{(k)}=\bigvee_{n=1}^{\infty} \widehat{M}^{(k)}_{\alpha, E_n}$.  Further, in view of the finite-measure case that has been proved, we have
\begin{equation*}
    \left(\widehat{S}_\alpha^{[1: k]}\left(\mathbf{1}_{A \cap E_n^k}\right)\right)_{A \in \mathcal{E}^{(k)},n\in \bb{Z}_+} \stackrel{d}{=}\left(\widehat{M}_{\alpha,E_n}^{(k)}(A )\right)_{A \in \mathcal{E}^{(k)},n\in \bb{Z}_+}.
\end{equation*}
The relation (\ref{r26}) follows by letting $n\rightarrow\infty$ above, and noting  $\widehat{S}_\alpha^{[1: k]}\left(\mathbf{1}_{A \cap E_n^k}\right)\rightarrow \widehat{S}_\alpha^{[1: k]}\left(\mathbf{1}_{A }\right)$ a.s..\ by monotonicity.
\end{proof}

\section{Definition consistency for multiple extremal integrals}

\begin{lem}[Consistency]\label{General consistency}
For $k \geq 1$, suppose $f_n \in \mathcal{S}_k$ for each $n \in \bb{Z}_+$ and  $f_n \nearrow f$ as $n\rightarrow\infty$. Assume another $g\in \mathcal{S}_k$ satisfies $0 \leq g \leq f$, then $\lim _n  I_k^e(f_n) \geq  I_k^e(g)$.
Hence, given $f_n \nearrow f$, $ g_n \nearrow f$  as $n\rightarrow\infty$
%$(f_n)_{n \in \mathbb{Z}_+}, (g_n)_{n \in \mathbb{Z}_+}$, 
where $f_n, g_n \in \mathcal{S}_k$ for each $n \in \mathbb{Z}_+$, we have $\lim_{n} I_k^e(f_n) = \lim_{n} I_k^e(g_n)$ a.s..
\end{lem}
\begin{proof}
If the claim holds for $g=\mathbf{1}_{A}$, where $A \in \mathcal{E}^{(k)}$, then it extends to any general simple $g \in \mathcal{S}_k$. Indeed, assume that the claim has been shown for these special indicator $g$'s. Now suppose $g = \sum_{i=1}^N a_i \mathbf{1}_{A_i}$, where $a_1, \ldots, a_N \in (0,\infty)$, and $A_1, \ldots, A_N \in \mathcal{E}^{(k)}$ are  disjoint. Since $f_n\mathbf{1}_{A_i} \nearrow f \mathbf{1}_{A_i}$ and $0 \leq a_i\mathbf{1}_{A_i} \leq f \mathbf{1}_{A_i}$, we have by assumption $\lim _n  I_k^e(f_n\mathbf{1}_{A_i}) \geq  I_k^e(a_i \mathbf{1}_{A_i})$ a.s.\ for $1 \leq i \leq N$. Then by Corollary \ref{cor:gen int}, 
the following relations hold a.s.:
$$
\lim _n  I_k^e(  f_n) \geq \lim _n  I_k^e\left( \bigvee_{i = 1}^N f_n\mathbf{1}_{A_i}\right) = \bigvee_{i = 1}^N\lim _n  I_k^e(f_n\mathbf{1}_{A_i}) \geq \bigvee_{i = 1}^N I_k^e(a_i \mathbf{1}_{A_i}) = I_k^e(g).
$$
Now, we prove the case where  $g=\boldsymbol{1}_{A}$ for $A \in \mathcal{E}^{(k)}$.  For  $\epsilon\in (0,1)$, define $$B_n=\left\{\boldsymbol{u}\in A \mid f_n(\boldsymbol{u}) \geq 1-\epsilon\right\}.$$
Observe that  $B_n \in \mathcal{E}^{(k)}$, $B_n\subset B_{n+1}$, $n\in \bb{Z}_+$,    $ \bigcup_{n=1}^{\infty} B_n  = A$, and $\boldsymbol{1}_{B_n} \nearrow \boldsymbol{1}_{A} $.
Using the monotonicity property in Corollary \ref{cor:gen int}, we can derive the following inequalities that hold a.s.:
\begin{equation}\label{eq 25}
I_k^e(f_n) \geq I_k^e(f_n\boldsymbol{1}_{B_n}) \geq I_k^e((1-\epsilon) \boldsymbol{1}_{B_n}) = (1-\epsilon)  I_k^e(\boldsymbol{1}_{B_n}).
\end{equation}
 In view of the $\sigma$-maxitive property of $M_\alpha^{(k)}$, we have $M_\alpha^{(k)}(B_n) \xrightarrow{\text{a.s.}}  M_\alpha^{(k)}(A)$ as $n \rightarrow \infty$.
 The conclusion follows from first letting $n\rightarrow\infty$ in \eqref{eq 25}, and then letting $\epsilon\rightarrow 0$.
\end{proof}

\section{Auxiliary results for LePage representation}

Suppose $\alpha\in (0,\infty)$, and 
$\pp{X_{\boldsymbol{j}}}_{\boldsymbol{j}\in \mathcal{D}_{k,<}}$ is an array of marginally identically distributed (possibly dependent) nonnegative random variables that are independent of the  unit-rate Poisson arrival times $\pp{ \Gamma_j}_{j \in \bb{Z}_+}$. For $m\in \bb{Z}_+$,   
introduce 
\begin{equation}\label{T}
\begin{aligned}
\mathcal{T}_{k, m}^{\prime}  & =  \bigvee_{\boldsymbol{j} \in \mathcal{D}_{k,<}, j_1 \geq m}\left[\Gamma_{\boldsymbol{j}}\right]^{-1 / \alpha} X_{\boldsymbol{j}}\,\,\boldsymbol{1}_{\{X_{\boldsymbol{j}}^\alpha \leq[\boldsymbol{j}]\}},   \\
\mathcal{T}_{k, m}^{\prime\prime}&=  \bigvee_{\boldsymbol{j} \in \mathcal{D}_{k,<}, j_1 \geq m}\left[\Gamma_{\boldsymbol{j}}\right]^{-1 / \alpha} X_{\boldsymbol{j}}\,\, \boldsymbol{1}_{\{X_{\boldsymbol{j}}^\alpha>[\boldsymbol{j}]\}}.
\end{aligned}
\end{equation}

The following proposition is an adaptation of \cite[Proposition 5.1]{samorodnitsky1989asymptotic}. 
%applied to the upper bounds  \eqref{eq:T_k,m bound max by sum}.   
\begin{prop}\label{prop6.7}~\\
(a) Let $r>\alpha, m \geq m_0>k r / \alpha$, and $k \geq 1$. Then, there exists a finite constant $C^{\prime}>0$, depending only on $\alpha, r, m_0$, and  {$k$}, but independent of $m$ and the distribution of $\pp{X_{\boldsymbol{j}}}_{\boldsymbol{j}\in \mathcal{D}_{k,<}}$, such that
$$
\mathbb{E}\left|\mathcal{T}_{k, m}^{\prime}\right|^\alpha \leq C^{\prime}\left\{ \mathbb{E}\left[ X_{\boldsymbol{j}} ^\alpha\left(1+\left(\ln _{+} X_{\boldsymbol{j}}\right)^{k-1}\right)\right]\right\}^{\alpha / r} .
$$
(b) Consider $\Phi(x)=x /(\ln (a+x))^{k-1}$ with $a$ chosen large enough to have $\Phi$ belong to the class $\mathscr{K}_\alpha$ as defined in \cite[Section 1]{samorodnitsky1989asymptotic}, and $\Phi_\alpha(x)=\Phi\left(x^\alpha\right)$, $x \geq 0$. Let $k\ge 2$ and $m \geq m_0>k$. Then, there is a finite constant $C^{\prime \prime}>0$ depending only on $\alpha, m_0$ and $k$, but independent of $m$ and the law of $\pp{X_{\boldsymbol{j}}}_{\boldsymbol{j}\in \mathcal{D}_{k,<}}$, such that
$$
\mathbb{E} \Phi_\alpha\left( \mathcal{T}_{k, m}^{\prime \prime} \right) \leq \begin{cases}C^{\prime \prime} \mathbb{E}\left[X_{\boldsymbol{j}}^\alpha\left(1+\left(\ln _{+}X_{\boldsymbol{j}}\right)^{k-1}\right)\right], & \text { if } k>2, \\ C^{\prime \prime} \mathbb{E}\left[X_{\boldsymbol{j}}^\alpha\left(1+\ln _{+}X_{\boldsymbol{j}} \ln _{+}|\ln X_{\boldsymbol{j}}|\right)\right], & \text { if } k=2.\end{cases}
$$
(c) Let $k \geq 1$ and $m \geq m_0>k$. Then, there exists a finite constant $C^{\prime \prime \prime}>0$, depending only on $\alpha$, $m_0$, and $k$, but independent of $m$ and the distribution of $\pp{X_{\boldsymbol{j}}}_{\boldsymbol{j} \in \mathcal{D}_{k,<}}$, such that
$$
\mathbb{E}\left|\mathcal{T}_{k, m}^{\prime \prime}\right|^\alpha \leq C^{\prime \prime \prime} \mathbb{E}\left[X_{\boldsymbol{j}}^\alpha\left(1+\left(\ln _{+}X_{\boldsymbol{j}}\right)^k\right)\right].
$$
\end{prop}
\begin{proof}

We only highlight the difference compared to the proof of \cite[Proposition 5.1]{samorodnitsky1989asymptotic}. 

For (a), by Hölder inequality, we have $ \mathbb{E}\left|\mathcal{T}_{k, m}^{\prime}\right|^\alpha \leq\left(\mathbb{E}\left|\mathcal{T}_{k, m}^{\prime}\right|^r\right)^{\alpha / r}$. Bounding supremum by sum, the right hand side of this inequality is further bounded by 
\begin{align*}
      \left[\sum_{\boldsymbol{j} \in \mathcal{D}_{k,<}, j_1 \geq m} \mathbb{E}\left[\Gamma_{\boldsymbol{j}}\right]^{-r / \alpha} \mathbb{E}\left[X_{\boldsymbol{j}}^r \boldsymbol{1}_{\left\{X_{\boldsymbol{j}}^\alpha \leq[\boldsymbol{j}]\right\}}\right]\right]^{\alpha / r}\text{.}
\end{align*}

For (b), since $\Phi_\alpha$ is an increasing function,   we can thus place the supremum  outside the function, and then bound the superemum by sum as: 
\begin{align*}
\mathbb{E} \Phi_\alpha\left(\left|\mathcal{T}_{k, m}^{\prime \prime}\right|\right) & \leq \sum_{\boldsymbol{j} \in \mathcal{D}_{k,<}, j_1 \geq m} \mathbb{E} \left[\Phi_\alpha\left(X_{\boldsymbol{j}}\left[\Gamma_{\boldsymbol{j}}\right]^{-1 / \alpha}\right) \boldsymbol{1}_{\left\{X_{\boldsymbol{j}}^\alpha>[\boldsymbol{j}]\right\}}\right] \text{.}
\end{align*}
The treatment for (c), and the rest of the proof, all  follow exactly similar arguments as in the proof of \cite[Proposition 5.1]{samorodnitsky1989asymptotic}. In the context of \cite[]{samorodnitsky1989asymptotic}, the range of $\alpha$ was restricted to $\alpha\in (0,2)$, while an inspection shows that the argument works for all $\alpha\in (0,\infty)$.
\end{proof}
We introduce the following corollary, a decoupled variant of \cite[Proposition 5.1]{samorodnitsky1989asymptotic},  which will be useful in proving Lemma \ref{thm: tail of tensor}.
\begin{cor}\label{at6.3}
Suppose $\alpha \in (0,1)$ and $p,q \in \mathbb{Z}_+$. Let
$\pp{X_{\boldsymbol{i}, \boldsymbol{j}}}_{\boldsymbol{i}\in \mathcal{D}_{p,<},\boldsymbol{j}\in \mathcal{D}_{q,<}}$ be an array of marginally identically distributed (possibly dependent) nonnegative random variables, independent of $\pp{\Gamma_j}_{j\in \bb{Z}_+}$. Further, set  
\begin{equation*} 
\begin{aligned}
\mathcal{T}_{p,q, m}^{\prime}  & =  \sum_{\substack{\boldsymbol{i} \in \mathcal{D}_{p,<}, \boldsymbol{j} \in \mathcal{D}_{q,<},\\ i_1 \geq m, j_1 \geq m, i_w \neq j_v,\\ 1 \leq w \leq p, 1\leq  v \leq q}}\left[\Gamma_{\boldsymbol{i}}\right]^{-1 / \alpha} \left[\Gamma_{\boldsymbol{j}}\right]^{-1 / \alpha} X_{\boldsymbol{i},\boldsymbol{j}}\,\,\boldsymbol{1}_{\{X_{\boldsymbol{i},\boldsymbol{j}}^\alpha \leq [\boldsymbol{i}][\boldsymbol{j}]\}},   \\
\mathcal{T}_{p,q, m}^{\prime\prime} &=  \sum_{\substack{\boldsymbol{i} \in \mathcal{D}_{p,<}, \boldsymbol{j} \in \mathcal{D}_{q,<},\\ i_1 \geq m, j_1 \geq m, i_w \neq j_v,\\1 \leq w \leq p, 1\leq  v \leq q}}\left[\Gamma_{\boldsymbol{i}}\right]^{-1 / \alpha} \left[\Gamma_{\boldsymbol{j}}\right]^{-1 / \alpha} X_{\boldsymbol{i},\boldsymbol{j}}\,\, \boldsymbol{1}_{\{X_{\boldsymbol{i},\boldsymbol{j}}^\alpha>[\boldsymbol{i}][\boldsymbol{j}]\}}.
\end{aligned}
\end{equation*}
Then, the conclusions in (a) and (c) of Proposition \ref{prop6.7} hold when $\mathcal{T}^{\prime}_{k,m}$ is replaced by    $\mathcal{T}_{p,q, m}^{\prime}$, $\mathcal{T}^{\prime\prime}_{k,m}$ is replaced by    $\mathcal{T}_{p,q, m}^{\prime\prime}$,  $X_{\boldsymbol{j}}$  replaced by  $X_{\boldsymbol{i},\boldsymbol{j}}$, $k$  replaced by $p+q$,  with $r$ in (a) further restricted to $r\in (\alpha,1)$,  and the constants $C'$ and $C'''$ depending only on $p,q,m_0$, but independent of $m$ and the marginal distribution of  $X_{\boldsymbol{i},\boldsymbol{j}}$.
\end{cor}
\begin{proof}
Below, let $C$ be a positive constant whose value may vary from one expression to another, depending only on $\alpha$, $m_0$, $p$ and $q$, but independent of $m$ and the marginal distribution of  $X_{\boldsymbol{i},\boldsymbol{j}}$. Recall that the bracket notation $[\ \cdot \ ]$ with a vector index inside stands for a product with respect to the indices; for example,
$[\boldsymbol{i}]=i_1\times \ldots \times i_p$ for $\boldsymbol{i}=(i_1,\ldots,i_p)\in \mathcal{D}_{p,<}$.\\
For (a),  by  Hölder inequality,  the inequality $\pp{\sum_\ell a_\ell}^r\le \sum_{\ell} a_\ell^r $ for positive sequence $(a_\ell)$ and $r\in (0,1)$, and   Fubini's theorem, we conclude that 
\begin{align}\label{eq 30}
\mathbb{E}\left|\mathcal{T}_{p,q, m}^{\prime}\right|^\alpha \le & \pp{\mathbb{E}\left|\mathcal{T}_{p,q, m}^{\prime}\right|^r}^{\alpha/r} \notag \\\leq & C\left[\sum_{\substack{\boldsymbol{i} \in \mathcal{D}_{p,<}, \boldsymbol{j} \in \mathcal{D}_{q,<},\\ i_1 \geq m, j_1 \geq m, i_w \neq j_v,\\ 1 \leq w \leq p, 1\leq  v \leq q}} \mathbb{E}\pp{\left[\Gamma_{\boldsymbol{i}}\right]^{-r / \alpha} \left[\Gamma_{\boldsymbol{j}}\right]^{-r / \alpha} }\mathbb{E}\pp{X_{\boldsymbol{i},\boldsymbol{j}}^r\mathbf{1}_{\left\{X_{\boldsymbol{i},\boldsymbol{j}}^\alpha \leq [\boldsymbol{i}][\boldsymbol{j}]\right\}}}\right]^{\alpha / r}.
\end{align}
  Further, first note  that $m >m_0 \geq (p+q)r/\alpha\ge 2$.  
 By \cite[Eq.\ (3.2)]{samorodnitsky1989asymptotic}, and the fact that $X_{\boldsymbol{i},\boldsymbol{j}}$'s are marginally identically distributed, the right hand side of (\ref{eq 30}) is bounded above by 
\begin{align}\label{eq 31}
& C\left[\sum_{\substack{\boldsymbol{i} \in \mathcal{D}_{p,<}, \boldsymbol{j} \in \mathcal{D}_{q,<},\\ i_1 \geq m_0, j_1 \geq m_0, i_w \neq j_v,\\ 1 \leq w \leq p, 1\leq  v \leq q}} [\boldsymbol{i}]^{-r / \alpha} [\boldsymbol{j}]^{-r / \alpha} \sum_{k=1}^{[\boldsymbol{i}][\boldsymbol{j}]} \mathbb{E}\left(X_{\boldsymbol{i},\boldsymbol{j}}^r \mathbf{1}_{\left\{k-1<X_{\boldsymbol{i},\boldsymbol{j}}^\alpha \leq k\right\}}\right)\right]^{\alpha / r} \notag\\
 \leq & C\left[\sum_{k=1}^{\infty} \mathbb{E}\left(X_{\boldsymbol{i}_0,\boldsymbol{j}_0}^r \mathbf{1}_{\left\{k-1<X_{\boldsymbol{i}_0,\boldsymbol{j}_0}^\alpha \leq k\right\}}\right) \sum_{\substack{[\boldsymbol{i}][\boldsymbol{j}] \geq k\\ [\boldsymbol{i}]>2,[\boldsymbol{j}]>2}}[\boldsymbol{i}]^{-r / \alpha} [\boldsymbol{j}]^{-r / \alpha}\right]^{\alpha / r},
\end{align}
where $\boldsymbol{i}_0$ and $\boldsymbol{j}_0$ are two fixed elements in $\cl{D}_{p,<}$ and $\cl{D}_{q ,<}$ respectively, and the second summation in \eqref{eq 31} is over all $\boldsymbol{i}\in \cl{D}_{p,<}$ and $\boldsymbol{j}\in \cl{D}_{q,<}$ satisfying the  constraint  indicated below the summation sign (similar notation will be used below).
Note that by \cite[Lemma 4.1 (ii), (iv)]{samorodnitsky1989asymptotic}, we have
\begin{align*}
   \sum_{\substack{[\boldsymbol{i}][\boldsymbol{j}] \geq k\\ [\boldsymbol{i}]>2,[\boldsymbol{j}]>2}}[\boldsymbol{i}]^{-r / \alpha} [\boldsymbol{j}]^{-r / \alpha} &\leq \sum_{2<[\boldsymbol{i}]\leq k} [\boldsymbol{i}]^{-r / \alpha}\sum_{[\boldsymbol{j}] \geq k/[\boldsymbol{i}]}  [\boldsymbol{j}]^{-r / \alpha} + \sum_{[\boldsymbol{i}]> k} [\boldsymbol{i}]^{-r / \alpha}  \sum_{[\boldsymbol{j}] >2}[\boldsymbol{j}]^{-r / \alpha}\\
    & \leq C \sum_{2<[\boldsymbol{i}]\leq k} [\boldsymbol{i}]^{-r / \alpha}\pp{k/[\boldsymbol{i}]}^{1-r/\alpha}(\ln \pp{k/[\boldsymbol{i}]})^{q-1} + C k^{1-r / \alpha}(\ln k)^{p-1}\\
    & \leq C\sum_{2<[\boldsymbol{i}]\leq k}[\boldsymbol{i}]^{-1}  k^{1-r / \alpha}(\ln k)^{q-1} + C k^{1-r / \alpha}(\ln k)^{p-1} \\
    &\leq C k^{1-r / \alpha}(\ln k)^{p+q-1}.
\end{align*}
Then, the  expression in (\ref{eq 31})   is bounded above by 
\begin{equation}
\begin{aligned}
&  C\left[\mathbb{E}\left(X_{\boldsymbol{i},\boldsymbol{j}}^r \mathbf{1}_{\left\{0<X_{\boldsymbol{i},\boldsymbol{j}}^\alpha \leq 2\right\}}\right) +\sum_{k=3}^{\infty} k(\ln k)^{p+q-1} \mathbb{P}\left(k-1<X_{\boldsymbol{i},\boldsymbol{j}}^\alpha \leq k\right)\right]^{\alpha / r} \\
& \leq C\left\{\mathbb{E}\left[X_{\boldsymbol{i},\boldsymbol{j}}^\alpha\left(1+\left(\ln _{+}X_{\boldsymbol{i},\boldsymbol{j}}\right)^{p+q-1}\right)\right]\right\}^{\alpha / r} .
\end{aligned}
\end{equation}
For (c), again by the inequality $\pp{\sum_\ell a_\ell}^\alpha\le \sum_{\ell} a_\ell^\alpha $ and \cite[Eq.\ (3.2)]{samorodnitsky1989asymptotic}, we have 
\begin{align*}
    \mathbb{E}\left|\mathcal{T}_{p,q, m}^{\prime\prime}\right|^\alpha  &\le \sum_{\substack{\boldsymbol{i} \in \mathcal{D}_{p,<}, \boldsymbol{j} \in \mathcal{D}_{q,<},\\ i_1 \geq m, j_1 \geq m, i_w \neq j_v,\\ 1 \leq w \leq p, 1\leq  v \leq q}} \mathbb{E} \pp{\left[\Gamma_{\boldsymbol{i}}\right]^{-1 }  \left[\Gamma_{\boldsymbol{j}}\right]^{-1 }  }\mathbb{E}\pp{X_{\boldsymbol{i},\boldsymbol{j}}^\alpha\,\, \boldsymbol{1}_{\{X_{\boldsymbol{i},\boldsymbol{j}}^\alpha>[\boldsymbol{i}][\boldsymbol{j}]\}}}\\
    &\leq C\sum_{\substack{\boldsymbol{i} \in \mathcal{D}_{p,<}, j_1 \geq m\\ \boldsymbol{j} \in \mathcal{D}_{q,<}, i_1 \geq m}} \left[\boldsymbol{i}\right]^{-1 }  \left[\boldsymbol{j}\right]^{-1 }  \mathbb{E}\pp{X_{\boldsymbol{i},\boldsymbol{j}}^\alpha\,\, \boldsymbol{1}_{\{X_{\boldsymbol{i},\boldsymbol{j}}^\alpha>[\boldsymbol{i}][\boldsymbol{j}]\}}}\\
    & \le C \sum_{\substack{\boldsymbol{i} \in \mathcal{D}_{p,<}, j_1 \geq m_0\\ \boldsymbol{j} \in \mathcal{D}_{q,<}, i_1 \geq m_0}} \left[\boldsymbol{i}\right]^{-1 }  \left[\boldsymbol{j}\right]^{-1 }  \sum_{k = [\boldsymbol{i}][\boldsymbol{j}]}^\infty\mathbb{E}\pp{X_{\boldsymbol{i},\boldsymbol{j}}^\alpha\,\, \boldsymbol{1}_{\{k<X_{\boldsymbol{i},\boldsymbol{j}}^\alpha \le k+1\}}}.
\end{align*}
 So changing the order of summation, and making use of \cite[Lemma 4.1(ii)]{samorodnitsky1989asymptotic}, recalling  $m_0>1$,  we obtain
\begin{align*}
    \mathbb{E}\left|\mathcal{T}_{p,q, m}^{\prime\prime}\right|^\alpha &\le C \sum_{k = 1}^\infty \sum_{[\boldsymbol{i}][\boldsymbol{j}] \le k} \left[\boldsymbol{i}\right]^{-1 }  \left[\boldsymbol{j}\right]^{-1 } (k+1)\mathbb{P}\pp{k<X_{\boldsymbol{i}_0,\boldsymbol{j}_0}^\alpha \le k+1}\\
    &\le C \sum_{k=1}^\infty (k+1)\mathbb{P}\pp{k<X_{\boldsymbol{i}_0,\boldsymbol{j}_0}^\alpha \le k+1} \sum_{[\boldsymbol{i}] \leq k} [\boldsymbol{i}]^{-1} \sum_{[\boldsymbol{j}] \leq k} [\boldsymbol{j}]^{-1}\\
    & \le C\sum_{k=1}^\infty (\ln k)^{p+q} (k+1)\mathbb{P}\pp{k<X_{\boldsymbol{i},\boldsymbol{j}}^\alpha \le k+1} \le C\mathbb{E}X_{\boldsymbol{i},\boldsymbol{j}}^\alpha\left(1+\left(\ln _{+}X_{\boldsymbol{i},\boldsymbol{j}}\right)^{p+q}\right).
\end{align*}
\end{proof}

\section{Additional discussions on sufficient conditions for integrability}

\begin{lem}\label{lem eqv}
    Suppose $\mu$ is a probability measure and $f: E^k \mapsto [0,\infty]$ is a measurable function. Let $\wt{f}$ be the max-symmetrization as in \eqref{sym def}. 
    \begin{enumerate}
        \item For $k = 2$, condition (\ref{I2}) holds if and only if it holds with  $f$  replaced by $\widetilde{f}$. 
        \item For $k >2$,
        $$
L^\alpha \ln ^{k-1} L(f, \mu) < \infty \Longleftrightarrow L^\alpha \ln ^{k-1} L(\widetilde{f}, \mu) < \infty.
$$
\item For $k = 2$, 
$$
L^\alpha \ln L \ln \ln L(f, \mu) < \infty \Longleftrightarrow L^\alpha \ln L \ln \ln L(\widetilde{f}, \mu) < \infty.
$$
\end{enumerate}
\end{lem}
\begin{proof}
For each permutation $\pi \in \Theta_k$, define $f_\pi: E^k \mapsto[0, \infty]$ by $f_\pi\left(u_1, \ldots, u_k\right)=f\left(u_{\pi(1)}, \ldots, u_{\pi(k)}\right)$, and set $D_\pi=\left\{\boldsymbol{u} \in E^k \mid \widetilde{f}(\boldsymbol{u})=f_\pi(\boldsymbol{u})\right\}$.
Consider the case $k = 2$ and suppose condition (\ref{I2}) holds.     Let $D_{(1,2)}=D_{\mathrm{id}}$ where $\mathrm{id}$ is identity and let $D_{(2,1)}=D_{\pi}$ where $\pi(1)=2$ and $\pi(2)=1$.   Note that $E^2 = D_{(1,2)} \cup D_{(2,1)}$, and for any $\left(u_1, u_2\right) \in E^2$,  
$$
\int_E f\left(u_1, v\right)^\alpha \mu\left(d v\right) \leq \int_E \widetilde{f}\left(u_1, v\right)^\alpha  \mu(dv), \quad \int_E f\left(w, u_2\right)^\alpha \mu\left(d w\right) \leq \int_E \widetilde{f}\left(w, u_2\right)^\alpha  \mu(dw).
$$
Hence, 
\begin{align*}
    &\int_{E^2} \widetilde{f}(u_1, u_2)^\alpha\left(1+\ln _{+}\left(\frac{\widetilde{f}(u_1, u_2)}{\left(\int_E \widetilde{f}(u_1, v)^\alpha \mu(dv)\right)^{1 / \alpha}\left(\int_E \widetilde{f}(w, u_2)^\alpha \mu(dw)\right)^{1 / \alpha}}\right)\right) \mu(d u_1) \mu(d u_2)\\
    \leq  & \int_{D_{(1,2)}} f(u_1, u_2)^\alpha\left(1+\ln _{+}\left(\frac{f(u_1, u_2)}{\left(\int_E f(u_1, v)^\alpha \mu(d v)\right)^{1 / \alpha}\left(\int_E f(w, u_2)^\alpha \mu(d w)\right)^{1 / \alpha}}\right)\right) \mu(d u_1) \mu(d u_2) \\
   &+  \int_{D_{(2,1)}} f(u_2, u_1)^\alpha\left(1+\ln _{+}\left(\frac{f(u_2, u_1)}{\left(\int_E f(v, u_1)^\alpha \mu(d v)\right)^{1 / \alpha}\left(\int_E f(u_2, w)^\alpha \mu(d w)\right)^{1 / \alpha}}\right)\right) \mu(d u_1) \mu(d u_2)\\
    \leq & 2  \int_{E^2} f(u_1, u_2)^\alpha\left(1+\ln _{+}\left(\frac{f(u_1, u_2)}{\left(\int_E f(u_1, v)^\alpha \mu(dv)\right)^{1 / \alpha}\left(\int_E f(w, u_2)^\alpha \mu(dw)\right)^{1 / \alpha}}\right)\right) \mu(d u_1) \mu(d u_2) < \infty.
\end{align*}
Now, we assume (\ref{I2}) holds with $f$ replaced by $\wt{f}$, and we want to establish \eqref{I2}. For any   measurable function $h: E^2 \mapsto [0,\infty]$, define
$$
A_h\left(u_1\right):=\int_E h\left(u_1, u_2\right)^\alpha \mu\left(d u_2\right), \quad B_h\left(u_2\right):=\int_E h\left(u_1, u_2\right)^\alpha \mu\left(d u_1\right).
$$
The assumption can be written as
$$
\int_{E^2} \widetilde{f}(u_1,u_2)^\alpha\left(1+\ln_+\frac{\widetilde{f}(u_1, u_2)}{A_{\widetilde{f}}\left(u_1\right)^{1 / \alpha} B_{\widetilde{f}}\left(u_2\right)^{1 / \alpha}}\right) \mu(d u_1)\mu(du_2)<\infty .
$$
Since $f \leq \widetilde{f}$ pointwise, we also have $A_f \leq A_{\widetilde{f}}$ and $B_f \leq B_{\widetilde{f}}$ pointwise. Moreover,
$$
\frac{f\left(u_1, u_2\right)}{A_f\left(u_1\right)^{1 / \alpha} B_f\left(u_2\right)^{1 / \alpha}}=\frac{\widetilde{f}\left(u_1, u_2\right)}{A_{\widetilde{f}}\left(u_1\right)^{1 / \alpha} B_{\widetilde{f}}\left(u_2\right)^{1 / \alpha}} \cdot \frac{f}{\widetilde{f}} \cdot\left(\frac{A_{\widetilde{f}}\left(u_1\right)}{A_f\left(u_1\right)}\right)^{1 / \alpha}\left(\frac{B_{\widetilde{f}}\left(u_2\right)}{B_f\left(u_2\right)}\right)^{1 / \alpha} .
$$
Using the inequality $\ln _{+}(x y) \leq \ln _{+} x+\ln _{+} y$, for all $x,y> 0$, it follows that
$$
\ln _{+} \frac{f\left(u_1, u_2\right)}{A_f\left(u_1\right)^{1 / \alpha} B_f\left(u_2\right)^{1 / \alpha}} \leq \ln _{+} \frac{\widetilde{f}\left(u_1, u_2\right)}{A_{\widetilde{f}}\left(u_1\right)^{1 / \alpha} B_{\widetilde{f}}\left(u_2\right)^{1 / \alpha}}+\frac{1}{\alpha} \ln \frac{A_{\widetilde{f}}\left(u_1\right)}{A_f\left(u_1\right)}+\frac{1}{\alpha} \ln \frac{B_{\widetilde{f}}\left(u_2\right)}{B_f\left(u_2\right)} .
$$
Note that $A_{\widetilde{f}}\left(u_1\right) \leq A_f\left(u_1\right)+B_f\left(u_1\right)$ pointwise. So by Fubini's theorem and the inequality $\ln (1+ x) \leq x$, $x\ge 0$, we obtain
$$
\begin{aligned}
\int_{E^2} f\left(u_1, u_2\right)^\alpha \ln \frac{A_{\wt{f}}\left(u_1\right)}{A_f\left(u_1\right)}  \mu\left(du_1\right)  \mu\left(d u_2\right) & =\int_E\left[\ln \frac{A_{\wt{f}}(u_1)}{A_{f}(u_1)}\right]\left(\int_E f(u_1, u_2)^\alpha  \mu(du_2)\right)  \mu(du_1) \\
& =\int_E A_f(u_1) \ln \frac{A_{\wt{f}}(u_1)}{A_f(u_1)} d \mu(u_1) \\
& \leq \int_E A_f(u_1) \ln \left(1+\frac{B_f(u_1)}{A_f(u_1)}\right) \mu(du_1)\\
&\leq \int_E B_f(u_1) \mu(du_1) = \int_{E^2} f(u_1, u_2)^\alpha\mu(d u_1)\mu(d u_2).
\end{aligned}
$$
A similar calculation yields
$$
\int_{E^2} f^\alpha(u_1, u_2) \ln \frac{B_{\wt{f}}(u_2)}{B_f(u_2)}  \mu(d u_1)  \mu(d u_2) \leq \int_{E^2} f^\alpha((u_1,u_2)) \mu(d u_1)\mu(d u_2).
$$
Combining the above inequalities yields
\begin{align*}
    &\int_{E^2} f^\alpha((u_1,u_2))\left(1+\frac{f\left((u_1,u_2)\right)}{A_f^{1 / \alpha}\left(u_1\right) B_f^{1 / \alpha}\left(u_2\right)}\right)\mu(d u_1)\mu(d u_2) \\
    & \leq \int_{E^2} \wt{f}^\alpha((u_1,u_2))\left(1+\ln _{+} \frac{\wt{f}(u_1,u_2)}{A_{\wt{f}}^{1 / \alpha}(u_1) B_{\wt{f}}^{1 / \alpha}(u_2)}\right) \mu(d u_1)\mu(d u_2)
     +\frac{2}{\alpha} \int_{E^2} f^\alpha((u_1,u_2)) \mu(d u_1)\mu(d u_2) < \infty.
\end{align*}
For part (ii), we have by monotonicity that
$$L^\alpha \ln ^{k-1} L(f , \mu) \leq L^\alpha \ln ^{k-1} L(\widetilde{f} , \mu) \le \sum_{\pi \in \Theta_k}L^\alpha \ln ^{k-1} L( f_\pi \mathbf{1}_{D_\pi}, \mu) \leq k! L^\alpha \ln ^{k-1} L(f , \mu).$$
Hence, the conclusion follows.

Also, part (iii) follows by an argument similar to the proof of part (ii).
\end{proof}
The following example shows that the integrability condition \eqref{suff 2} in Theorem \ref{suff thm} may hold for one probability measure equivalent to $\mu$, but fail for another such equivalent probability measure.
\begin{exam}\label{exam thm 3.2}
 Let $E=(0,1)$ and $\mu=\lambda$. Define two probability measures equivalent to $\lambda$.  The first is $m_1$, the uniform law, so $\psi_1(x)=(d \lambda/ d m_1)(x)=1$, whereas the second is $m_2$, with density
$\frac{d m_2}{d \lambda}(x)=C^{-1} e^{-(\ln (1 / x))^2}$, $C=\int_0^1 e^{-(\ln (1 / t))^2} d t$ so that $\psi_2(x)=(d \lambda / d m_2)(x)=Ce^{(\ln (1 / x))^2}$.
Fix $\alpha>0$,  define
$$
f(x, y)=\left(\frac{1}{x L(x)^3} \cdot \frac{1}{y L(y)^3}\right)\quad \text{on }\left(0, e^{-e}\right)^2,
$$
and set $f \equiv 1$ elsewhere on $(0,1)^2$, where $L(x)=\ln (1 / x)$. For simplicity, we take $\alpha$ in condition \eqref{suff 2} to be $1$ throughout.
Then, consider
$$
I_i=L^\alpha \ln L\left(f \cdot \psi_i^{\otimes 2}, m_i\right) = \iint f(x, y)\left(1+\ln _{+}\left(f(x, y) \psi_i(x) \psi_i(y)\right)\right)  \lambda( d x)  \lambda(dy), \quad i=1,2.
$$
%For $i=1, \psi_1 \equiv 1$.
On $\left(0, e^{-e}\right)^2$, we have
\begin{equation}\label{eq:f(x,y)<= 1+L+L}
\ln f(x, y)= L(x)+L(y)-3 \ln L(x)-3 \ln L(y) \leq  L(x)+L(y),
\end{equation}
where $L(x),L(y)>0$.
Thus, on $\left(0, e^{-e}\right)^2$, we have
\begin{equation*}
f(x,y)\left(1+\ln _{+}\left(f(x,y)   \psi_1(x) \psi_1(y)\right)\right) \leq \frac{1}{x L(x)^3} \frac{1}{y L(y)^3}\left(1+ L(x)+L(y)\right) .
\end{equation*}
To see that the double integral of the  right-hand side above over $\left(0, e^{-e}\right)^2$ is finite, note that   $\int_0^{e^{-e}} \frac{d x}{x L(x)^3}=\int_1^{\infty} \frac{d t}{t^3}<\infty $ and $\int_0^{e^{-e}} \frac{  d x}{x L(x)^2}=\int_1^{\infty} \frac{d t}{t^2}<\infty$. Hence, one can conclude $I_1<\infty$.

For $i=2$, we have $\ln \left(f(x, y) \psi_2(x) \psi_2(y)\right)=\ln f(x, y)+2\ln C+L(x)^2+L(y)^2$. In view of \eqref{eq:f(x,y)<= 1+L+L}, one can choose  $\delta>0$  such that when 
$(x,y)\in (0,\delta)^2$, we have 
$$
\ln _{+}\left(f(x, y) \psi_2(x) \psi_2(y)\right) \geq \frac{1}{2}\left(L(x)^2+L(y)^2\right),
$$
which leads to 
$$
f(x,y)\left(1+\ln _{+}\left(f(x,y) \psi_2(x) \psi_2(y)\right)\right) \geq \frac{1}{2}\left(\frac{1}{x L(x)} \cdot \frac{1}{y L(y)^3}+\frac{1}{x L(x)^3} \cdot \frac{1}{y L(y)}\right).
$$
Consequently, 
$$I_2 \geq \frac{1}{4}\left(\int_0^\delta \frac{d x}{x L(x)}\right)\left(\int_0^\delta \frac{d y}{y L(y)^3}\right)+\frac{1}{4}\left(\int_0^\delta \frac{d x}{x L(x)^3}\right)\left(\int_0^\delta \frac{d y}{y L(y)}\right),$$
where $\int_0^\delta \frac{d x}{x L(x)}=\int_{\ln (1 / \delta)}^{\infty} \frac{d t}{t}=\infty$. Hence, $I_2=\infty$.
\end{exam}

The following example shows that the condition $L^\alpha \ln ^{k-1} L(f, \mu)<\infty$ is not sufficient for integrability, that is, $I_k^e(f)<\infty$ a.s..
\begin{exam}\label{not suff exam}
Under the setup of Example \ref{S1}, let $f = \mathbf{1}_A$, where $A=\bigcup_{i \in \mathbb{Z}_{+}} A_i$ and sets $A_i$, $i\in \mathbb{Z}_+$ are as defined in \eqref{Ai}. We know from Theorem \ref{suff thm} that failure of integrability implies that $L^\alpha \ln ^{k-1} L\left(f \cdot\left(\psi^{\otimes k}\right)^{1 / \alpha}, m\right) = \infty$ for any probability measure $m$ equivalent to $\lambda$ and $\psi = d\mu / d m \in (0,\infty)$ $m$-a.e.. However, note that 
$$L^\alpha \ln ^{k-1} L\left(f, \mu\right)= \int_{\mathbb{R}^k} f d\lambda^k = \lambda^k(A) < \infty.$$
\end{exam}

\section{Auxiliary results for Proposition \ref{extremal diff order}}

 % We present a few lemmas used in the proof of Proposition \ref{extremal diff order}.

\begin{lem}\label{lem F.1}
Suppose $X_1, X_2, \ldots X_n$, $n \in \mathbb{Z}_+$, are i.i.d.\ nonnegative random variables and satisfy
$$
\mathbb{P}\left(X_1 \leq x\right) \sim c x^\alpha, \quad \text { as } x \downarrow 0,
$$
for some constants $c, \alpha>0$. Then
\begin{equation}\label{Fn}
    F_n(x): = \mathbb{P}\left(X_1 X_2 \cdots X_n\leq x\right) \sim \frac{\alpha^{n-1} c^n}{(n-1)!} x^{\alpha}(\ln(1/x))^{n-1} , \text{ as } x \downarrow 0.
\end{equation}
\end{lem}
\begin{proof}
 It follows     from \cite[(1.6)]{kasahara2018product}.
\end{proof}
\begin{lem}\label{lem F.2}
Let $X_i, i \in \mathbb{Z}_+$, be a sequence of i.i.d.\ nonnegative random variables and satisfy
$$
\mathbb{P}\left(X_1 \leq x\right) \sim c x^\alpha, \quad \text { as } x \downarrow 0,
$$
for some constants $c, \alpha>0$.
Let $1 \leq p<q$.  There exists a constant $C>0$, such that for all $0<2y<x<1/2$, we have 
    \begin{equation}
        \mathcal{P}_{p, q}(x, y): = \mathbb{P}\pp{X_1 \ldots X_p \leq x, X_1\ldots X_q\leq y} \le C y\pp{\ln \pp{1/y}}^{p-1}\pp{\ln \pp{ x/y}}^{q-p}.
    \end{equation}
    \end{lem}
\begin{proof}
Throughout, $C\in (0,\infty)$ denotes a generic constant whose value can change from expression to expression.  First note that 
\begin{equation}
    \mathcal{P}_{p, q}(x, y)  =\mathbb{E} \left[\mathbb{P}\left(X_1 \cdots X_p<x, \left.X_{p+1} \cdots X_q<\frac{y}{X_1 \cdots X_p} \right\rvert\, X_1 \cdots, X_p\right)\right].
\end{equation}
This can be rewritten as
$$\begin{aligned} \mathcal{P}_{p, q}(x, y) &  =\int_{0<s \leq x} F_{q-p}\left(\frac{y}{s}\right) F_p(d s) \\
    & =  \int_{2 y<s \leqslant x} \quad F_{q-p}\left(\frac{y}{s}\right) F_p(d s) +  \int_{\substack{0<s \leqslant x\\ y \geq \frac{s}{2}}} \quad F_{q-p}\left(\frac{y}{s}\right) F_p(d s) =: \mathrm{I}_{p,q}(x,y) +  \Pi_{p,q}(x,y),
    \end{aligned}$$
where $F_p, F_{q-p}$ are as defined in \eqref{Fn}. 

Lemma \ref{lem F.1} implies that when $0<2y<x < \frac{1}{2}$,
\begin{equation}\label{IIpq}
 \Pi_{p, q}(x, y) = \int_{s \leq 2 y} F_p(d s)\leq C \left(y^\alpha(\ln( 1/y))^{p-1}\right),
\end{equation}
and
\begin{equation}\label{Ipq}
\begin{aligned}
    \mathrm{I}_{p,q}(x, y) &\leq C y^\alpha \int_{2 y<s \leqslant x}\frac{1}{s^\alpha} \pp{\ln \left(s/y\right)}^{q-p-1} F_p(d s)\\
    & =: C y^\alpha \int_{2 y<s \leqslant x} g(s) F_p(d s)   = C y^\alpha \left[ \left[g(s) F_p(s)\right]\Big|_{2 y}^x -\int_{2 y}^x F_p(s-)  g^\prime(s) ds\right],
\end{aligned}
\end{equation}
where $F_p(s-)$ denotes the left limit of $F_p$ at $s$.
Since $$g^\prime(s) = s^{-\alpha-1}\left[(q-p-1)(\ln (s / y))^{q-p-2}-\alpha(\ln (s / y))^{q-p-1}\right],$$
we have $|g^\prime(s)|\leq C s^{-\alpha-1} \pp{\ln \left(s/y\right)}^{q-p-1}$ for $s\in (2y, x]$.
Applying Lemma \ref{lem F.1} again, we obtain
\begin{equation}\label{last}
    \begin{aligned}
   \left|\int_{2 y}^x F_p(s-) d g(s)\right|  &\leq C \int_{2 y}^x \frac{F_p (s-)}{s^{\alpha + 1}} \pp{\ln \left(s/y\right)}^{q-p-1} d s \\
    &\leq C \int_{2 y}^x s^\alpha(\ln (1 / s))^{p-1} \frac{1}{s^{\alpha + 1}} \pp{\ln \left(s/y\right)}^{q-p-1} d s.
\end{aligned}
\end{equation}
By a change of variable $u=\frac{\ln (s / y)}{\ln (1 / y)}$, the last integral is \eqref{last} is equal to 
\begin{align*}
        \pp{\ln \left(1/y\right)}^{q-1} \int_{\frac{\ln 2}{\ln (1 / y)}}^{\frac{\ln (x / y)}{\ln (1 / y)}} u^{q-p-1}(1-u)^{p-1} d u &\leq C\pp{\ln \left(1/y\right)}^{q-1} \int_0^{\frac{\ln (x /y)}{\ln (1/y)}} u^{q-p-1} d u  \\&= C\pp{\ln \left(1/y\right)}^{p-1} \ln \left(x/y\right)^{q-p}.
\end{align*}
Also, by Lemma \ref{lem F.1} and the restriction $0<2y<x<1/2$, 
\begin{align}\label{F.4}
\left|  \left[g(s) F_p(s)\right] \Big|_{2 y}^x \right|   & \leq C [\ln \left(x/y\right)^{q- p-1} \ln \left(1/x\right)^{p-1}+\ln \left(1/y\right)^{p-1}]  \leq C [\pp{\ln \left(1/y\right)}^{p-1} \ln \left(x/y\right)^{q-p}].
\end{align}
Hence, combining \eqref{Ipq}-\eqref{F.4}, we have $$\mathrm{I}_{p, q}(x, y) \leq C y^\alpha(\ln (1 / y))^{p-1} \ln (x / y)^{q-p}, \quad 0<2y<x<1/2.$$ Together with \eqref{IIpq}, this completes the proof.

\end{proof}

\begin{lem}\label{lem F.3}
 Suppose random variables $X, Y \geq 0$ a.s.\ with $X$ independent of $Y$ and $\mathbb{E}[X]<\infty$. Assume   
$$
\mathbb{P}(Y \leq t) \sim   t L(t),\quad \text{ as } t \downarrow 0,
$$
for some  positive slowly varying function $L$ non-increasing in a neighborhood of $0$.
Then  for any $x >0$, we have $\frac{\mathbb{P}( Y\leq tx)}{\mathbb{P}(Y \leq t)} \longrightarrow x$ as $t \downarrow 0$ and
$$
\frac{\mathbb{P}( Y\leq tX)}{\mathbb{P}(Y \leq t)} \longrightarrow \mathbb{E}[X], \quad \text{ as }t \downarrow 0.
$$
\end{lem}
\begin{proof}
% Throughout, $C\in (0,\infty)$ denotes a generic constant whose value may change from expression to expression. 
For $x\ge 0$ and $t>0$, set
$
G_t(x)=\frac{\mathbb{P}(Y \leq x t)}{\mathbb{P}(Y \leq t)}.
$
The first conclusion $$G_t(x) \rightarrow x \quad  \text{ as } t \downarrow 0$$
for any $x\ge 0$
follows from the slow variation of $L$.

% Define
% $$
% M(t):=\frac{\mathbb{P}(Y \leq t)}{t}, \quad t>0.
% $$
% By assumption, $M(t) \sim C(\log (1 / t))^{n-1}$ as $t \downarrow 0$, so $M$ is slowly varying at 0. 
% Slow variation of $M$ at 0 implies $$G_t(x) \rightarrow x$$ as $t \downarrow 0$,  and hence the first conclusion.

   We derive the second conclusion via the dominated convergence theorem. 
Fix $\delta>0$ sufficiently small so that, for some constant $C_1\in(0,\infty)$,
\[
G_t(x)\;\le\;
\begin{cases}
1, & 0\le x\le 1,\\[4pt]
C_1\,\dfrac{\mathbb{P}(Y\le tx)}{t\,L(t)}, & x>1,
\end{cases}
\qquad \text{for all } t\in(0,\delta).
\]
Moreover, for all $t\in(0,\delta)$ and $x>1$,  applying the non-increasing assumption on $L$ (adjusting $\delta>0$ smaller if needed), we have
\[
\frac{\mathbb{P}(Y\le tx)}{t\,L(t)}
\;\le\;
\begin{cases}
x\,\dfrac{L(tx)}{L(t)}\le x, & \text{ if } 0< tx\le\delta,\\ 
\dfrac{1}{t\,L(t)}\le  \dfrac{x}{\delta L(\delta)}, & \text{ if } tx\ge \delta.
\end{cases} ~
% \le ~ \begin{cases}
%  x, & 0< tx\le\delta,\\[6pt]
%  \frac{x}{\delta L(\delta)} , & tx\ge \delta.
% \end{cases}
\] 
Consequently, combining the above bounds yields
\[
G_t(x)\le C_2\,(x+1)
\]
for some constant $C_2\in(0,\infty)$. 
Since $\mathbb{E}(X+1)<\infty$, the dominated convergence theorem applies.

% Further, note that $M$ is bounded on $[ e^{-1}, \infty)$ and bounded above by $C(\log (1 / t))^{n-1}$ on $(0,e^{-1})$. Thus,  for $x>1$ and $t$ small enough,
% $$
%  \frac{M(x t)}{(\log (1 / t))^{n-1}}\leq \begin{cases} C\frac{(\log (1 /t) - \log x) ^{n-1}}{(\log (1 / t))^{n-1}}, & \text{ if } tx \in (0,e^{-1}), \\ \frac{C}{(\log (1 / t))^{n-1}}, & \text{ if } tx \in [e^{-1}, \infty), \end{cases}
% $$
% and thus $\frac{M(x t)}{(\log (1 / t))^{n-1}} \leq C$. Combining this with the case $0 \leq x \leq 1$ shows that for every $x>0$ and all $t$ small enough, $G_t(x) \leq C(x+1)$. Since $\mathbb{E}(X+1) < \infty$, the dominated convergence theorem yields $\mathbb{E}\left[G_t(X)\right] \rightarrow \mathbb{E}[X]$, as $t \downarrow 0$.

\end{proof}
\section{Proof of Lemma \ref{thm: tail of tensor} }\label{app lem 6.3}
The proof here overall follows arguments similar to those in the proof of \cite[Theorem 3.1]{rosinski1999product}. We include some details for reader's convenience.
Without loss of generality, suppose $\mu$ is a probability measure, and $\psi\equiv 1$. We introduce the following notation: for a sequence $(a_i)_{i \in \mathbb{Z}_+}$, write $a_{[1:m]} = (a_1, \ldots, a_m)$, and $a_{\boldsymbol{s}+c} = (a_{s_1+c},a_{s_1+c},\ldots,a_{s_m+c})$, where $m \in \mathbb{Z}_+$, $c \in \mathbb{Z}$ and $\boldsymbol{s}= (s_1, \ldots, s_m)\in \bb{Z}_+^m$. Also for convenience, set $\Gamma_0 = 0$.

By the symmetry of $f$ and $g$, one can write
\begin{align*}
S_{p,q}^{(r)}(f \otimes g) &=\bigvee_{\mathbf{k} \in \mathcal{D}_{r,<}} \bigvee_{\mathbf{s} \in \mathcal{D}_{p+q-2 r}}  h_r\left(T_{\mathbf{k}}, T_{\mathbf{s}}\right)\left[\Gamma_{\mathbf{k}}\right]^{-2 / \alpha}\left[\Gamma_{\mathbf{s}}\right]^{-1 / \alpha}, 
\end{align*}
for $r=1,\ldots,p\wedge q$. If $p + q -2 r = 0$, the conclusion follows from Theorem \ref{lambda_{k-1}}. For $p+q-2r \geq 1$, we identify the main term of $S_{p,q}^{(r)}(f \otimes g)$ as 
\begin{equation*}
    M_r: = \Gamma_1^{-2/\alpha}\Gamma_2^{-2/\alpha}\ldots\Gamma_r^{-2/\alpha} \bigvee_{\mathbf{s} \in \mathcal{D}_{p+q-2r}} h_r(T_{[1:r]}, T_{r+\mathbf{s}}) [\Gamma_{r+\mathbf{s}}]^{-1/\alpha},
\end{equation*}
and will show that 
\begin{equation}\label{main}
\lim _{\lambda \rightarrow \infty} \lambda^{\alpha / 2}(\ln \lambda)^{-(r-1)} \mathbb{P}\left( M_r >\lambda\right)  = k_{r, \alpha} C_r(f, g).
\end{equation}
Then we look at the  remainder term:
\begin{align*}
    R_r: = \bigvee_{\substack{{\mathbf{s} \in \mathcal{D}_{p+q-2 r}}\\ \mathbf{k} \in \mathcal{D}_{r,<}\backslash\{(1,2,\ldots, r)\}  }} [\Gamma_{\mathbf{k}}]^{-2 / \alpha}  h_r\left(T_{\mathbf{k}}, T_{\mathbf{s}}\right)[\Gamma_{\mathbf{s}}]^{-1/\alpha}.
\end{align*}
%\{ (\mathbf{k},\mathbf{s}) \mid \mathbf{s} \in \mathcal{D}_{p+q-2 r}, \mathbf{k} \in \mathcal{D}_r \}\setminus \{ (\mathbf{k},\mathbf{s}) \mid \mathbf{s} \in \mathcal{D}_{p+q-2 r}, \mathbf{k} = (1, \ldots,r) \}
In view of Lemma \ref{lem 1.1}, once we show that 
\begin{equation}\label{reminder}
\lim _{\lambda \rightarrow \infty} \lambda^{\alpha / 2}(\ln \lambda)^{-(r-1)} \mathbb{P}\left(R_r>\lambda\right)=0,
\end{equation}
the conclusion follows. 

We now outline the proof of the relation (\ref{main}). Let $U_i$, $1\leq i \leq r$,  be i.i.d.\ uniformly distributed random variables in $[0,1]$. Using similar arguments as those used to establish \cite[Eq.(3.7)]{rosinski1999product}, one can derive for $\lambda>0$ that
\begin{align}\label{Mr}
     \mathbb{P}\left(M_r>\lambda\right)=\int_0^{\infty} e^{-x} \frac{x^r}{r !} P_x(\lambda) d x,
 \end{align}
 where
$ 
  P_x(\lambda)=\mathbb{P}\left( Y_x \prod_{i=1}^r U_i^{-1} >\lambda^{\alpha / 2} x^r\right), 
$
 and
$$
\begin{gathered}
Y_x:= \left[\bigvee_{\mathbf{s} \in \mathcal{D}_{p+q-2 r}} h_r\left(T_{[1:r]}, T_{r+\mathbf{s}}\right)\left[x+\Gamma_{\mathbf{s} -1 }\right]^{-1 / \alpha} \right]^{\alpha/2},
\end{gathered}
$$
for arbitrary $x>0$ but fixed. 
 Splitting the supremum in $Y_x$ according to whether $s_1 \geq 2$, or $s_1=1$, we get
\begin{align}
    Y_x = Y_x^\prime \vee \pp{x^{-1/2} Y_x^{\prime \prime}}\le  Y_x^\prime +  x^{-1/2} Y_x^{\prime \prime},
\end{align}
where 
\begin{equation*}
Y_x^\prime : = \left[\bigvee_{\mathbf{s} \in \mathcal{D}_{p+q-2 r}} h_r\left(T_{[1: r]}, T_{r+\boldsymbol{s}}\right) \left[x+\Gamma_{\boldsymbol{s}}\right]^{-1 / \alpha} 
 \right]^{\alpha/2},
\end{equation*}
and
\begin{equation*}
Y_x^{\prime \prime}:=\left[\bigvee_{\boldsymbol{s} \in \mathcal{D}_{p+q-2 r-1}} h_r\left(T_{[1: r+1]}, T_{r+1+\boldsymbol{s}}\right) \left[x+\Gamma_{\boldsymbol{s}}\right]^{-1 / \alpha} \right]^{\alpha / 2} .
\end{equation*}
We claim that there exists $M>0$ such that
\begin{align}\label{leq M}
\mathbb{E}[Y_x^\prime] \leq M \quad \text{  and   }\quad \mathbb{E}[Y_x^{\prime\prime}] \leq M,  \qquad  \text{for all }x>0.
\end{align}
If this is the case, then $\mathbb{E}[Y_x] < \infty$,  and the relation
\begin{equation}\label{eq 33}
\lim _{\lambda \rightarrow \infty} \lambda^{\alpha / 2}(\ln \lambda)^{-(r-1)} P_x(\lambda)=\frac{\alpha^{r-1} x^{-r}}{2^{r-1}(r-1) !} \mathbb{E} Y_x 
\end{equation}
follows from \cite[Lemma 3.2]{rosinski1999product}. Moreover, the limit  (\ref{eq 33}) leads to the following:
\begin{equation*}
%\label{main par}
\begin{aligned}
& \lim _{\lambda \rightarrow \infty} \lambda^{\alpha / 2}(\ln \lambda)^{-(r-1)} \int_0^{\infty} e^{-x} \frac{x^r}{r !} P_x(\lambda) d x  
  =\int_0^{\infty} e^{-x} \frac{x^r}{r !} \left(\frac{\alpha}{2}\right)^{r-1} \frac{x^{-r}}{(r-1) !} \mathbb{E} Y_x d x\\
 =  &\frac{1}{r!(r-1)!} \left(\frac{\alpha}{2} \right)^{r-1} \mathbb{E}\left(\bigvee_{\mathbf{s} \in \mathcal{D}_{p+q-2 r}} h_r \left(T_{[1:r]}, T_{r+\mathbf{s}}\right)\left[\Gamma_{\mathbf{s}}\right]^{-1 / \alpha} \right)^{\alpha/2}\\
 = &\frac{1}{r!(r-1)!} \left(\frac{\alpha}{2} \right)^{r-1} \mathbb{E}\left| I^e_{p+q-2r}\left( h_r\left(T_{[1:r]}, \cdot\right)\right) \right|^{\alpha/2}.
\end{aligned}
\end{equation*}
Here, one can verify the interchange of integral and limit in the first step in a similar manner to P.15 of \cite[]{rosinski1999product} via the dominated convergence theorem with the help of \eqref{leq M}. Then \eqref{main} follows combining the above with \eqref{Mr}.
% To check $\mathbb{E}\left[ Y_x\right] < \infty$, we can apply  arguments similar to those from (3.19) to the line prior to (3.21) in \cite{rosinski1999product}, 
%\sum_{i=1}^k \mathbf{1}_{\{s_i>m\}}=j,\,\sum_{i=1}^k \mathbf{1}_{\{s_i>m\}}=k-j
%\boldsymbol{s}_{[1: j]}<m, s_{[j+1: k]} \geq m
% and $\boldsymbol{s}_{[k+1:k]}$ is interpreted as a zero-length (vanishing) vector
Now we  verify (\ref{leq M}). Introduce the following index set
\begin{align*}
& \mathcal{D}_{k}^{j,N} = \pc{  \boldsymbol{s} \in \mathcal{D}_{k} \mid \sum_{i=1}^k \mathbf{1}_{\{s_i< N\}}=j,\,\sum_{i=1}^k \mathbf{1}_{\{s_i\geq N\}}=k-j }, \,k\ge 1, \,  0 \leq j \leq k,\text{ and } N \in \mathbb{Z}_+.
\end{align*}
 We only present the proof of the first relation in \eqref{leq M} and the second follows similar arguments.  It suffices to show for every $J=0,1, \ldots, p+q-2 r$,
\begin{align}\label{3.21}
  \mathbb{E}  \left[\bigvee_{\mathbf{s} \in \mathcal{D}_{p+q-2 r-J}} h_r\left(T_{[1:r+J]}, T_{\mathbf{s}+r+J}\right) \left[\Gamma_{[1:J]}\right]^{-1/\alpha} \left[\Gamma_{\boldsymbol{s} +J}\right]^{-1/\alpha}\right]^{\alpha/2}  < \infty,
\end{align}
where $\Gamma_{[1:0]} = 1$. Note that  $Y_x^\prime$ in \eqref{leq M} corresponds to $J=0$ above, where   $x$ is eliminated by monotonicity. When $J=p+q-2 r$, relation (\ref{3.21}) holds trivially since the supremum above is understood as $h_r\left(T_{[1: p+q-r]}\right)\left[\Gamma_{[1: p+q-2r]}\right]^{-1 / \alpha}$. We establish (\ref{3.21}) by downward induction in $J$.  Assume that (\ref{3.21}) holds for all $J>J_0$, for some $J_0\in \{0,\ldots, p+q-2 r-1\}$. We now prove it for $J=J_0$.  Observe that for any $N \geq 1$ fixed,
\begin{align}
    \mathcal{D}_{p+q-2r-J_0} = \bigcup_{j = 0}^{p+q-2r-J_0} \mathcal{D}_{p+q-2r-J_0}^{j,N}.
\end{align}
We claim that  for $1 \leq j \leq p+q-2r-J_0$ and $J = J_0$, if  in (\ref{3.21}) the supremum is restricted to $\boldsymbol{s} \in D_{p+q-2r-J_0}^{j,N}$, the expectation is finite. Indeed, in this case,   the restriction to $D_{p+q-2r-J_0}^{j,N}$ consists of a finite number of terms for each of which we can apply the induction hypothesis  by exploring the monotonicity of $\Gamma_j^{-1/\alpha}$ in $j$.
Hence, it is enough to prove  (\ref{3.21}) when the supremum is restricted to $\boldsymbol{s} \in D_{p+q-2r-J_0}^{0,N}$  for some large enough $N\geq 1$. This can be verified through a similiar argument as on P.15-18 in \cite[]{rosinski1999product}, which now relies on  Proposition \ref{prop6.7} and the assumptions that $f \in \mathcal{L}_{p,+}^\alpha(\mu)$ and $g \in \mathcal{L}_{q,+}^\alpha(\mu)$. 

Next,  we establish (\ref{reminder})  using the {decomposition} $R_r=\bigvee_{r_0=0}^{r-1} R_{r, r_0}$,  where for  $r_0\in\{0,1,\ldots,r-1\}$,
\begin{align*}
 R_{r, r_0}:= &
 \left[\Gamma_{[1:r_0]}\right]^{-2 / \alpha} \bigvee_{r_0 +2\leq i_1 < \ldots < i_{r-r_0}}\left[\Gamma_{\boldsymbol{i}_{[1:r-r_0]}}\right]^{-2/\alpha}  \times  \\
& \bigvee_{
\substack{\boldsymbol{j} \in \mathcal{D}_{p+q-2 r},\\ j_v+r_0 \neq i_w \text{ for }\\
1\leq v \leq  p+q-2 r, 1\leq w\leq r-r_0}
} 
\left[\Gamma_{\boldsymbol{j}_{[1:p+q-2r]}+r_0} \right]^{-1/\alpha}     h_r\left(T_{[1:r_0]}, T_{\boldsymbol{i}_{[1:r-r_0]}}, T_{\boldsymbol{j}_{[1:p+q-2r]}+r_0}\right). 
\end{align*}
In view of Lemma \ref{lem 1.1}, the relation (\ref{reminder}) will follow once we check that for each $r_0\in\{0,1,\ldots,r-1\}$,
\begin{align*}
   \lim _{\lambda \rightarrow \infty} \lambda^{\alpha / 2}(\ln \lambda)^{-(r-1)} \mathbb{P}\left(R_{r, r_0}>\lambda\right)=0. 
\end{align*}
To prove the relation above, without loss of generality, we assume $\alpha\in (0,1)$; otherwise, it can be reduced to this case via a suitable power transformation applied to $R_{r,r_0}$. 

Next, an upper bound of $R_{r, r_0}$ is
%\sy{How about assume without loss of generality $\alpha\in (0,1)$, so that we can skip $l$ below?  The relation above can be extended to other range of $\alpha$ by power transform.}
\begin{equation}\label{eq:R_r r_0 2}
\begin{aligned}
 R_{r, r_0}^{(2)}= & \left[\Gamma_{[1:r_0]}\right]^{-2 / \alpha} \sum_{r_0 +2\leq i_1 < \ldots < i_{r-r_0}}\left[\Gamma_{\boldsymbol{i}_{[1:r-r_0]}}\right]^{-2/\alpha} \times  \\
&  \sum_{\substack{\boldsymbol{j} \in \mathcal{D}_{p+q-2 r}, j_v+r_0 \neq i_w \\
1\leq v\leq  p+q-2 r, 1\leq w\leq r-r_0}} \left[\Gamma_{\boldsymbol{j}_{[1:p+q-2r]}+r_0}  \right]^{-1/\alpha} h_r\left(T_{[1:r_0]}, T_{\boldsymbol{i}_{[1:r-r_0]}}, T_{\boldsymbol{j}_{[1:p+q-2r]}+r_0}\right). 
\end{aligned}
\end{equation}
Then $\mathbb{P}\left(R_{r, r_0}>\lambda\right) 
    \leq  \mathbb{P}\left(R_{r, r_0}^{(2)}>\lambda \right)$ for all $\lambda>0$. Hence, the relation (\ref{reminder}) follows once we establish that
\begin{align}\label{3.32}
 \lim _{\lambda \rightarrow \infty} \lambda^{\alpha / 2}(\ln \lambda)^{-(r-1)} \mathbb{P}\left(R_{r, r_0}^{(2)}>\lambda\right)=0,   
\end{align}
for every $r_0\in \{0, \ldots, r-1\}$. To show (\ref{3.32}), we start with a version of $R_{r, r_0}^{(2)}$ that truncates small $i$ and $j$ indices. For an $M \geq 1$ and $r_0\in \{0, \ldots, r-1\}$, define
\begin{align}
 R_{r, r_0}^{(2)}(M)=&\left[\Gamma_{[1:r_0]}\right]^{-2 / \alpha} \sum_{M\leq i_1 < \ldots < i_{r-r_0}}[\Gamma_{\boldsymbol{i}_{[1:r-r_0] }}]^{-2/\alpha} \sum_{\substack{\boldsymbol{j} \in \mathcal{D}_{p+q-2 r},\,j_v \geq M,\, j_v+r_0 \neq i_w \notag \\
1\leq v\leq  p+q-2 r, 1\leq w\leq r-r_0}}   \\
& \left[\Gamma_{\boldsymbol{j}_{[1:p+q-2r]}+r_0}  \right]^{-1/\alpha}  h_r\left(T_{[1:r_0]}, T_{\boldsymbol{i}_{[1:r-r_0]}}, T_{\boldsymbol{j}_{[1:p+q-2r]}+r_0}\right)\notag\\
 \leq  & \left[\Gamma_{[1:r_0]}\right]^{-2 / \alpha} \sum_{M\leq i_1 < \ldots < i_{r-r_0}}[\Gamma_{\boldsymbol{i}_{[1:r-r_0]-r_0}}^{(1)}]^{-2/\alpha} \sum_{\substack{\boldsymbol{j} \in \mathcal{D}_{p+q-2 r},\,j_v \geq M,\, j_v+r_0 \neq i_w \\
1\leq v\leq  p+q-2 r, 1\leq w\leq r-r_0}}  \label{M2} \\
& \left[\Gamma^{(1)}_{\boldsymbol{j}_{[1:p+q-2r]}}  \right]^{-1/\alpha}  h_r\left(T_{[1:r_0]}, T_{\boldsymbol{i}_{[1:r-r_0]}}, T_{\boldsymbol{j}_{[1:p+q-2r]}+r_0}\right)\notag \\
=&:   \left[\Gamma_{[1:r_0]}\right]^{-2 / \alpha} \times R_{r, r_0}^{(3)}(M), \notag
\end{align}
 where $\Gamma^{(1)}_i = \Gamma_{i+r_0} - \Gamma_{r_0}$ for any $i \in \mathbb{Z}_+$ is independent of $\Gamma_{j}$, $1\leq j \leq r_0$.
 We claim that for all $M$ large enough,
\begin{align}\label{3.33}
    \lim _{\lambda \rightarrow \infty} \lambda^{\alpha / 2}(\ln \lambda)^{-(r-1)} \mathbb{P}\left(R_{r, r_0}^{(2)}(M)>\lambda\right)=0,
\end{align}
 for every $r_0\in \{0, \ldots, r-1\}$. Indeed, by \cite[Lemma 3.3]{rosinski1999product}, it suffices to show for all $M$ sufficiently large, 
\begin{align}\label{3.36}
 \mathbb{E}\left|R_{r, r_0}^{(3)}(M)\right|^{\alpha  / 2}<\infty,   
\end{align}
 for every $r_0\in \{0, \ldots, r-1\}$.
 To verify \eqref{3.36}, 
 recall $$h_r\left(x_1, \ldots, x_{p+q-r}\right)= f \left(x_1, \ldots, x_r,  \ldots, x_p\right) g\left(x_1, \ldots, x_r, x_{p+1}, \ldots, x_{p+q-r}\right), \ x_1,\ldots, x_{p+q-r}\in E.
 $$
By first properly enlarging the summation domain of the $i$, $j$ indices in \eqref{M2}, and then 
applying the Cauchy–Schwarz inequality, we have
\begin{align*}
&  \mathbb{E}\left|R_{r, r_0}^{(3)}(M)\right|^{\alpha / 2} \leq\mathbb{E}\Bigg[  \Bigg(  \sum_{\substack{\boldsymbol{i} \in \mathcal{D}_{r-r_0} \\
i_w \geq M - r_0, 1\leq w\leq r-r_0}}\left[\Gamma^{(1)}_{\boldsymbol{i}_{[1:r-r_0]}}\right]^{-1/\alpha}\sum_{\substack{\boldsymbol{j} \in \mathcal{D}_{p-r} \\ j_v \geq M,\, j_v \neq i_w\\1\leq  v\leq  p-r, 1\leq w \leq r-r_0}}  \left[\Gamma^{(1)}_{\boldsymbol{j}_{[1:p-r]}}  \right]^{-1/\alpha} f\big( T_{[1:r_0]}, \\
& \quad  T_{\boldsymbol{i}_{[1:r-r_0]}}, T_{\boldsymbol{j}_{[1:p-r]}+r_0}\big)  \Bigg)\times \Bigg(\sum_{\substack{ \boldsymbol{i}\in \mathcal{D}_{r-r_0} \\
i_w \geq M - r_0, 1\leq w\leq  r-r_0  }}  \left[\Gamma^{(1)}_{\boldsymbol{i}_{[1:r-r_0]}}\right]^{-1/\alpha} \sum_{\substack{\boldsymbol{j} \in \mathcal{D}_{q-r}, \\ j_v \geq M,j_v \neq i_w  \\1\leq v\leq  q-r, 1 \leq w \leq r-r_0}}\left[\Gamma^{(1)}_{\boldsymbol{j}_{[1:p-r]}}  \right]^{-1/\alpha}  \\
&  g\big(T_{[1:r_0]},   T_{\boldsymbol{i}_{[1:r-r_0]}}, T_{\boldsymbol{j}_{[1:q-r]}+r_0}\big) \Bigg) \Bigg]^{\alpha/2}\\
& \leq \Bigg\{ \mathbb{E}\Bigg[\sum_{\substack{\boldsymbol{i} \in \mathcal{D}_{r-r_0} \\
i_w \geq M - r_0, 1\leq w\leq r-r_0}}\left[\Gamma^{(1)}_{\boldsymbol{i}_{[1:r-r_0]}}\right]^{-1/\alpha}\sum_{\substack{\boldsymbol{j} \in \mathcal{D}_{p-r} \\ j_v \geq M,\, j_v \neq i_w\\1\leq  v\leq  p-r, 1\leq w \leq r-r_0}}  \left[\Gamma^{(1)}_{\boldsymbol{j}_{[1:p-r]}}  \right]^{-1/\alpha} f\big(T_{[1:r_0]},  T_{\boldsymbol{i}_{[1:r-r_0]}}, \\
& T_{\boldsymbol{j}_{[1:p-r]}+r_0} \big) \Bigg]^{\alpha} \Bigg\}^{1/2}\Bigg\{ \mathbb{E}\Bigg[\sum_{\substack{ \boldsymbol{i}\in \mathcal{D}_{r-r_0} \\
i_w \geq M - r_0, 1\leq w\leq  r-r_0  }}  \left[\Gamma^{(1)}_{\boldsymbol{i}_{[1:r-r_0]}}\right]^{-1/\alpha} \sum_{\substack{\boldsymbol{j} \in \mathcal{D}_{q-r}, \\ j_v \geq M,j_v \neq i_w  \\1\leq v\leq  q-r, 1 \leq w \leq r-r_0}}\left[\Gamma^{(1)}_{\boldsymbol{j}_{[1:p-r]}}  \right]^{-1/\alpha} \\
& g\big(T_{[1:r_0]}, T_{\boldsymbol{i}_{[1:r-r_0]}}, T_{\boldsymbol{j}_{[1:q-r]}+r_0}\big) \Bigg]^{\alpha} \Bigg\}^{1/2},
\end{align*}
where both expectations in the last expression are finite for $M$ large enough due to Corollary \ref{at6.3} above and the assumptions that $f \in \mathcal{L}_{p,+}^\alpha(\mu), g \in \mathcal{L}_{q,+}^\alpha(\mu)$. 

%as shown in the derivation of expressions \cite[(3.25)-(3.27)]{rosinski1999product}, replacing the supremum operator with the summation operator and restricting $\alpha \in  (0,1)$ in (\ref{T}) preserves the validity of Proposition  \ref{prop6.7}. Moreover, with an inspection of the proof of Proposition  \ref{prop6.7}

The next step is consider $R_{r, r_0}^{(4)}(M)$, a version of $R_{r, r_0}^{(2)}$ that truncates only small $j$ indices  (thus ``less truncated'' compared to $R_{r, r_0}^{(2)}(M)$). Specifically,   
 $R_{r, r_0}^{(4)}(M)$ is defined by  modifying the first summation in $R_{r, r_0}^{(2)}(M)$, replacing the range $M\leq i_1 < \ldots < i_{r-r_0}$ with $r_0 +2\leq i_1 < \ldots < i_{r-r_0}$ for $M\geq 1$.
 We claim that for all $M$ large enough, 
\begin{align}\label{3.38}
\lim _{\lambda \rightarrow \infty} \lambda^{\alpha / 2}(\ln \lambda)^{-(r-1)} \mathbb{P}\left(R_{r, r_0}^{(4)}(M)>\lambda\right)=0,
\end{align}
for every $r_0\in \{0, \ldots, r-1\}$. Indeed,
% where $R_{r, r_0}^{(4)}(M)$ is defined by modifying the first summation in (\ref{M2}), replacing the range $r_0 +2\leq i_1 < \ldots < i_{r-r_0}$ with $M\leq i_1 < \ldots < i_{r-r_0}$ for $M\geq 1$. 
observe that for all $M \geq r_0+2$, 
\begin{equation}\label{decomp}
R_{r, r_0}^{(4)}(M)=R_{r, r_0}^{(2)}(M)+\sum_{n=r_0+2}^{M-1} R_{r, r_0}(M, n),  
\end{equation}
where for $n\in \{r_0+2,\ldots, M\}$, 
\begin{align*}
R_{r, r_0}(M, n)= & \left[\Gamma_{[1:r_0]}\right]^{-2/\alpha}  \Gamma_n^{-2 / \alpha} \sum_{n+1\leq i_2 < \ldots < i_{r-r_0}} [\Gamma_{\boldsymbol{i}_{[2:r-r_0]}}]^{-2 / \alpha} \times \sum_{\substack{\boldsymbol{j} \in \mathcal{D}_{p+q-2 r}, j_v \geq M, j_v+r_0 \neq i_w \\
1\leq v\leq  p+q-2 r, 1\leq w\leq r-r_0}}\\
& \left[\Gamma_{\boldsymbol{j}_{[1: p+q-2 r] +r_0}}\right]^{-1 /\alpha }  h_r\left(T_{[1:r_0]}, T_n, T_{\boldsymbol{i}[2:r-r_0]},T_{\boldsymbol{j}[1:p+q-2r]+r_0} \right).
\end{align*}
In light of (\ref{decomp}), since (\ref{3.33}) holds for all $M$ large enough,  the relation (\ref{3.38}) will follow once we establish that for all fixed $M$  large enough,
\begin{align}\label{3.39}
\lim _{\lambda \rightarrow \infty} \lambda^{\alpha / 2}(\ln \lambda)^{-(r-1)} \mathbb{P}\left(R_{r, r_0}(M, n)>\lambda\right)=0,
\end{align}
for every $n\in \{ r_0+2, \ldots, M-1\}$.
%Applying the same downward induction as the one used in the proof of (\ref{3.21}), we see that (\ref{3.39}) will follow if we prove that for all $M$ large enough,
Applying arguments similar to  those in the derivation of  \eqref{3.36}, which involves \cite[Lemma 3.4]{rosinski1999product}, applying the Cauchy-Schwartz inequality and Proposition \ref{prop6.7},
one can show  for all $M$ large enough  that
\begin{align}\label{3.40}
\lim _{\lambda \rightarrow \infty} \lambda^{\alpha / 2}(\ln \lambda)^{-(r-1)} \mathbb{P}\left(R_{r, r_0}^{(5)}(M)>\lambda\right)=0, \quad r_0 \in \{ 0,\ldots, r-1\},
\end{align}
where
\begin{align*}
R_{r, r_0}^{(5)}(M)=& [\Gamma_{[1:r_0]}]^{-2/\alpha}   [\Gamma_{[r_0+2:r+1]}]^{-2/\alpha} \times
\\ & \sum_{\substack{\boldsymbol{j} \in \mathcal{D}_{p+q-2 r} \\
j_v \geq M, 1 \leq v \leq p+q-2 r}}     \left[\Gamma_{\boldsymbol{j}_{[1: p+q-2 r]+r_0}}\right]^{-1 /\alpha }  h_r\left(T_{[1:r_0]}, T_{[r_0+2:r+1]}, T_{\boldsymbol{j}[1:p+q-2r]+r_0}\right).
\end{align*}
Then combining (\ref{3.40}) with (\ref{3.33}) yields (\ref{3.39}) via a downward induction similar to the one used to derive \eqref{3.21} above. We have thus concluded \eqref{3.38}.

Fix an $M$ for which (\ref{3.38}) holds. Observe that   (\ref{3.32}) follows  from combining \eqref{3.38} and  the following relation: for every   $r_0\in\{0,1,\ldots,r-1\}$ and every $n\in\{ 1 , \ldots, M-1\}$,
\begin{align}\label{3.41}
    \lim _{\lambda \rightarrow \infty} \lambda^{\alpha / 2}(\ln \lambda)^{-(r-1)} \mathbb{P}\left(R_{r, r_0}^{(6)}(n)>\lambda\right)=0,
\end{align}
where  
\begin{align*}
 R_{r, r_0}^{(6)}(n)=&\left[\Gamma_{[1:r_0]}\right]^{-2/\alpha} \left(\Gamma_n\right)^{-1 / \alpha} \sum_{\substack{r_0+2 \leq i_1<\ldots<i_{r-r_0} \\
i_w \neq n,\,1\leq w\leq r-r_0}} \left[\Gamma_{\boldsymbol{i}_{\left[1: r-r_0\right]}}\right]^{-2 /\alpha}  \sum_{\substack{\boldsymbol{j} \in \mathcal{D}_{p+q-2 r-1}, j_v>n, j_v+r_0 \neq i_w \\
1\leq v\leq p+q-2 r-1, \,1\leq w\leq r-r_0}} \\
& \left[\Gamma_{\boldsymbol{j}_{[1: p+q-2 r-1]+r_0}}\right]^{-1 /(\alpha / l)} h_r\left(T_{[1:r_0]}, T_{\boldsymbol{i}[1:r-r_0]},  T_{n+r_0}, T_{\boldsymbol{j}[1:p+q-2r-1]+r_0}\right).
\end{align*}
Now we prove \eqref{3.41}. With the help of (\ref{3.38}), we repeat the truncation and downward induction argument used earlier,  so that   (\ref{3.41}) follows from  
\begin{align}\label{3.42}
\lim _{\lambda \rightarrow \infty} \lambda^{\alpha / 2}(\ln \lambda)^{-(r-1)} \mathbb{P}\left(R_{r, r_0}^{(7)}>\lambda\right)=0,
\end{align}
where  
$$
\begin{aligned}
& R_{r, r_0}^{(7)}= \left[\Gamma_{[1:r_0]}\right]^{-2/\alpha}  \left[\Gamma_{[r_0+2:r+1]}\right]^{-2/\alpha}  \left[\Gamma_{[1+r_0+r: p+q-r+r_0]}\right]^{-1 /\alpha} f\left(T_{[1:p]}\right)g\left(T_{[1:r]}, T_{[p+1:p+q-r]}\right).
\end{aligned}
$$
Following similar arguments as in the derivation of \cite[Lemma 3.4]{rosinski1999product}, establishing (\ref{3.42}) reduces to proving that  $ \mathbb{E} \left|  f\left(T_{[1:p]}\right)
g\left(T_{[1:r]}, T_{[p+1:p+q-r]}\right)\right|^{\alpha / 2}<\infty$. This follows from $f \in \mathcal{L}_{p,+}^\alpha(\mu)$ and $g \in \mathcal{L}_{q,+}^\alpha(\mu)$ and Hölder's inequality. This proves (\ref{3.41}), and consequently (\ref{3.32}).  The proof is now complete.

\end{document}